\documentclass[12pt]{article}
\usepackage[latin1]{inputenc}
\usepackage[english]{babel}
\newtheorem{theorem}{Theorem}
\newtheorem{lemma}{Lemma}
\newtheorem{definition}{Definition}
\newtheorem{proof}{Proof}
\usepackage{stmaryrd}
\usepackage{enumerate}
\usepackage[noend]{algpseudocode}
\usepackage{calrsfs}
\usepackage{graphicx}
\usepackage{amssymb}
\usepackage{amsmath}
\usepackage{epstopdf}
\usepackage{epsfig}
\usepackage{subcaption}
\usepackage{pstricks}
\usepackage{fancyheadings}
\usepackage{pgfplots}
\usepackage{physics}
\usepackage{stackengine}
\usepackage[margin=1.0in]{geometry}
\usepackage{footnote}

\usepackage[framemethod=tikz]{mdframed}
\usepackage{lipsum}
\usepackage{url}
\usepackage{xcolor}
\usepackage{framed,multirow}
\usepackage{latexsym}

\usepackage{lipsum}
\usepackage{amsfonts}
\usepackage{graphicx}
\usepackage{epstopdf}
\usepackage{mathtools}
\usepackage{amsmath}
\usepackage{bm}
\usepackage{float}
\usepackage{tikz}
\usepackage{stackengine}

\def\R{\mathbb{R}}

\def\Omegat{\widetilde{\Omega}}
\def\Pt{\widetilde{\bm{P} }}

\def\l{\left\langle}
\def\r{\right\rangle}
\def\1{\chi}

\newcommand{\pf}[2]{\frac{\partial #1 }{\partial #2}}

\DeclareMathAlphabet{\pazocal}{OMS}{zplm}{m}{n}

\newtheorem{remark}{Remark}

\newcommand{\lJump}{[\![}
\newcommand{\rJump}{]\!]}

\newcommand{\en}[1]{{\left\vert\kern-0.25ex\left\vert\kern-0.25ex\left\vert #1 
    \right\vert\kern-0.25ex\r6ight\vert\kern-0.25ex\right\vert}}

\makeatletter
\def\BState{\State\hskip-\ALG@thistlm}
\makeatother

\begin{document}
\title{Accurate simulations of nonlinear dynamic shear ruptures on pre-existing faults in 3D elastic solids  with dual-pairing SBP methods}
\author{Kenneth Duru\thanks{Mathematical Sciences Institute, The Australian National University, Australia.}, \and Christopher  Williams \thanks{Mathematical Sciences Institute, The Australian National University, Australia, Department of Statistics, University of Oxford}, \and  Frederick  Fung \thanks{Mathematical Sciences Institute, The Australian National University, National Computational Infrastructure, Australia} }
\pagenumbering{arabic}
\maketitle
\begin{abstract}
In this paper we derive and analyse efficient and stable numerical methods for accurate numerical simulations of nonlinear dynamic shear ruptures on non-planar faults embedded in 3D elastic solids using  dual-paring (DP) summation by parts (SBP) finite difference (FD) methods. Specifically, for nonlinear dynamic earthquake ruptures,  we demonstrate that the DP SBP FD operators [K. Mattsson. J. Comput. Phys., 335:283-310, 2017] generate spurious catastrophic high frequency wave modes that do not diminish with mesh refinement. Meanwhile our new dispersion relation preserving (DRP)  SBP FD operators [C. Williams and K Duru, https://arxiv.org/abs/2110.04957, 2021] have more accurate numerical dispersion relation properties and do not support  poisonous spurious high frequency wave modes. Numerical simulations are performed in 3D with geometrically complex fault surfaces verifying the efficacy of the method. Our method accurately reproduces community developed dynamic rupture benchmark problems, proposed by Southern California Earthquake Center, with less computational effort than standard methods based on traditional SBP  FD operators.
\end{abstract}

\section{Introduction}
Numerical simulations of nonlinear frictional sliding along interfaces embedded in elastic solids are of utmost importance in science and engineering \cite{kostrov1988principles,bowden2001friction,Scholz1998}. Frictional processes in both industrial materials such as steel pipes, plates and beams, and natural structures like faults in the Earth's subsurface can lead to failure or disaster. For instance, current understanding of crustal earthquakes is that they nucleate as frictional instabilities on pre-existing fault surfaces under slow tectonic loading \cite{Oden2010, Scholz1998, Madriaga1998,Svetlizky2014,dieterich1979modeling, rice1983stability, ruina1983slip}. Earthquakes occur when two sides of the fault, initially held in contact by a high level of static frictional resistance, slip suddenly when the high level frictional strength catastrophically decreases. Frictional fault-slip generates strong ground shaking which can be carried by seismic (elastic) waves to far field, far away from fault zones. Thus, spontaneously propagating shear ruptures along internal interfaces in elastic solids are a useful idealisation of earthquake source processes \cite{Scholz1998,shi2013rupture,dunham2011earthquake,andrews1985dynamic,ramos2019transition, dieterich1979modeling, rice1983stability, ruina1983slip, falk2001rice,Freund19792199,SCEC2009,Kozdon2013,SCEC2009,GEUBELLE19951791,SCEC2018,DuruandDunham2016}.  However, this problem is both numerically and computationally challenging, because of the unprecedented large gradients generated by nonlinear coupling of fields across the fault interface and the presence of discontinuous solutions on the fault. 


Efficient and accurate numerical methods are necessary for effective and reliable numerical simulations of  nonlinear dynamic shear ruptures in elastic solids. Different numerical methods have been developed  \cite{Kozdon2013,SCEC2009,GEUBELLE19951791,SCEC2018,DuruandDunham2016, dieterich1979modeling,DuruGabrielIgel2021}. These include standard numerical methods such as spectral method \cite{Kaneko2008}, finite element method \cite{Oglesby1998,Oglesby2000}, finite volume method \cite{Benjemaa2007}, finite difference (FD) method \cite{DuruandDunham2016,Kozdon2013,shi2013rupture}, and the discontinuous Galerkin finite element method \cite{Puente2009,DuruGabrielIgel2021,Pelties2012}. They all have different strengths and weaknesses.  However, FD methods on structured grids are attractive because they are  efficient.

In this study, we derive and analyse high order accurate and provably stable FD methods for nonlinear frictional motion at interfaces in elastic solids. We consider the elastic wave equation in first-order form, subject to nonlinear friction laws at interfaces. The corresponding initial boundary value problem (IBVP) supports  energy bounded solutions. For energy-stable IBVPs, a standard approach to achieve numerical stability is to derive a discrete energy estimate, by mimicking the procedure in the continuous setting. If such an estimate is possible, numerical stability is ensured.
The summation-by-parts (SBP) FD \cite{BStrand1994} with the simultaneous approximation term (SAT) \cite{CarpenterGottliebAbarbanel1994,Mattsson2003,NordstromCarpenter2001} technique for implementing boundary conditions enables the development of energy-stable numerical approximations. 

The SBP-SAT FD method, with traditional SBP operators, have been developed for linear elastic waves and nonlinear friction laws \cite{DuruandDunham2016,Kozdon2013}.
However, traditional SBP FD operators based on central finite difference stencils often suffer from spurious unresolved wave-modes in their numerical solutions. These spurious wave-modes have the potential to destroy the accuracy of numerical solutions. To improve the accuracy of SBP operators, the dual-pairing SBP (DP SBP)  non-central  (upwind)  FD  stencils were recently introduce in  \cite{Mattsson2017,DovgilovichSofronov2015}. In \cite{DURU2022110966} these operators were shown to improve the accuracy of the numerical solutions of the 3D elastic wave equations in complex geometries.  The main benefit for the DP SBP operators \cite{Mattsson2017,DURU2022110966} is that they have the potential to suppress poisonous spurious oscillations from unresolved wave-modes, which can destroy the accuracy of numerical simulations on marginally resolved meshes. However, for dynamic shear ruptures composed of discontinuous and/or (nearly) singular exact solutions, numerical errors arising from high order accurate methods may not diminish with increasing mesh resolution, because of "Gibbs/Runge's phenomena". For these situations carefully designed numerical methods are necessary in order to compute accurate numerical solutions. Consequently, detailed theoretical analysis is needed in order to  understand the behaviours of numerical errors and control them.

In this paper, our first objective is to derive a provably stable numerical approximation of the elastic wave equation subject to nonlinear friction laws using the DP SBP framework. As in \cite{DuruandDunham2016,Kozdon2013}, first we construct a subspace of solutions satisfying the (discrete) linear elastodynamics equations and the interface condition exactly. We then project the solutions on the interface to this subspace, extracting data. A main result in this paper is the proof (see Theorem \ref{thm:uniq_and_existence}) of existence and uniqueness of the interface data for nonlinear friction laws in general anisotropic elastic media. We discretise the 3D elastic wave equation with the DP SBP operators, on boundary-conforming curvilinear meshes, and enforce the nonlinear interface conditions weakly using penalties.
 Using the energy method and the dual-pairing SBP framework \cite{Mattsson2017,DURU2022110966,CWilliams2021,williams2021provably}, we prove that the numerical method is energy-stable. The proof of energy-stability, however, bounds the growth of the numerical solutions in time but it does not guarantee the convergence of the solution, in particular in the presence of nearly singular and discontinuous exact solutions.

 Our second objective is to perform error analysis. Numerical error analysis for the discrete approximation of the 3D elastic wave equation subject to the nonlinear friction laws  at an internal fault/interface, is a nontrivial task. 
 To succeed we will make some simplifying assumptions.
 However, our analysis are plausible and are able to  explain some of the behavior of the error seen in the numerical simulations performed later in this paper.
 First, we linearise the friction law using a frozen coefficient argument.
 We derive a priori error estimates and prove the convergence of the numerical error for a linearised friction law.
 Second, we analyse the dependence of the parity (even/odd) of the operators on the numerical errors. For odd order accurate operators, our error analysis reveals that the DP SBP FD operators \cite{Mattsson2017} have the potential to generate spurious catastrophic high frequency errors that do not diminish with mesh refinement. Numerical simulations in 3D corroborates the theoretical analysis.

Our third objective is to eliminate the high frequency catastrophic error introduced by the DP SBP FD operators \cite{Mattsson2017}. Using the DP framework, it is possible to design optimised finite difference stencils such that numerical dispersion errors are significantly diminished for all wave numbers. This was the approach taken by the authors in \cite{williams2021provably} where $\alpha$-DRP SBP operators were derived. These DRP SBP operators are designed to guarantee maximum relative error $\le \alpha$ for all wave numbers. In particular  the $\alpha$-DRP SBP operators can resolve the highest frequency ($\pi$-mode) present on any equidistant grid at a tolerance of $\alpha = 5\%$ maximum error. For explicit schemes, these operators may provide a better resolution than any volume discretisation available today, including spectral methods, and significantly improves on the current standard for traditional operators that have a tolerance of $100 \%$ maximum error.  The DRP  SBP operators \cite{williams2021provably,CWilliams2021}  potentially eliminates the catastrophic high frequency errors. Numerical simulations  performed in 3D corroborates  the analysis.
Our DRP  SBP operators reproduces dynamic rupture benchmark problems, proposed by Southern California Earthquake Center, with less computational effort than standard methods based on traditional SBP  operators.

Efficient and scalable parallel implementation is necessary  for effective large scale numerical simulations. The DRP operators are implemented in the large scale code WaveQLab3D \cite{DURU2022110966,DuruandDunham2016} with near perfect scaling.  The resources used is slightly larger than the $6$th order upwind DP and traditional  SBP schemes, but not by a large amount.  We consider the DRP operators to have a marginal higher cost than the traditional and DP SBP methods. 

The remaining parts  of the paper will proceed as follows. 
In the next section, we present model equations for linear elastic solids in complex geometries and the nonlinear friction laws coupling the elastic solids at the interface. We also derive energy estimates which a stable numerical method should emulate. In Section 3 we introduce the frictional gluing (hat-) variables and prove the existence and uniqueness of the hat-variables for nonlinear friction laws in  general anisotropic elastic solids. In Section 4, we approximate the elastic wave equation in space using the DP SBP operators, numerically enforce the friction laws using penalties and prove numerical stability. In Section 5 we perform numerical error analysis. Numerical simulations are presented in Section 6, corroborating the theoretical analysis. In Section 7 we summarise the results.

\section{Physical model}
    \begin{figure}[h!]
    \centering
    \begin{tikzpicture}

\node at (6,-2.5) {$\uparrow{ \Phi^{-1}   } $};
\node at (8,-2.5) {$\downarrow{\ \Phi \ }$};

\begin{scope}
\node at (1,4.35) {$\Omega^{-}$};
\begin{scope}[every node/.append style={yslant=-0.5},yslant=-0.5]

    \coordinate (A1) at (0,0);
    \coordinate (A2) at (1,0);
    \coordinate (A3) at (2,0);
    \coordinate (A4) at (3,0);
    
    \coordinate (B1) at (0,1);
    \coordinate (B2) at (1,1);
    \coordinate (B3) at (2,1);
    \coordinate (B4) at (3,1);
    
    \coordinate (C1) at (0,2);
    \coordinate (C2) at (1,2);
    \coordinate (C3) at (2,2);
    \coordinate (C4) at (3,2);
    
    \coordinate (D1) at (0,3);
    \coordinate (D2) at (1,3);
    \coordinate (D3) at (2,3);
    \coordinate (D4) at (3,3);

  \node at (0.5,2.5) {};
  \node at (1.5,2.5) {};
  \node at (2.5,2.5) {};
  \node at (0.5,1.5) {};
  \node at (1.5,1.5) {$F_{r,0}^{-}$};
  \node at (2.5,1.5) {};
  \node at (0.5,0.5) {};
  \node at (1.5,0.5) {};
  \node at (2.5,0.5) {};
  \draw[color=black] (A1) to [bend left=40] (A2) to [bend right=50] (A3) to [bend left=20] (A4) ; 
  \draw[dashed, gray] (B1) to [bend left=40] (B2) to [bend right=50] (B3) to [bend left=20] (B4) ; 
  \draw[dashed, gray] (C1) to [bend left=40] (C2) to [bend right=50] (C3) to [bend left=20] (C4) ; 
  \draw[color=black] (D1) to [bend left=40] (D2) to [bend right=50] (D3) to [bend left=20] (D4) ; 
  
  \draw[color=black] (A1) to [bend left=40] (B1) to [bend right=50] (C1) to [bend left=20] (D1) ; 
  \draw[dashed, gray] (A2) to [bend left=40] (B2) to [bend right=50] (C2) to [bend left=20] (D2) ; 
  \draw[dashed, gray] (A3) to [bend left=40] (B3) to [bend right=50] (C3) to [bend left=20] (D3) ; 
\end{scope}
\begin{scope}[every node/.append style={yslant=0.5},yslant=0.5, shift = {(3,-3)}]
    \coordinate (A-1) at (0,0) + (-3,3);
    \coordinate (A-2) at (1,0);
    \coordinate (A-3) at (2,0);
    \coordinate (A-4) at (3,0);
    
    \coordinate (A1) at (0,0) + (-3,3);
    \coordinate (A2) at (1,0);
    \coordinate (A3) at (2,0);
    \coordinate (A4) at (3,0);
    
    \coordinate (B1) at (0,1);
    \coordinate (B2) at (1,1);
    \coordinate (B3) at (2,1);
    \coordinate (B4) at (3,1);
    
    \coordinate (C1) at (0,2);
    \coordinate (C2) at (1,2);
    \coordinate (C3) at (2,2);
    \coordinate (C4) at (3,2);
    
    \coordinate (D-1) at (0,3);
    \coordinate (D-2) at (1,3);
    \coordinate (D-3) at (2,3);
    \coordinate (D-4) at (3,3);
    
    \coordinate (D1) at (0,3);
    \coordinate (D2) at (1,3);
    \coordinate (D3) at (2,3);
    \coordinate (D4) at (3,3);

    \draw[fill=gray!30, gray!30] (A1) to [bend left=40] (A2) to [bend right=50] (A3) to [bend left=20] (A4) to [bend left=40] (B4) to [bend right=50] (C4) to [bend left=20] (D4);
  \draw[fill=gray!30, gray!30] (A1) to [bend left=40] (B1) to [bend right=50] (C1) to [bend left=20] (D1) to (D1) to [bend left=40] (D2) to [bend right=50] (D3) to [bend left=20] (D4) ;

  \node at (0.5,2.5) {};
  \node at (1.5,2.5) {};
  \node at (2.5,2.5) {};
  \node at (0.5,1.5) {};
  \node at (1.5,1.5) {${\Gamma}_F$};
  \node at (2.5,1.5) {};
  \node at (0.5,0.5) {};
  \node at (1.5,0.5) {};
  \node at (2.5,0.5) {};
  
  \draw[color=black] (A1) to [bend left=40] (A2) to [bend right=50] (A3) to [bend left=20] (A4) ; 
  \draw[dashed, gray] (B1) to [bend left=40] (B2) to [bend right=50] (B3) to [bend left=20] (B4) ; 
  \draw[dashed, gray] (C1) to [bend left=40] (C2) to [bend right=50] (C3) to [bend left=20] (C4) ; 
  \draw[color=black] (D1) to [bend left=40] (D2) to [bend right=50] (D3) to [bend left=20] (D4) ; 
  
  \draw[color=black] (A1) to [bend left=40] (B1) to [bend right=50] (C1) to [bend left=20] (D1) ; 
  \draw[dashed, gray] (A2) to [bend left=40] (B2) to [bend right=50] (C2) to [bend left=20] (D2) ; 
  \draw[dashed, gray] (A3) to [bend left=40] (B3) to [bend right=50] (C3) to [bend left=20] (D3) ; 
  \draw[color=black] (A4) to [bend left=40] (B4) to [bend right=50] (C4) to [bend left=20] (D4) ;

\end{scope}
\begin{scope}[every node/.append style={
    yslant=0.5,xslant=-1},yslant=0.5,xslant=-1,
    shift = {(3,0)}
  ]
    \coordinate (A1) at (0,0);
    \coordinate (A2) at (1,0);
    \coordinate (A3) at (2,0);
    \coordinate (A4) at (3,0);
    
    \coordinate (B1) at (0,1);
    \coordinate (B2) at (1,1);
    \coordinate (B3) at (2,1);
    \coordinate (B4) at (3,1);
    
    \coordinate (C1) at (0,2);
    \coordinate (C2) at (1,2);
    \coordinate (C3) at (2,2);
    \coordinate (C4) at (3,2);
    
    \coordinate (D1) at (0,3);
    \coordinate (D2) at (1,3);
    \coordinate (D3) at (2,3);
    \coordinate (D4) at (3,3);

  \node at (0.5,2.5) {};
  \node at (1.5,2.5) {};
  \node at (2.5,2.5) {};
  \node at (0.5,1.5) {};
  \node at (1.5,1.5) {$F_{s,0}^{-}$};
  \node at (2.5,1.5) {};
  \node at (0.5,0.5) {};
  \node at (1.5,0.5) {};
  \node at (2.5,0.5) {};
  
  \draw[dashed, gray] (B1) to [bend left=40] (B2) to [bend right=50] (B3) to [bend left=20] (B4) ; 
  \draw[dashed, gray] (C1) to [bend left=40] (C2) to [bend right=50] (C3) to [bend left=20] (C4) ; 
  \draw[color=black] (D1) to [bend left=40] (D2) to [bend right=50] (D3) to [bend left=20] (D4) ; 
  
  \draw[dashed, gray] (A2) to [bend left=40] (B2) to [bend right=50] (C2) to [bend left=20] (D2) ; 
  \draw[dashed, gray] (A3) to [bend left=40] (B3) to [bend right=50] (C3) to [bend left=20] (D3) ; 
  \draw[color=black] (A4) to [bend left=40] (B4) to [bend right=50] (C4) to [bend left=20] (D4) ; 
\end{scope}
\end{scope}

\begin{scope}[shift = {(8,0)}]
\node at (5,4.35) {$\Omega^{+}$};
\begin{scope}[every node/.append style={yslant=-0.5},yslant=-0.5]

    \coordinate (A+1) at (0,0);
    \coordinate (A+2) at (1,0);
    \coordinate (A+3) at (2,0);
    \coordinate (A+4) at (3,0);
    
    \coordinate (A1) at (0,0);
    \coordinate (A2) at (1,0);
    \coordinate (A3) at (2,0);
    \coordinate (A4) at (3,0);
    
    \coordinate (B1) at (0,1);
    \coordinate (B2) at (1,1);
    \coordinate (B3) at (2,1);
    \coordinate (B4) at (3,1);
    
    \coordinate (C1) at (0,2);
    \coordinate (C2) at (1,2);
    \coordinate (C3) at (2,2);
    \coordinate (C4) at (3,2);
    
    \coordinate (D+1) at (0,3);
    \coordinate (D+2) at (1,3);
    \coordinate (D+3) at (2,3);
    \coordinate (D+4) at (3,3);
    
    \coordinate (D1) at (0,3);
    \coordinate (D2) at (1,3);
    \coordinate (D3) at (2,3);
    \coordinate (D4) at (3,3);

     \draw[fill=gray!30, gray!30] (A1) to [bend left=40] (A2) to [bend right=50] (A3) to [bend left=20] (A4) to [bend left=40] (B4) to [bend right=50] (C4) to [bend left=20] (D4);
  \draw[fill=gray!30, gray!30] (A1) to [bend left=40] (B1) to [bend right=50] (C1) to [bend left=20] (D1) to (D1) to [bend left=40] (D2) to [bend right=50] (D3) to [bend left=20] (D4) ;

  \node at (0.5,2.5) {};
  \node at (1.5,2.5) {};
  \node at (2.5,2.5) {};
  \node at (0.5,1.5) {};
  \node at (1.5,1.5) {${\Gamma}_F$};
  \node at (2.5,1.5) {};
  \node at (0.5,0.5) {};
  \node at (1.5,0.5) {};
  \node at (2.5,0.5) {};

  \draw[color=black] (A1) to [bend left=40] (A2) to [bend right=50] (A3) to [bend left=20] (A4) ; 
  \draw[dashed, gray] (B1) to [bend left=40] (B2) to [bend right=50] (B3) to [bend left=20] (B4) ; 
  \draw[dashed, gray] (C1) to [bend left=40] (C2) to [bend right=50] (C3) to [bend left=20] (C4) ; 
  \draw[color=black] (D1) to [bend left=40] (D2) to [bend right=50] (D3) to [bend left=20] (D4) ; 
  
  \draw[color=black] (A1) to [bend left=40] (B1) to [bend right=50] (C1) to [bend left=20] (D1) ; 
  \draw[dashed, gray] (A2) to [bend left=40] (B2) to [bend right=50] (C2) to [bend left=20] (D2) ; 
  \draw[dashed, gray] (A3) to [bend left=40] (B3) to [bend right=50] (C3) to [bend left=20] (D3) ; 
  \draw[color=black] (A4) to [bend left=40] (B4) to [bend right=50] (C4) to [bend left=20] (D4) ; 
\end{scope}
\begin{scope}[every node/.append style={yslant=0.5},yslant=0.5, shift = {(3,-3)}]
    \coordinate (A1) at (0,0) + (-3,3);
    \coordinate (A2) at (1,0);
    \coordinate (A3) at (2,0);
    \coordinate (A4) at (3,0);
    
    \coordinate (B1) at (0,1);
    \coordinate (B2) at (1,1);
    \coordinate (B3) at (2,1);
    \coordinate (B4) at (3,1);
    
    \coordinate (C1) at (0,2);
    \coordinate (C2) at (1,2);
    \coordinate (C3) at (2,2);
    \coordinate (C4) at (3,2);
    
    \coordinate (D1) at (0,3);
    \coordinate (D2) at (1,3);
    \coordinate (D3) at (2,3);
    \coordinate (D4) at (3,3);

  \node at (0.5,2.5) {};
  \node at (1.5,2.5) {};
  \node at (2.5,2.5) {};
  \node at (0.5,1.5) {};
  \node at (1.5,1.5) {$F_{r,0}^{+}$};
  \node at (2.5,1.5) {};
  \node at (0.5,0.5) {};
  \node at (1.5,0.5) {};
  \node at (2.5,0.5) {};
  
  \draw[color=black] (A1) to [bend left=40] (A2) to [bend right=50] (A3) to [bend left=20] (A4) ; 
  \draw[dashed, gray] (B1) to [bend left=40] (B2) to [bend right=50] (B3) to [bend left=20] (B4) ; 
  \draw[dashed, gray] (C1) to [bend left=40] (C2) to [bend right=50] (C3) to [bend left=20] (C4) ; 
  \draw[color=black] (D1) to [bend left=40] (D2) to [bend right=50] (D3) to [bend left=20] (D4) ; 
  
  \draw[dashed, gray] (A2) to [bend left=40] (B2) to [bend right=50] (C2) to [bend left=20] (D2) ; 
  \draw[dashed, gray] (A3) to [bend left=40] (B3) to [bend right=50] (C3) to [bend left=20] (D3) ; 
  \draw[color=black] (A4) to [bend left=40] (B4) to [bend right=50] (C4) to [bend left=20] (D4) ; 
\end{scope}
\begin{scope}[every node/.append style={
    yslant=0.5,xslant=-1},yslant=0.5,xslant=-1,
    shift = {(3,0)}
  ]
    \coordinate (A1) at (0,0);
    \coordinate (A2) at (1,0);
    \coordinate (A3) at (2,0);
    \coordinate (A4) at (3,0);
    
    \coordinate (B1) at (0,1);
    \coordinate (B2) at (1,1);
    \coordinate (B3) at (2,1);
    \coordinate (B4) at (3,1);
    
    \coordinate (C1) at (0,2);
    \coordinate (C2) at (1,2);
    \coordinate (C3) at (2,2);
    \coordinate (C4) at (3,2);
    
    \coordinate (D1) at (0,3);
    \coordinate (D2) at (1,3);
    \coordinate (D3) at (2,3);
    \coordinate (D4) at (3,3);

  \node at (0.5,2.5) {};
  \node at (1.5,2.5) {};
  \node at (2.5,2.5) {};
  \node at (0.5,1.5) {};
  \node at (1.5,1.5) {$F_{s,0}^{+}$};
  \node at (2.5,1.5) {};
  \node at (0.5,0.5) {};
  \node at (1.5,0.5) {};
  \node at (2.5,0.5) {};
  
  \draw[dashed, gray] (B1) to [bend left=40] (B2) to [bend right=50] (B3) to [bend left=20] (B4) ; 
  \draw[dashed, gray] (C1) to [bend left=40] (C2) to [bend right=50] (C3) to [bend left=20] (C4) ; 
  \draw[color=black] (D1) to [bend left=40] (D2) to [bend right=50] (D3) to [bend left=20] (D4) ; 
  
  \draw[dashed, gray] (A2) to [bend left=40] (B2) to [bend right=50] (C2) to [bend left=20] (D2) ; 
  \draw[dashed, gray] (A3) to [bend left=40] (B3) to [bend right=50] (C3) to [bend left=20] (D3) ; 
  \draw[color=black] (A4) to [bend left=40] (B4) to [bend right=50] (C4) to [bend left=20] (D4) ; 
\end{scope}
\end{scope}

\draw[dotted, gray] (A-1) -- (A+4);
\draw[dotted, gray] (A-2) -- (A+3);
\draw[dotted, gray] (A-3) -- (A+2);
\draw[dotted, gray] (A-4) -- (A+1);
\draw[dotted, gray] (D-1) -- (D+4);
\draw[dotted, gray] (D-2) -- (D+3);
\draw[dotted, gray] (D-3) -- (D+2);
\draw[dotted, gray] (D-4) -- (D+1);

\begin{scope}[shift = {(0,-7)}]
    
\node at (1,4.35) {$\Omegat^{-}$};
\begin{scope}[every node/.append style={yslant=-0.5},yslant=-0.5]
  \node at (1.5,2.5) {};
  \node at (2.5,2.5) {};
  \node at (0.5,1.5) {};
  \node at (1.5,1.5) {$\widetilde{F}_{r,0}^{-}$};
  \node at (2.5,1.5) {};
  \node at (0.5,0.5) {};
  \node at (1.5,0.5) {};
  \node at (2.5,0.5) {};
  \draw[dashed, gray] (0,0) grid (3,3);
  \draw (0,0) rectangle (3,3);
\end{scope}
\begin{scope}[every node/.append style={yslant=0.5},yslant=0.5]
  \coordinate (At-) at (6,0) {};
  \coordinate (Bt-) at (5,0) {};
  \coordinate (Ct-) at (4,0) {};
  \coordinate (Dt-) at (3,0) {};
  \coordinate (Et-) at (6,-3) {};
  \coordinate (Ft-) at (5,-3) {};
  \coordinate (Gt-) at (4,-3) {};
  \coordinate (Ht-) at (3,-3) {};
 
  \draw[fill=gray!30, gray!30] (3,-3) rectangle (6,0);
  
  \node at (4.5,-0.5) {};
  \node at (5.5,-0.5) {};
  \node at (3.5,-1.5) {};
  \node at (4.5,-1.5) {$\widetilde{\Gamma}_F$};
  \node at (5.5,-1.5) {};
  \node at (3.5,-2.5) {};
  \node at (4.5,-2.5) {};
  \node at (5.5,-2.5) {};
  \draw[dashed, gray] (3,-3) grid (6,0);
  \draw (3,-3) rectangle (6,0);
\end{scope}
\begin{scope}[every node/.append style={
    yslant=0.5,xslant=-1},yslant=0.5,xslant=-1
  ]
  \node at (3.5,2.5) {};
  \node at (3.5,1.5) {};
  \node at (3.5,0.5) {};
  \node at (4.5,2.5) {};
  \node at (4.5,1.5) {$\widetilde{F}_{s,0}^{-}$};
  \node at (4.5,0.5) {};
  \node at (5.5,2.5) {};
  \node at (5.5,1.5) {};
  \node at (5.5,0.5) {};
  \draw[dashed, gray] (3,0) grid (6,3);
  \draw (3,0) rectangle (6,3);
\end{scope}

\begin{scope}[shift = {(8,0)}]
\node at (5,4.35) {$\Omegat^{+}$};
\begin{scope}[every node/.append style={yslant=-0.5},yslant=-0.5]
  \coordinate (At+) at (0,3) {};
  \coordinate (Bt+) at (1,3) {};
  \coordinate (Ct+) at (2,3) {};
  \coordinate (Dt+) at (3,3) {};
  \coordinate (Et+) at (0,0) {};
  \coordinate (Ft+) at (1,0) {};
  \coordinate (Gt+) at (2,0) {};
  \coordinate (Ht+) at (3,0) {};
  
  \draw[fill=gray!30, gray!30] (0,0) rectangle (3,3);
  
  \node at (1.5,2.5) {};
  \node at (2.5,2.5) {};
  \node at (0.5,1.5) {};
  \node at (1.5,1.5) {$\widetilde{\Gamma}_F$};
  \node at (2.5,1.5) {};
  \node at (0.5,0.5) {};
  \node at (1.5,0.5) {};
  \node at (2.5,0.5) {};

  \draw[dashed, gray] (0,0) grid (3,3);
  \draw (0,0) rectangle (3,3);
\end{scope}
\begin{scope}[every node/.append style={yslant=0.5},yslant=0.5]
  \node at (3.5,-0.5) {};
  \node at (4.5,-0.5) {};
  \node at (5.5,-0.5) {};
  \node at (3.5,-1.5) {};
  \node at (4.5,-1.5) {$\widetilde{F}_{r,0}^{+}$};
  \node at (5.5,-1.5) {};
  \node at (3.5,-2.5) {};
  \node at (4.5,-2.5) {};
  \node at (5.5,-2.5) {};
  \draw[dashed, gray] (3,-3) grid (6,0);
  \draw (3,-3) rectangle (6,0);
\end{scope}
\begin{scope}[every node/.append style={
    yslant=0.5,xslant=-1},yslant=0.5,xslant=-1
  ]
  \node at (3.5,2.5) {};
  \node at (3.5,1.5) {};
  \node at (3.5,0.5) {};
  \node at (4.5,2.5) {};
  \node at (4.5,1.5) {$\widetilde{F}_{s,0}^{+}$};
  \node at (4.5,0.5) {};
  \node at (5.5,2.5) {};
  \node at (5.5,1.5) {};
  \node at (5.5,0.5) {};
  \draw[dashed, gray] (3,0) grid (6,3);
  \draw (3,0) rectangle (6,3);
\end{scope}
\end{scope}

\draw[dotted, gray] (At-) -- (At+);
\draw[dotted, gray] (Bt-) -- (Bt+);
\draw[dotted, gray] (Ct-) -- (Ct+);
\draw[dotted, gray] (Dt-) -- (Dt+);
\draw[dotted, gray] (Et-) -- (Et+);
\draw[dotted, gray] (Ft-) -- (Ft+);
\draw[dotted, gray] (Gt-) -- (Gt+);
\draw[dotted, gray] (Ht-) -- (Ht+);

\end{scope}

\end{tikzpicture}
    \caption{Curvilinear coordinate transform between the computational space $\Omegat$ and modelling space $\Omega$. The shaded face $\Gamma_F$ is a frictional interface connecting the two blocks, $\Omega^{-}$ and $\Omega^{+}$, of elastic solid.}
    \label{fig:cur-transform}
\end{figure}
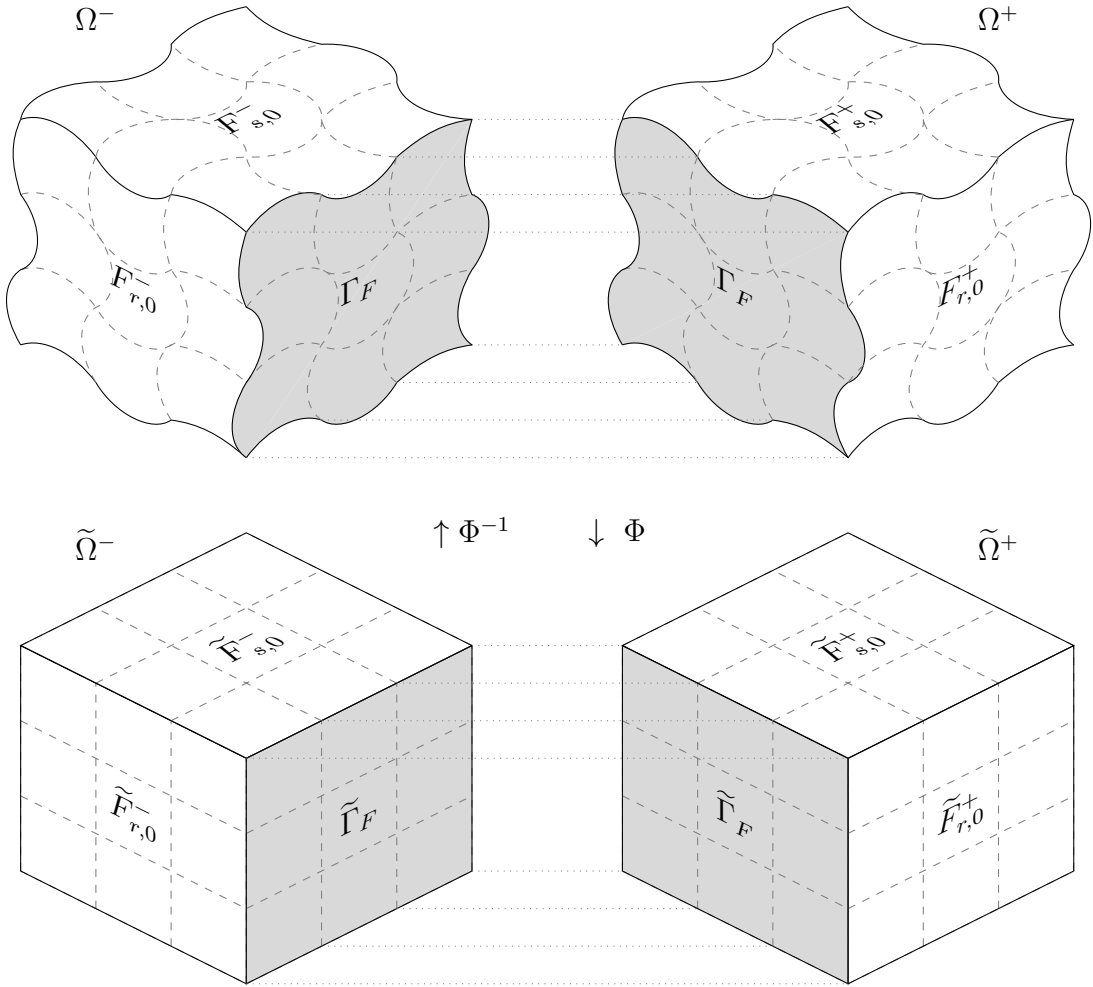
  We consider two geometrically complex elastic solids, in the bounded domains $\Omega^{-}$ and $\Omega^{+}$,  separated by a smooth fault of arbitrarily shaped surface at the interface $\Gamma_{F}$, see Figure \ref{fig:cur-transform}. 
    The two elastic solids, $\Omega^{-}$ and $\Omega^{+}$ are held together at the interface $\Gamma_{F}$ by a high level of  static frictional strength. 
 We are particularly interested in accurate modelling and efficient simulation  of the dynamics of the fault-interface undergoing a nonlinear frictional sliding, when a load of finite magnitude overcomes the  frictional strength at the interface, as well as the numerical computation of the  emitted elastic waves propagating in the medium. 
    \subsection{Elastic wave equation in curvilinear coordinates}
    As in \cite{DURU2022110966,DuruandDunham2016}, for efficient numerical treatments, we map each complex elastic solids $\Omega^{\pm}$  to a reference  unit cube  $\Omegat^{\pm}$ defined by the curvilinear coordinates $(q, r, s) \in \Omegat = [0, 1]^3$, see Figure \ref{fig:cur-transform}.
 
 The Jacobian determinant for the transformation $\Phi^{-1}$ can be written as
\begin{align*}
	J = x_q (y_r z_s - z_r y_s) - y_q(x_r z_s - z_r x_s) + z_q(x_r y_s - y_r x_s).
\end{align*}
Here $\xi_{\eta}$ is the partial derivative ${\partial \xi}/{\partial \eta}$ for $\xi,\eta \in \{x,y,z,q,r,s\}$.
Similarly, 
\begin{align*}
	q_x = \frac{1}{J} (y_r z_s - z_r y_s) && r_x = \frac{1}{J} (z_q y_s - y_q z_s) && s_x = \frac{1}{J} (y_q z_r - z_q y_r), \\
	q_y = \frac{1}{J} (z_r x_s - x_r z_s) && r_y = \frac{1}{J} (x_q z_s - z_q x_s) && s_y = \frac{1}{J} (z_q x_r - x_q z_r), \\
	q_z = \frac{1}{J} (z_r y_s - y_r x_s) && r_z = \frac{1}{J} (y_q x_s - y_s x_q) && s_z = \frac{1}{J} (x_q y_r - x_r y_q) .
\end{align*}

 Let $\bm{Q} = \left(\mathbf{v},  \bm{\sigma}\right)^T$ denote the unknown vector fields in the medium where $\bm{\sigma} \coloneqq (\sigma_{xx},\sigma_{yy},\sigma_{zz}\sigma_{xy},\sigma_{xz},\sigma_{yz})^T$ is the vector of stresses and $\mathbf{v} \coloneqq (v_x,v_y,v_z)^T$ is the particle velocity vector in three space dimensions. The material parameters and the unknown vector fields in the sub-blocks $\Omega^{-}$ and $\Omega^{+}$ will be denoted with the superscripts $-, +$, respectively, that is $\{\mathbf{Q}^{-}, \mathbf{Q}^{+}\}$.
 
  Let $t\ge 0$ be the time variable. For $\mathbf{Q} \in \{\mathbf{Q}^{-}, \mathbf{Q}^{+}\}$,  the equation of motion for each sub-block  in the  transformed  curvilinear coordinates $(q, r, s)$ is given by the elastic wave equation 
\begin{align}\label{eq:transformedEQ}
    \Pt^{-1} \pf{}{t} \bm{Q} = \nabla \cdot \bm{F} (\bm{Q}) + \sum_{\xi \in \{q,r,s\} } \bm{B}_{\xi} (\nabla \bm{Q} ),
\end{align}
where
\begin{align}
         \bm{F}_{\xi} (\bm{Q} )  = 
         \begin{pmatrix}
            J( { \xi_x} \sigma_{xx} +  { \xi_y} \sigma_{xy} +  { \xi_z} \sigma_{xz})\\
            J( { \xi_x} \sigma_{xx} +  { \xi_y} \sigma_{xy} +  { \xi_z} \sigma_{xz})\\
            J( { \xi_x} \sigma_{xx} +  { \xi_y} \sigma_{xy} +  { \xi_z} \sigma_{xz})\\
            0 \\
            0 \\
            0 \\
            0 \\
            0 \\
            0 
         \end{pmatrix},
         &&
         \bm{B}_{\xi} (\nabla \bm{Q}) =
         \begin{pmatrix}
            0 \\
            0 \\
            0 \\
            J  { \xi_x} \pf{v_x}{\xi} \\
            J  { \xi_y} \pf{v_y}{\xi}\\
            J  { \xi_z} \pf{v_z}{\xi}\\
            J( { \xi_y} \pf{v_x}{\xi} +  { \xi_x} \pf{v_y}{\xi}) \\
            J( { \xi_z} \pf{v_x}{\xi} +  { \xi_x} \pf{v_z}{\xi}) \\
            J( { \xi_z} \pf{v_y}{\xi} +  { \xi_y} \pf{v_z}{\xi}) 
         \end{pmatrix},
    \end{align}
with the gradient operator redefined as $\nabla \coloneqq (\pf{}{q},\pf{}{r},\pf{}{s})$ and $\Pt = J^{-1} \bm{P}$, where
\begin{align}\label{eq:S matrix}
        \bm{P} \coloneqq \begin{pmatrix} \rho^{-1} \bm{1} & \bm{0} \\ \bm{0}^T & \bm{C} \end{pmatrix}, && \bm{S}^{-1} \coloneqq \bm{C}.
    \end{align}
Here $\rho : \Omega \mapsto \R_+$ is the density of the medium, $\bm{S}$ is the compliance matrix and $\bm{C}$ is the stiffness matrix of elastic coefficients, present in Hooke's law. In general anisotropic elastic medium  $\bm{C}$ is described by 21 independent coefficients. 
If we consider  an isotropic elastic medium $\bm{C}$ will be described the two independent Lam\'e parameters  $\mu, \lambda \in \R$
 with $\mu > 0$, $\lambda > -\mu $, \cite{Achenbach1973,MarsdenHughes1994,DuruandDunham2016}.
However, as shown in \cite{MarsdenHughes1994} (pages 241--243), for strong ellipticity  we must  have $\mu > 0$, $\lambda > -\mu/2 $. 

  %
 
    Nominally, we will take the positive block to be the `right block', and the negative block to be the `left block'. 
We state a lemma which was proven in \cite{Duru_exhype_2_2019}
\begin{lemma}\label{lem:anti_symmetry}
Consider the split spatial operators in the transformed equation \eqref{eq:transformedEQ}.
The conservative flux term $\bm{F}_{\xi} (\bm{Q} )$ and the non-conservative-products flux term $\bm{B}_{\xi} (\nabla \bm{Q})$ defined in \eqref{eq:transformedEQ} satisfy the skew-symmetric property, that is
\[
\mathbf{Q}^T\mathbf{B}_\xi \left(\grad\mathbf{Q}\right) - \frac{\partial \mathbf{Q}^T}{\partial \xi} \mathbf{F}_\xi \left(\mathbf{Q} \right) = 0.
\]
\end{lemma}

On the interface $(x_0,y_0,z_0) \in \Gamma_F$, define the positively pointing unit normal vector
\begin{align}\label{eq:normal_vector}
    \bm{n} (x_0,y_0,z_0) =  \frac{1}{|\mathbf{q}_\xi|} \begin{pmatrix} q_x \\ q_y \\ q_z \end{pmatrix} \Big|_{(x_0,y_0,z_0)},  \quad |\mathbf{q}_\xi| = \sqrt{q_x^2 + q_y^2 + q_z^2}.
\end{align}
The outward unit normal on the fault for left block $\Omega^{-}$ is 
$\bm{n}^{-} = \bm{n} $, and $\bm{n}^{+} = -\bm{n}$ for right block $\Omega^{+}$.
Define the  Cartesian components of the velocity vector $\mathbf{v}$ and traction vector $\mathbf{T}$ on the fault interface $ \Gamma_F$
\begin{align}\label{eq:velocity_tractions}
\mathbf{v} = \begin{pmatrix}
v_{x}\\
v_{y}\\
v_{z}
\end{pmatrix},
 \quad
\mathbf{T} = \begin{pmatrix}
T_{x} \\
T_{y} \\
T_{z} 
\end{pmatrix} = \bar{\bar{\sigma}}\mathbf{n},
\quad
 \bar{\bar{\sigma}} =
\begin{pmatrix}
\sigma_{xx}&\sigma_{xy}&\sigma_{xz}\\
\sigma_{xy}&\sigma_{yy}&\sigma_{yz}\\
\sigma_{xz}&\sigma_{yz}&\sigma_{zz}
\end{pmatrix}.
\end{align}
Another important consequence of the choice of the split spatial operators in the transformed equation \eqref{eq:transformedEQ} is the fact
\begin{align}\label{state_times_flux}
 \mathbf{Q}^T\mathbf{F}_q \left(\mathbf{Q} \right)  = {J} |\mathbf{q}_\xi|\mathbf{v}^T\mathbf{T}.
\end{align}
Introduce the ${L}^2$ inner-product $\langle \cdot, \cdot \rangle$ in the transformed coordinates given through 
\begin{align}
\langle \mathbf{Q}, \mathbf{F} \rangle \coloneqq \int_{\widetilde{\Omega}} {\left(\mathbf{Q}^T\mathbf{F}\right) Jdqdrds}.
\end{align}
Define the mechanical energy in the medium
We introduce the weighted $L^2$-norm $\|Q\|_{{{P}}}$, and mechanical energy $E(t)$ defined by
\begin{align}\label{eq:energy_cts}
{E}(t):=\|Q\|_{P}^2 =  \langle\mathbf{Q}, \frac{1}{2} \widetilde{\mathbf{P}}^{-1} \mathbf{Q}\rangle =  \int_{\widetilde{\Omega}}\left(\frac{\rho}{2} \abs{\mathbf{v}}^2  + \frac{1}{2}\boldsymbol{\sigma}^T\mathbf{S}\boldsymbol{\sigma}\right)Jdqdrds, 
\end{align}
where $E(t)$ is the sum of the kinetic energy and the strain energy. 
\begin{theorem}\label{Theo:energy_estimate_BC}
Consider the transformed equation of motion \eqref{eq:transformedEQ} defined in the negative and positive elastic blocks, $\widetilde{\Omega}^{-}$ and $\widetilde{\Omega}^{+}$. We denote the elastic energy in the elastic blocks by $E^{\pm}(t)$ and particle velocity and traction vectors at both sides of the interface $\widetilde{\Gamma}_{F}$ by $\mathbf{v}^{\pm}$, $\mathbf{T}^{\pm}$.
The total  mechanical energy in the medium $E(t) = E^-\left(t\right) + E^+\left(t\right)$  satisfy 
  \begin{align}\label{eq:energy_estimate_bc}
\frac{d}{dt}E(t) = \mathrm{IT}_{s}=-\int_{\widetilde{\Gamma}_{F}} \left(\left(\mathbf{v}^+\right)^T\mathbf{T}^+ - \left(\mathbf{v}^-\right)^T\mathbf{T}^- \right)|\mathbf{q}_\xi|Jdrds.
 \end{align}
 \end{theorem}
\begin{proof}
For each $\mathbf{Q} \in \{\mathbf{Q}^-, \mathbf{Q}^+\}$, consider 
    \begin{align}\label{eq:energy_transformed_1}
\frac{d}{dt}\langle\mathbf{Q}, \frac{1}{2} \widetilde{\mathbf{P}}^{-1} \mathbf{Q}\rangle = \int_{\widetilde{\Omega}}\mathbf{Q}^T \div \mathbf{F} \left(\mathbf{Q} \right) dqdrqs + \sum_{{\xi}= q, r, s}\int_{\widetilde{\Omega}}\mathbf{Q}^T\mathbf{B}_{\xi}\left(\grad\mathbf{Q}\right) dqdrqs .
\end{align}
 On the left hand side  of \eqref{eq:energy_transformed_1} we recognize time derivative of the energy.
  On the right hand side of \eqref{eq:energy_transformed_1}, we integrate the conservative flux term by parts and ignore contributions from the boundaries in the $r$- and $s$-directions, we have
      \begin{align}\label{eq:energy_transformed_4}
\frac{d}{dt}\langle\mathbf{Q}, \frac{1}{2} \widetilde{\mathbf{P}}^{-1} \mathbf{Q}\rangle &= \sum_{{\xi}= q, r, s}\int_{\widetilde{\Omega}}\left[\mathbf{Q}^T\mathbf{B}_\xi \left(\grad\mathbf{Q}\right) - \frac{\partial \mathbf{Q}^T}{\partial \xi} \mathbf{F}_\xi \left(\mathbf{Q} \right)\right]dqdrds \\
&+ \int_{0}^{1}\int_{0}^{1}\left(\mathbf{v}^T\mathbf{T}|\mathbf{q}_\xi|{J}\Big|_{q = 1} - \mathbf{v}^T\mathbf{T}|\mathbf{q}_\xi|J\Big|_{q = 0}\right)drds.
\nonumber
\end{align}
 By Lemma \ref{lem:anti_symmetry}, the volume terms in \eqref{eq:energy_transformed_4} vanish. Collecting contributions from both sides of the interface and considering boundary terms only from the fault interface $\widetilde{\Gamma}_{F}$ having
  \begin{align}\label{eq:energy_transformed_3}
\frac{d}{dt}E(t) = \mathrm{IT}_{s}=-\int_{\widetilde{\Gamma}_{F}} \left(\left(\mathbf{v}^+\right)^T\mathbf{T}^+ - \left(\mathbf{v}^-\right)^T\mathbf{T}^- \right)|\mathbf{q}_\xi|Jdrds.
 \end{align}
\end{proof}

Next, we will discuss frictional interface conditions coupling the two blocks at the interface, $\Gamma_F$. In general, friction laws are designed such that the interface term is never positive $\mathrm{IT}_{s} \le 0$, and dissipates the mechanical energy in the medium.

\subsection{Frictional interface conditions}
Here, we prescribe the nonlinear frictional interface conditions coupling the elastic solids. 
On the interface $(x_0,y_0,z_0) \in \Gamma_{F}$, consider the positively pointing unit normal vector $\bm{n}(x_0,y_0,z_0)$ defined  in \eqref{eq:normal_vector}.
Next we  form a locally spanning orthonormal basis with the vectors $\bm{m}(x_0,y_0,z_0)$ and $\bm{l}(x_0,y_0,z_0)$ given through a change of variable to the computational space
\begin{align*}
    \bm{m}\left(x_0,y_0,z_0\right) \coloneqq 
     \frac{\bm{m_0} - \l \bm{n} , \bm{m}_0 \r \bm{n} }{|\bm{m_0} - \l \bm{n} , \bm{m}_0 \r \bm{n}|} \Big|_{\left(x_0,y_0,z_0\right)}, 
    \quad 
    \bm{l} \left(x_0,y_0,z_0\right) \coloneqq  \bm{n} \times \bm{m}\Big|_{\left(x_0,y_0,z_0\right)},
\end{align*}
where $\bm{m}_0$ is a vector not in the span of $\bm{n}$. 
We often drop the evaluation point when this is clear from context.

For each point $(x_0,y_0,z_0)$ on the interface, we denote the particle velocity and traction vectors at both sides of the interface $\widetilde{\Gamma}_{I}$ by $\mathbf{v}^{\pm}$, $\mathbf{T}^{\pm}$. Thus $\mathbf{T}^-$ is the traction exerted by the left block on the fault interface and $-\mathbf{T}^+$ is the traction exerted by the right block on the fault interface.
Rotate the traction and particle velocity into the local orthonormal basis vector
\begin{align}
    \begin{pmatrix}
    v_n^\pm\\
    v_m^\pm\\
    v_\ell^\pm
    \end{pmatrix}
    =
    \mathbf{R}\mathbf{v}^{\pm},
    \quad
    \begin{pmatrix}
    T_n^\pm\\
    T_m^\pm\\
    T_\ell^\pm
    \end{pmatrix}
    =
    \mathbf{R}\mathbf{T}^{\pm},
    \quad
    \mathbf{R} = \begin{pmatrix}
 \bm{n}^T  \\
\bm{m}^T\\
\bm{l}^T 
\end{pmatrix}.
\end{align}
Note that $\mathbf{R}^{-1} = \mathbf{R}^T$.

Newtonian physics imposes the force balance condition
\begin{align}\label{faulttraction}
    \mathbf{T}^{-} = -(-\mathbf{T}^{+}) = \mathbf{T}^{+} \iff T_j^- = T_j^+ = T_j,\quad j \in \{ n, m, l \}.
\end{align}

The particle velocity can be discontinuous across the interface, hence we define the jump of velocities across the interface by
\begin{align}
    \lJump{{v}_j \rJump} \coloneqq v_{j}^{+} - v_{j}^{-}, \quad j \in \{ l, m, n \}.
\end{align}
It is assumed that the interface does not open, that is
\begin{align}
     \lJump{{v}_n \rJump} : = 0 \qquad \text{for all } t \in [0,T],
\end{align}
We allow for shear ruptures. 
Hence the magnitude of the slip rate $V$ and the shear strength $\tau$ are calculated by
\begin{align}\label{sliprate_and_strength}
    V  =  \left( \lJump{{v}_l \rJump}^2 + \lJump{{v}_m \rJump}^2 \right)^{1/2}, \quad \tau = \left(T_m^2 + T_l^2\right)^{1/2}.
\end{align}

Define the compressive-norm-stress $\sigma_n \coloneqq -T_n > 0$, on the interface. 
The shear strength $\tau$ is constrained by  the frictional constitutive relation 
\begin{align}\label{rate-and-state-frictionlaw}
    \tau = f(V, \psi) \sigma_n,
\end{align}
where $f(V, \psi) > 0$ is a generic nonlinear friction, coefficient encapsulated in the rate-and-state framework, $V$ is the slip-rate, and $\psi$ is a state variable that will be defined below.
Further, we assume that the motion is parallel to the shear traction (forces), that is
\begin{align}
    \frac{T_j}{\tau} =  \frac{\lJump{{v}_j \rJump}}{V},  \quad j \in \{ l, m\},  
 \end{align}
and we have
\begin{align}\label{eq:Friction_law_0}
    {T_j} =   \alpha \lJump{{v}_j \rJump}, \quad \alpha( V, \sigma_n, \psi) = \sigma_n\frac{f(V, \psi)}{V} \ge 0, \quad j \in \{ l, m \}.
\end{align}
The functional $\alpha\ge 0$ parameterises the frictional strength of the fault interface. Note that $\alpha \to \infty \implies V \to 0$, any load of finite magnitude can never overcome the frictional strength and break the interface.
The interface conditions can be summarised as:
\begin{align}
    &\text{Force balance}, \quad {T}_j^- = {T}^+_j = T_j;
    \label{eq:Force_balance}\\
    &\text{No opening}, \quad \lJump{{v}_n \rJump}  = 0;
    \label{eq:No_opening}\\
    &\text{Friction law}, \quad   {T_j} =   \alpha \lJump{{v}_j \rJump} , \quad \alpha = \sigma_n\frac{f(V, \psi)}{V} \ge 0, \quad j \in \{ l, m \}.
     \label{eq:Friction_law}
\end{align}

\begin{theorem} [Theorem 1, \cite{DuruandDunham2016}] \label{theo:boundedness}
Consider the transformed equation of motion \eqref{eq:transformedEQ} defined in the negative and positive elastic blocks, $\widetilde{\Omega}^{-}$ and $\widetilde{\Omega}^{+}$, subject to the nonlinear interface conditions \eqref{eq:Force_balance}--\eqref{eq:Friction_law} at the interface $\widetilde{\Gamma}_{I}$. We denote the mechanical energy in the medium by $E^{\pm}(t)$.
The total  mechanical energy in the medium $E(t) = E^-\left(t\right) + E^+\left(t\right)$  satisfy 
\begin{align}
    \frac{d}{dt}{E}(t) = \mathrm{IT}_{s}=-\int_{\widetilde{\Gamma_F}} \alpha |V|^2 \ J|\mathbf{q_\xi}| drds\le 0, \quad \alpha |V|^2= \sigma_n \ V f(V, \psi) \ge 0. 
\end{align}
\end{theorem}
\begin{proof}
The proof follows from Theorem \ref{eq:energy_estimate_bc} and the interface conditions \eqref{eq:Force_balance}--\eqref{eq:Friction_law}.
Consider the interface term in \eqref{eq:energy_estimate_bc} with force balance \eqref{eq:Force_balance}, we have 
 \begin{align*}
 \mathrm{IT}_{s}=-\int_{\widetilde{\Gamma}_{I}} \left(T_l\lJump{{v}_l \rJump} + T_m\lJump{{v}_m \rJump} + T_n\lJump{{v}_n \rJump}\right)|\mathbf{q}_\xi|Jdr \, ds.
 \end{align*}
 The no opening condition $\lJump{{v}_n \rJump}  = 0$ and the friction law \eqref{eq:Friction_law} yield the result
 \begin{align*}
 \mathrm{IT}_{s}=-\int_{\widetilde{\Gamma_F}} \alpha |V|^2  \ J|\mathbf{q_\xi}| dr \, ds<0, \quad \alpha = \sigma_n\frac{f(V, \psi)}{V} \ge 0.
\end{align*}
\end{proof}
In particular, the interface term $\mathrm{IT}_{s}$ encodes the rate of work done by friction on the fault which is dissipated as heat. 
Note that Theorem \ref{theo:boundedness} proves that friction dissipates energy and the solutions of IBVP  are bounded.  
This does not prove well-posedness of IBVP, though. 
The proof of well-posedness of IBVPs for non-linear friction laws is beyond the scope of this work.

Below we will describe typical non-linear friction laws considered in this study.
\subsubsection{Slip-weakening friction law}
%
We will consider first the slip-weakening friction law \cite{andrews1985dynamic}. This is a regularisation of Coulomb's friction, and it is widely used  in applications e.g., to model megathrust earthquake dynamics \cite{ramos2019transition}.
The interface is held initially by a high-level of static friction coefficient, and the  evolving friction coefficient weakens linear with slip $S$ until it reaches the dynamic friction coefficient, as 
\begin{equation}\label{eq:slip-weakening}
\begin{split}
& f\left( S\right)  = \left \{
\begin{array}{rl}
 f_s -\left(f_s-f_d\right)\frac{ {S}}{d_c},  & \text{if}  \quad {S} \le d_c,\\
f_d, \quad {}  \quad {}& \text{if} \quad {S} \ge d_c,
\end{array} \right.
\end{split}
\end{equation}
where $ 0< f_d < f_s$  are the dynamic and static friction coefficients, $d_c>0$ is the critical slip-distance. The slip $S$ evolves according to
\begin{align}
\frac{dS}{dt} = V,
\end{align}
where $V $ is the absolute slip-rate defined in \eqref{sliprate_and_strength}. 
We introduce the peak frictional strength on the fault $\tau_p =  f_s \sigma_n$ and the residual  frictional strength on the fault $\tau_r =  f_d \sigma_n$, where $\sigma_n > 0$ is the compressive-normal-stress.
By \eqref{eq:slip-weakening}, as soon as the load on the fault exceeds the peak strength $\tau_p$, the interface will begin to slip and  the frictional strength will weaken linearly with slip $S$,  until slip  reaches the critical slip-distance  $d_c$. 
When the fault is fully weakened  its residual strength is $\tau_r$. After yielding, the  rupture will spread to adjacent parts of the fault and will be arrested when it meets unfavourable conditions due to geometrical or stress asperities.

\subsubsection{Rate-and-state friction law}\label{sec: rate-and-state friction law}
Modern nonlinear friction laws are encapsulated in the rate-and-state framework \eqref{rate-and-state-frictionlaw}, $\tau = f(V, \psi) \sigma_n$
with $f(V, \psi) \ge 0$, $f\left(0,\psi\right) = 0$ and $\partial f\left(V,\psi\right)/\partial V > 0$. Here, $\sigma_n >0$ is the compressive-normal-stress $f\left(V,\psi\right)$ is the nonlinear friction coefficient,  $V$ is the slip-rate, and $\psi$ is the state variable.
We consider formulations in which the state variable $\psi$ is non-dimensional and is governed by the generic state evolution equation
\begin{equation}
  \frac{d \psi}{d t} = g\left(V,\psi\right).
  \label{eq:StateLaw}
\end{equation}
In general $f\left(V,\psi\right)$ and $g\left(V,\psi\right)$ are empirical expressions obtained from laboratory experiments \cite{dieterich1979modeling,Rice1983,rice1983stability}. 
A typical rate-and-state friction coefficient is \cite{falk2001rice}
\begin{equation}
  f(V,\psi) = a  \operatorname{arcsinh}\left( \frac{V}{2V_0}
  e^{\psi/a} \right),
  \label{eq:regularizedFriction}
\end{equation}
where the friction parameters  $a$ and $V_0$ are real and positive. These parameters will be described below. Note that $f(V,\psi) \ge 0$, for all $\psi$.  Common evolution laws for the  state variable $\psi$ are: \cite{ruina1983slip}
 \begin{enumerate}
 \item aging law:
\begin{equation}
 g(V,\psi) = \frac{b V_0}{d_c} \left( e^{(f_0-\psi)/b} - \frac{V}{V_0} \right) \text{; and,}
  \label{eq:agingLaw}
\end{equation}
\item slip law:
\begin{equation}
  g(V,\psi) = -\frac{V}{d_c}\left(f(V,\psi) - f_{ss}(V)\right),
  \quad
  \label{eq:slipLaw}
\end{equation}
   \end{enumerate}
where $f_{ss}(V)$ is an arbitrary steady state friction coefficient. Some commonly used forms of the steady state friction coefficient are  the standard expression \cite{falk2001rice}
\begin{align}
f_{ss}(V) = f_0 - \left(b-a\right)\ln{\frac{V}{V_0}}.
 \label{eq:standardslipLaw}
\end{align}
There is also the strongly rate-weakening friction law \cite{dunham2011earthquake}, but this will not be used in this work.
%
Here, \emph{a} is the direct effect parameter, \emph{b} is evolution effect parameter, $V_0$ is the reference slip velocity, $d_c$ is the state-evolution distance, $f_0$ is the steady state friction coefficient at $V_0$.

%

\section{Frictional Gluing Variables (Hat-variables)}\label{sec:hat_variable}
We aim to construct a consistent and stable discrete approximation of the IBVP such that a discrete energy estimate analogous to Theorem \ref{theo:boundedness} can be derived. Some of the discussions here can be found in \cite{DuruandDunham2016}. We have however included the details in order to have a self-contained narrative.  The new result in this section is Theorem \ref{thm:uniq_and_existence}, which proves the existence and uniqueness of the frictional gluing variables for nonlinear friction laws in general anisotropic media. 

As in \cite{DuruandDunham2016}, we will construct a numerical treatment of the interface without introducing stiffness such that the semi-discrete problem can be integrated efficiently using an explicit time-stepping method. We will reformulate the interface condition by introducing transformed (hat-) variables so that we can simultaneously construct (numerical) boundary/interface data for particle velocities and traction.  The hat-variables encode the solution of the IBVP on the interface. The hat-variables  will be constructed such that they preserve the amplitude of the outgoing characteristics and exactly satisfy the physical boundary conditions \cite{DuruandDunham2016}. 

Let ${v_j},{T_j}$ be the velocity and traction vectors at the fault interface, with $j \in \{l, m, n\}$. For $Z_j  > 0$, we define the characteristics
\begin{align}\label{eq:charac}
q_j  = \frac{1}{2}\left(Z_j  v_j  + T_j \right), \quad
p_j  = \frac{1}{2}\left(Z_j  v_j  -  T_j \right), \quad  j \in \{ l, m, n\}.
\end{align}
Here, $q_j$ are the left going characteristics and $p_j$ are the right going characteristics.
The outgoing characteristics at the interface are
\begin{align}\label{eq:charac_2}
q_j^+  = \frac{1}{2}\left(Z_j^+  v_j^+  + T_j^+ \right), \quad
p_j^-  = \frac{1}{2}\left(Z_j^-  v_j^-  -  T_j^- \right), \quad  j \in \{ l, m, n \}.
\end{align}
The characteristics are viewed as functions of velocity, traction and impedance. 
We define the hat-variables $\widehat{v}_j ^{\pm}, \widehat{T}_j ^{\pm}$  preserving the amplitude of  the outgoing characteristics at the interface to be
 \begin{align}\label{eq:hat_variables}
\widehat{q}_j ^{+} \left(\widehat{v}_j ^{+}, \widehat{T}_j^{+}, Z_j ^{+}\right)   \coloneqq {q}_j ^{+} \left({v}_j^{+}, {T}_j^{+}, Z_j^{+} \right), 
&&
\widehat{p}_j ^{-} \left(\widehat{v}_j^{-}, \widehat{T}_j^{-}, Z_j^{-}\right)  \coloneqq {p}_j ^{-}\left({v}_j ^{-}, {T}_j^{-}, Z_j^{-} \right).
 \end{align}
The hat-variables also satisfy the  interface conditions, \eqref{eq:Force_balance}--\eqref{eq:Friction_law}, exactly. Given ${q}_j ^{+},  {p}_j ^{-} $, the procedure will solve \eqref{eq:hat_variables} and \eqref{eq:Force_balance}--\eqref{eq:Friction_law} for the hat-variables, $\widehat{v}_j ^{\pm}, \widehat{T}_j ^{\pm}$. 

Combining the two equations in \eqref{eq:hat_variables}  and ensuring force balance, $\widehat{T}_j ^{-} = \widehat{T}_j ^{+} = \widehat{T}_j $, we have
\begin{align}\label{eq:radiation_damping_line}
\widehat{T}_j  + \eta_j  \lJump{\widehat{v}_j \rJump}  = \Phi_j ,    \quad \Phi_j  = \eta_j \left(\frac{2}{Z_j ^{+}}q_j ^{+} - \frac{2}{Z_j ^{-}}p_j ^{-}\right), \quad  \eta_j  = \frac{Z_j ^{+}Z_j ^{-}}{Z_j ^{+} + Z_j ^{-}} >0, \quad j \in \{ l, m, n \}.
\end{align}
Here, $\Phi_j$ are the stress transfer functionals, that is stresses on a locked interface when $\lJump{\widehat{v}_j \rJump} =0$. And by \eqref{eq:radiation_damping_line}, the stresses will be altered by wave radiations when the fault is slipping $\lJump{\widehat{v}_j \rJump} \ne 0$.
Introducing the compressive normal-stress $\widehat{\sigma}_n = -\widehat{T}_n$, as before we have 
\begin{align}
    &\text{Force balance}, \quad \widehat{T}_j  + \eta_j  \lJump{\widehat{v}_j \rJump}  = \Phi_j, \quad  j \in \{l, m, n \};
    \label{eq:Force_balance_hat}\\
    &\text{No opening}, \quad \lJump{\widehat{v}_n \rJump}  = 0;
    \label{eq:No_opening_hat}\\
    &\text{Friction law}, \quad   \widehat{T}_j =   \widehat{\alpha} \lJump{\widehat{v}_j \rJump} , \quad j \in \{l, m \}, \quad \widehat{\alpha}\left(\widehat{V}\right)=\widehat{\sigma}_n\frac{f\left(\widehat{V}, \psi\right)}{\widehat{V}}, \quad \widehat{V}  =  \left(\lJump{\widehat{v}_l \rJump}^2 + \lJump{\widehat{v}_m \rJump}^2 \right)^{1/2}.
     \label{eq:Friction_law_hat}
\end{align}

We need to solve \eqref{eq:Force_balance_hat}--\eqref{eq:Friction_law_hat} for the hat-functions  $\widehat{T}_j$ and $\lJump{\widehat{v}_j \rJump}$. In the normal direction $j=n$, we have closed form solutions 
$$
\lJump{\widehat{v}_n \rJump} = 0, \quad \widehat{T}_n = \Phi_n.
$$
In the tangential directions we have four coupled algebraic equations to solve
\begin{align}\label{eq:hat_radiation_line}
    \widehat{T}_j  + \eta_j  \lJump{\widehat{v}_j \rJump}  = \Phi_j, \quad 
    \widehat{T}_j =   \widehat{\alpha} \lJump{\widehat{v}_j \rJump} , \quad j \in \{ l, m \}.
\end{align}
Next, using 
\begin{align}\label{eq:hat_slip_velocity}
    \lJump{\widehat{v}_j \rJump} = \left(\eta_j +\widehat{\alpha}\right)^{-1}\Phi_j,\quad
    j \in \{ l, m \},
\end{align}
we obtain a nonlinear algebraic problem for the absolute slip-rate $\widehat{V}$
\begin{align}\label{normalisingslitrate}
         \sqrt{( \eta_l \widehat{V} + \widehat{\alpha} \widehat{V})^{-2} \Phi_{l}^2 + ( \eta_m \widehat{V} + \widehat{\alpha} \widehat{V})^{-2} \Phi_{m}^2 }   = 1.
\end{align}
 
 Once we compute the absolute slip-rate $\widehat{V}$ from \eqref{normalisingslitrate}, we can determine $\widehat{T}_j$ and $\lJump{\widehat{v}_j \rJump}$ from \eqref{eq:hat_radiation_line} and \eqref{eq:hat_slip_velocity}. Finally we have
 \begin{align}
      \widehat{T}_j ^{-} = \widehat{T}_j ^{+} = \widehat{T}_\eta , \quad
     \widehat{v}_j^+ = \frac{2p^{-}_j + \widehat{T}_j}{Z_j^{-}} + \lJump{\widehat{v}_j \rJump}, \quad
     \widehat{v}_j^- = \frac{2q^{+}_j - \widehat{T}_j}{Z_j^{+}} - \lJump{\widehat{v}_j \rJump}
 \end{align}
 We have equivalently redefined the interface condition
 \begin{align}\label{eq:transformed_interface_condition}
    {v}_j^{\pm} = \widehat{v}_j^{\pm}, \quad  {T}_j^{\pm} = \widehat{T}_j,
 \end{align}
 
By construction, the hat-variables $\widehat{v}_j ^{\pm}$, $\widehat{T}_j ^{\pm}$ satisfy the following algebraic identities:
\begin{align}\label{eq:identity_1}
{q}_j  \left(\widehat{v}_j ^{+}, \widehat{T}_j ^{+}, Z_j ^{+}\right)   = {q}_j  \left({v}_j ^{+}, {T}_j ^{+}, Z_j ^{+} \right), \quad
{p}_\eta  \left(\widehat{v}_j ^{-}, \widehat{T}_j ^{-}, Z_j ^{-}\right)  = {p}_j \left({v}_j ^{-}, {T}_j ^{-}, Z_j ^{-} \right), 
\end{align}
\begin{align}\label{eq:identity_2}
\left(q_j ^2  \left({v}_j ^{+}, {T}_j ^{+}, Z_j ^{+}\right)\right)-{p}_j ^2 \left(\widehat{v}_j ^{+},\widehat{T}_j ^{+}, Z_j ^{+}\right) = Z_j ^{+}\widehat{T}_j \widehat{v}_j^{+}, \quad
 p^2_j \left({v}_j ^{-}, {T}_j^{-}, Z_j ^{-} \right) -{q}^2_j \left(\widehat{v}_j^{-}, \widehat{T}_j^{-}, Z_j^{-}\right) = -Z_j ^{-}\widehat{T}_j \widehat{v}_j ^{-},
\end{align}
\begin{align}\label{eq:identity_3}
&\frac{1}{Z_j ^+}\left(q^2_j \left({v}_j ^{+}, {T}_j ^{+}, Z_j ^{+} \right)- {p}^2_j \left(\widehat{v}_j ^{+}, \widehat{T}_j ^{+},Z_j ^{+}\right)\right)  + \frac{1}{Z_j ^-}\left(p^2_j \left({v}_j ^{-}, {T}_j ^{-}, Z_j ^{-} \right) -{q}^2_j \left(\widehat{v}_j ^{-}, \widehat{T}_j ^{-}, Z_j ^{-}\right)\right) =  \widehat{T}_j  \lJump{\widehat{v}_j \rJump},
\end{align}
\begin{align}\label{eq:identity_4}
 \widehat{T}_n  \lJump{\widehat{v}_n \rJump} = 0, \quad \widehat{T}_j  \lJump{\widehat{v}_j \rJump} &= \frac{\widehat{\alpha}}{ \left(\eta_j +\widehat{\alpha}\right)^2}\Phi_j^2 \ge  0 ,
 \quad j \in \{ l, m \},\\
 \sum_{j \in \{l, m,n\} } \widehat{T}_j  \lJump{\widehat{v}_j \rJump} &= \sum_{j \in \{l, m\} }\frac{\widehat{\alpha}}{ \left(\eta_j +\widehat{\alpha}\right)^2}\Phi_j^2 = \widehat{\alpha}|\widehat{V}|^2  \ge 0.
 \label{eq:identity_5}
\end{align}

We note that  hat-variables $(\widehat{\mathbf{v}}^{\pm}, \widehat{T})$  involve the computation of $\widehat{V}$, which is the solution of the  nonlinear algebraic problem \eqref{normalisingslitrate} for nonlinear friction laws. 
It is, however, not obvious if there is a solution for \eqref{normalisingslitrate}, and if the solution is unique in general media with $0<\eta_m \ne \eta_l >0$.
As we will show below, under the assumptions given for the nonlinear friction coefficient $f(V, \psi)$, we can prove that the nonlinear algebraic problem \eqref{normalisingslitrate} has a unique solution in general anisotropic elastic media.

Consider the derived quantities $(\mathbf{v}^{\pm}, T^{\pm})$ defined on the fault interface $\Gamma_{F}$ associated to a solution $(\mathbf{v}^{\pm}, \boldsymbol{\sigma}^{\pm})$ that solves the IVP without the specified interface conditions.
As we have not yet specified the interface conditions for $(\mathbf{v}^{\pm}, T^{\pm})$, the IBVP across the two blocks is under-determined and multiple solutions $(\mathbf{v}^{\pm}, \sigma^{\pm})$ could exist.
We will formulate a projection operator $P$ that maps $(\mathbf{v}^{\pm}, T^{\pm})$ to a (hat-) quantity $(\widehat{\mathbf{v}}^{\pm}, \widehat{T})$ which are uniquely determined.
The hat-quantities additionally satisfy the interface conditions given in \eqref{eq:hat_radiation_line}--\eqref{eq:hat_slip_velocity} of the rupturing fault. 
We now make the discussion more formal.
Let $(v^{\pm},\sigma^{\pm})$ be a solution on $\overline{\Omega}$ of \eqref{eq:transformedEQ} corresponding to a derived quantity $(\mathbf{v}^{\pm}, T^{\pm})$. 
In the following discussion, assume that all solutions $(v^{\pm},\sigma^{\pm})$ obey the linear elastic wave equation \eqref{eq:transformedEQ}.
Define 
$$U \coloneqq \{ (\mathbf{v}^{\pm}, T^{\pm}) \ | \ (v^{\pm},\sigma^{\pm}) \quad \text{satisfies the Cauchy problem \eqref{eq:transformedEQ} without the interface conditions \eqref{eq:hat_radiation_line}--\eqref{eq:hat_slip_velocity}} \},$$
and
$$\widehat{U} \coloneqq U \cap \{ (\widehat{\mathbf{v}}^{\pm}, \widehat{T}) \ | \ (\widehat{\mathbf{v}}^{\pm}, \widehat{T}) \text{ satisfies  \eqref{eq:hat_radiation_line}--\eqref{eq:hat_slip_velocity}} \},$$
 then $\widehat{U} \subset U$.
We require $P : U \mapsto \widehat{U}$.
Define the projection operator $P$ through 
\begin{align}\label{eq: projection}
    P( (\mathbf{v}^{\pm}, T^{\pm}) ) = (\widehat{\mathbf{v}}^{\pm}, \widehat{T})
\end{align}
where $(\widehat{\mathbf{v}}^{\pm}, \widehat{T})$ is obtained by solving \eqref{normalisingslitrate} for $\widehat{V}$ and following the procedure described above. 
Under reasonable assumptions of the interface and boundary conditions, like those given above, we can prove well-posedness of this projection operator. 
\begin{theorem}\label{thm:uniq_and_existence}
    Let $P : U \mapsto \widehat{U}$ be as described in \eqref{normalisingslitrate} and assume the interface conditions previously given. 
    Assume that the friction co-efficient function $f$ obeys $f(V, \psi) \ge 0$, $f\left(0,\psi\right) = 0$ and $\partial f\left(V,\psi\right)/\partial V > 0$.
    Then the operator $P$ is continuous, $U \ne \emptyset$, and $|\widehat{U}| = 1$. 
\end{theorem}
\begin{proof}
    { The set of functions $U$ is non-empty as this contains the case of isotropic linear elasticity, with $\eta_l = \eta_m = \eta >0$, studied in \cite{DuruandDunham2016}.}
    Now consider the continuous function $\Theta : (0,\infty) \mapsto (0,\infty)$ given through 
    \begin{align}\label{eq:function_theta}
        \Theta(\theta) \coloneqq \sqrt{( \eta_l \theta + \widehat{\sigma}_n f (\theta, \psi ))^{-2} \Phi_{l}^2 + ( \eta_m \theta + \widehat{\sigma}_n f (\theta, \psi ))^{-2} \Phi_{m}^2 }.
    \end{align}
    The function $\Theta$ obeys: $\lim_{\theta \to 0^+} \Theta (\theta) = \infty$; $\lim_{\theta \to \infty} \Theta (\theta) = 0$; and is (strictly) monotonically decreasing. 
    \begin{figure}[h!]
        \centering
        \includegraphics[width = 0.8\textwidth]{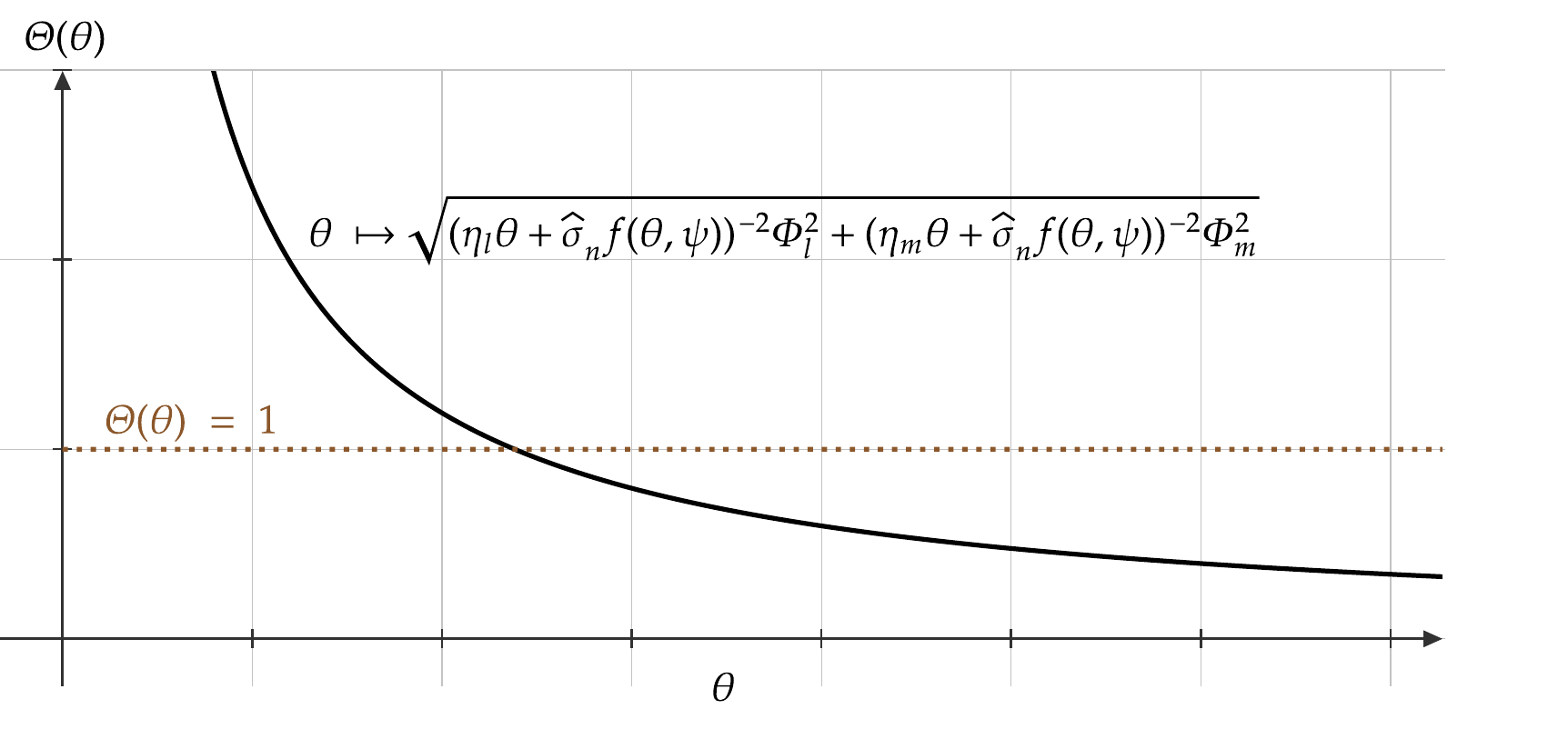}
        \caption{A graphical description of the function $\theta \to \Theta (\theta)$ given in \eqref{eq:function_theta}  and the constant function $\Theta (\theta) = 1$ as defined by \eqref{normalisingslitrate}. The point  of intersection of the functions, $(\theta^*, 1)$ for $\Theta (\theta^*) = 1$,   determines the  unique solution $\widehat{V} = \theta^*$ for the algebraic problem  \eqref{normalisingslitrate}. }
        \label{fig:friction}
    \end{figure}
    Therefore the level set $\{ \theta^* \ | \ \Theta(\theta^*) = 1 \}$ contains precisely one element, call this $\widehat{V}$. Equations \eqref{eq:hat_radiation_line}--\eqref{eq:hat_slip_velocity} are completely determined by $\widehat{V}$, showing $|\widehat{U}| = 1$. Finally, as $P$ projects to a unique element, it is a continuous operator.
    
\end{proof}
\begin{remark}
    Theorem \ref{thm:uniq_and_existence} proves existence and uniqueness of the hat-variables for nonlinear friction laws in general anisotropic elastic solids, as conjectured in \cite{DuruandDunham2016}.
\end{remark}

We conclude this section by reformulating Theorem \ref{theo:boundedness} in terms of the hat-variables.
\begin{theorem}  \label{theo:boundedness_hat}
Let $\mathbf{E}^{\pm}$ be given in \eqref{eq:energy_cts} and $(q,r,s)$ be the transformed curvilinear coordinates. 
Define $\mathbf{E}(t) \coloneqq \mathbf{E}^+(t) + \mathbf{E}^- (t)$, then 
\begin{align}\label{eq:boundedness_hat}
    \frac{d}{dt}\mathbf{E}(t) = \mathrm{IT}_{s}=
    -\int_{\widetilde{\Gamma}_F} \widehat{\alpha}|\widehat{V}|^2 J|\mathbf{q_\xi}| dr \, ds\le 0. 
\end{align} 
\end{theorem}
\begin{proof}
The proof is similar the  proof of Theorem \ref{theo:boundedness}.
Consider the interface term in \eqref{eq:energy_estimate_bc} and impose the transformed interface conditions \eqref{eq:transformed_interface_condition}, that is 
$$
    {v}_j^{\pm} = \widehat{v}_j^{\pm}, \quad  {T}_j^{\pm} = \widehat{T}_j, 
$$
 we have 
 \begin{align}\label{eq:IT_term_0}
 \mathrm{IT}_{s}=-\int_{\widetilde{\Gamma}_{F}} \left(\widehat{T}_l\lJump{\widehat{v}_l \rJump} + \widehat{T}_m\lJump{\widehat{v}_m \rJump} + \widehat{T}_n\lJump{\widehat{v}_n \rJump}\right)|\mathbf{q}_\xi|Jdrds.
 \end{align}
 The no opening condition $\lJump{\widehat{v}_n \rJump}  = 0$ and  the identities \eqref{eq:identity_4}--\eqref{eq:identity_5} yield the result
 \begin{align}
 \mathrm{IT}_{s}=-\int_{\widetilde{\Gamma}_F} \left(\sum_{j \in \{l, m\} }\frac{\widehat{\alpha}}{ \left(\eta_j +\widehat{\alpha}\right)^2}\Phi_j^2\right) J|\mathbf{q_\xi}| dr \, ds =-\int_{\widetilde{\Gamma}_F} \widehat{\alpha}|\widehat{V}|^2 J|\mathbf{q_\xi}| dr \, ds\le 0.
\end{align}
\end{proof}

\section{DP SBP numerical approximation}
In this section, we derive high order accurate and stable discrete numerical approximation of the  elastic wave equation \eqref{eq:transformedEQ} in curvilinear coordinates subject to the nonlinear friction laws \eqref{eq:Force_balance}--\eqref{eq:Friction_law} on the fault.
We will discretise the spatial derivatives using SBP finite difference operators formulated in the DP framework \cite{Mattsson2017,DURU2022110966,williams2021provably}.
In particular, the standard DP upwind SBP operators \cite{Mattsson2017}  and the recently derived $\alpha$-DRP SBP operators \cite{williams2021provably} both satisfy the DP SBP property and will be used. 
The DP SBP operators come in pairs, the forward difference operator $D_{+}$ and the backward difference operator $D_{-}$. 
As in \cite{DURU2022110966} we will carefully combine the DP SBP operators to preserve, in the discrete setting, the skew-symmetric property given in Lemma \ref{lem:anti_symmetry}.
Our primary objective is to numerically implement the nonlinear friction laws \eqref{eq:Force_balance}--\eqref{eq:Friction_law} using penalties such that we can prove energy-stability similar to the continuous analogue Theorem \ref{theo:boundedness_hat}.

\subsection{DP and upwind SBP operators for approximating spatial derivative}
In \cite{DURU2022110966} the term DP SBP operator was introduced. For our discussions to be self-contained we will also formally define the DP SBP framework. Firstly, discretise the reference computational cube $\Omegat =  [0,1]^3$ with an evenly spaced mesh across each axis, $\xi = q, r, s$. 
    We will treat the left and right unit cubes independently in the discretisation, and use penalties, with the friction law \eqref{eq:Force_balance}--\eqref{eq:Friction_law}, to glue them together.
   %
    For each $\xi \in \{q,r,s\}$, consider the uniform discretisation of the unit interval $\xi \in [0, 1]$ 
    \begin{align}
        \xi_{i} \coloneqq i/n_{\xi}, && i \in \{0,\dots, n_{\xi} \}, 
    \end{align}
    where $n_{\xi} + 1$ is the number of grid-points on the $\xi$-axis. 
    
    We will use DP SBP operators \cite{Mattsson2017,DURU2022110966,williams2021provably,DovgilovichSofronov2015} to approximate the spatial derivatives, $\partial/{\partial\xi}$. 
	For each $\xi \in \{q,r,s\}$ define $H_{\xi} \coloneqq \text{diag}\left(h_0^{(\xi)} , \dots, h_{n_{\xi}}^{(\xi)}\right)$, with $h_j^{(\xi)} >0$ for all $j \in \{ 0, 1, \dots, n_\xi \}$.
	We consider the DP derivative operators $D_{+\xi},D_{-\xi} : \R^{n_{\xi} +1 } \mapsto \R^{n_{\xi} +1 }$ so that the SBP property holds
	\begin{align}\label{eq:upw_SBP}
        (D_{+\xi} \bm{f} )^T H_{\xi} \bm{g} +  \bm{f}^T H_{\xi} (D_{-\xi }\bm{g}) = f(\xi_{n_{\xi}})g(\xi_{n_{\xi}}) - f(\xi_0)g(\xi_0),
    \end{align}
    where  $\bm{f} = (f(\xi_0), \dots , f(\xi_{n_{\xi}}) )^T$, $\bm{g} = (g(\xi_0), \dots , g(\xi_{n_{\xi}}) )^T$ are vectors sampled from weakly differentiable functions of the $\xi$ variable. 
 %
    Introduce the matrices
    \begin{align*}
     B = e_ne_n^T - e_1 e_1^T, &&
     e_1 \coloneqq (1, 0,\dots , 0)^T, &&
      e_n \coloneqq (0, \dots , 0, 1 )^T,
\end{align*}
\begin{align}\label{eq:upwind}
    Q_{+\xi} = H_{\xi}D_{+\xi} - \frac{1}{2}B, \quad Q_{-\xi} = H_{\xi}D_{-\xi} - \frac{1}{2}B.
\end{align}
\begin{definition}\label{def:upwind}
		Let $D_{-\xi}$, $D_{+\xi}$, $Q_{-\xi}$, $Q_{+\xi} : \R^{n_\xi} \mapsto \R^{n_\xi}$ be linear operators that solve Equations \eqref{eq:upw_SBP} and \eqref{eq:upwind} for a diagonal weight matrix $H_\xi \in \R^{n_\xi \times n_\xi}$. 
		If the matrix $S_{+} \coloneqq {Q_{+\xi} + Q_{+\xi}^T}$ is negative semi-definite and $S_{-} \coloneqq {Q_{-\xi} + Q_{-\xi}^T}$ is  positive semi-definite, then the 3-tuple $(H_\xi,D_{-\xi},D_{+\xi})$ is called an upwind diagonal-norm DP SBP operator.  
	\end{definition}
%
    We call $(H_{\xi},D_{-\xi},D_{+\xi})$ an upwind diagonal-norm DP SBP operator of order $m$ if the accuracy conditions
	\begin{align} \label{eq:acc}
		    D_{\eta\xi}( \bm{\xi}^i) = i \bm{\xi}^{i-1}
	\end{align}
	are satisfied for all $i \in \{0,\dots,m\}$ and $\eta \in \{-,+\}$ where $\bm{\xi}^i \coloneqq (\xi_0^i, \dots, \xi_{n_\xi}^i)^T$.

	The 1D SBP operators can be extended to higher space dimensions using tensor products $\otimes$. Let $f(q,r,s)$ denote a 3D scalar function, and $f_{ijk} \coloneqq {f}(q_i,r_j, s_k)$ denote the corresponding 3D grid function.
	The  3D scalar grid function $f_{ijk}$ is rearranged row-wise as a vector $\bm{f}$ of length $n_qn_rn_s$. For $\xi \in \{q,r,s\}$ and $\eta \in \{-,+\}$ define:
    %
     {\small
    \begin{align*}
\centering
&\bm{D}_{\eta q} = \left( {D}_{\eta q} \otimes I_{n_r} \otimes I_{n_s}\right), \quad \bm{D}_{\eta r} = \left(I_{n_q} \otimes  {D}_{\eta r}  \otimes I_{n_s}\right), \quad \bm{D}_{\eta s} = \left(I_{n_q} \otimes  I_{n_r}  \otimes  {D}_{\eta s}\right),\quad
\bm{H} = \left( H_{q} \otimes H_r \otimes H_{s}\right), 
\end{align*}
    where $I_{n_{\xi}}$ is the identity matrix of size $n_{\xi} \times n_{\xi}$. 
    }
    The matrix operator $\bm{D}_{\eta\xi}$  will approximate the partial derivative operator in the $\xi$ direction.
    A discrete inner product on $\R^{n_{q} + 1} \times \R^{n_{r} + 1} \times \R^{n_{s} + 1}$ is induced by $\bm{H}$ through 
     \begin{align}
        \l \bm{g} , \bm{f} \r_{\bm{H}} \coloneqq \bm{g}^T \bm{H} \bm{f} = \sum_{i=0}^{n_q} \sum_{j=0}^{n_r} \sum_{k=0}^{n_s}f_{ijk}g_{ijk} h_i^{(q)} h_j^{(r)} h_k^{(s)}, \end{align}
        and  the  discrete energy-norm  for the elastic wave equation is given by
     \begin{align}\label{eq:physical_energy_discrete}
         \|\mathbf{Q}\left(\cdot, \cdot, \cdot, t\right)\|_{HP}^2 = \l\mathbf{Q}, \frac{1}{2} \mathbf{P}^{-1} \mathbf{Q}\r_{\bm{H}}. 
     \end{align}

    \subsection{Numerical approximation in space}
    We consider the transformed elastic wave equation \eqref{eq:transformedEQ} in a reference computational cube $(q, r, s) \in [0,1]^3$, and approximate the spatial operators using the upwind SBP operators. We introduce compact notations by rearranging every 3D scalar grid function  row-wise as a vector of length $n_qn_rn_s$ and the unknown vector field $\mathbf{Q}$ is a vector of length $9 n_qn_rn_s$. As in \cite{DURU2022110966,williams2021provably}, the backward difference operator $D_{-\xi}$ is used to approximate the spatial derivative for the conservative flux term, and the forward difference operator $D_{+\xi}$ approximates derivatives in the non-conservative product term. For  $\mathbf{Q} \in \{\mathbf{Q}^{-}, \mathbf{Q}^{+}\}$, we have
    %
    \begin{align}\label{eq:gen_hyp_transformed_discrete}
    \widetilde{\mathbf{P}}^{-1} \frac{d }{d t} {\mathbf{Q}} = \grad_{D_{-}}\bullet {\mathbf{F} \left({\mathbf{Q}} \right)} + \sum_{\xi \in \{q, r, s\} }\mathbf{B}_\xi\left(\grad_{D_{+}}{\mathbf{Q}}\right),
    \end{align}
where the discrete operator  $\grad_{D_{\eta}} = \left(\mathbf{D}_{\eta q}, \mathbf{D}_{\eta r}, \mathbf{D}_{\eta s}\right)^T$, with $\eta \in \{+, -\}$, is analogous to the continuous gradient operator $\grad = \left(\partial/\partial q, \partial/\partial r, \partial/\partial s\right)^T$. 
Note that the numerical solutions $\{\mathbf{Q}^{-}, \mathbf{Q}^{+}\}$ in the two blocks are unconnected, and independent of each other.
Numerical implementation of  the nonlinear interface conditions, discussed in the next subsection, will be used to coupled them together. 

We state the discrete equivalence of Lemma \ref{lem:anti_symmetry} which was proven in \cite{DURU2022110966}.
 \begin{lemma}[Lemma 4.2 in \cite{DURU2022110966}]\label{lem:anti_sym_prop_disc}
        Consider the semi-discrete approximation given in \eqref{eq:gen_hyp_transformed_discrete}. We have the discrete skew-symmetric form
       {\small
        \begin{align*}
         \left(\left(\left(I_9\otimes\bm{D}_{+\xi}\right)\bm{Q}\right)^T \bm{F}_{\xi}\left(\bm{Q}\right)
            - \bm{Q}^T\mathbf{B}_\xi\left(\grad_{D_{+}}{\bm{Q}}\right)\right)= 0.
        \end{align*}
        }
    \end{lemma}

    For a 2D scalar field $f_{jk} = f(y_j, z_k)$ we also introduce the surface cubature
\begin{align}\label{eq:surface_cubature_1D}
    &\mathbb{I}_{q}\left(\mathbf{f}\right) = \sum_{j=0}^{n_r} \sum_{k=0}^{n_s} J_{jk}\sqrt{q_{xjk}^2+ q_{yjk}^2 + q_{zjk}^2} f_{jk}  h_j^{(r)} h_k^{(s)},
	\end{align}
%
    and 
	\begin{align}\label{eq:disc_boundary_term}
	    \mathbb{I}\mathbb{T}\left((\bm{v}^+)^T\bm{T}^+-(\bm{v}^-)^T\bm{T}^-\right) = \sum_{j=0}^{n_r} \sum_{k=0}^{n_s} J_{jk}\sqrt{q_{xjk}^2+ q_{yjk}^2 + q_{zjk}^2} \left((\bm{v}_{jk}^+)^T\bm{T}_{jk}^+
	    -(\bm{v}_{jk}^-)^T\bm{T}_{jk}^-\right) h_j^{(r)} h_k^{(s)}.
	\end{align}
	The interface surface term $\mathbb{I}\mathbb{T}_s\left((\bm{v}^+)^T\bm{T}^+-(\bm{v}^-)^T\bm{T}^-\right)$ is a numerical approximation of the continuous analogue $\mathrm{IT}_{s}$ defined in \eqref{eq:energy_estimate_bc}.
    \begin{theorem}\label{theo_sbp_sem_discrete}
    Consider the semi-discrete approximation \eqref{eq:gen_hyp_transformed_discrete} of the elastic wave equation.  
    We have
    \begin{align*}
        \frac{d}{dt} \left(\|\mathbf{Q}^-\left(\cdot, \cdot, \cdot, t\right)\|_{HP}^2 + \|\mathbf{Q}^+\left(\cdot, \cdot, \cdot, t\right)\|_{HP}^2\right) =  -\mathbb{I}\mathbb{T}\left((\bm{v}^+)^T\bm{T}^+-(\bm{v}^-)^T\bm{T}^-\right),
    \end{align*}
    where $\mathbb{I}\mathbb{T}_s\left((\bm{v}^+)^T\bm{T}^+-(\bm{v}^-)^T\bm{T}^-\right)$ is the surface term defined in \eqref{eq:disc_boundary_term}.
    \begin{proof}
   For $\mathbf{Q} \in \{\mathbf{Q}^{-}, \mathbf{Q}^{+}\}$ consider
    {\small
    \begin{align}
        \frac{d}{dt} \|\mathbf{Q}\left(\cdot, \cdot, \cdot, t\right)\|_{HP}^2 = \l \bm{Q} , P^{-1} \pf{}{t}  \bm{Q} \r_{H}  = \l \bm{Q}, \grad_{D_{-}}\bullet {\mathbf{F} \left({\mathbf{Q}} \right)} + \sum_{\xi \in \{ q, r, s \} }\mathbf{B}_\xi\left(\grad_{D_{+}}{\mathbf{Q}}\right) \r_{H}.
    \end{align}
    }
    Expanding the right hand side, applying the SBP property \eqref{eq:upw_SBP} and ignoring contributions from the boundaries in the $r$- and $s$-directions gives
    \begin{align*}
       &\sum_{\xi \in \{q,r,s\}} \left(\l \bm{Q},  \left(I_9\otimes\bm{D}_{-\xi}\right)\bm{F}_{\xi} \left(\bm{Q}) \r_{H}  + \l \bm{Q},  \mathbf{B}_\xi\left(\grad_{D_{+}}{\mathbf{Q}}\right)\right) \r_{H}\right) \\
       &=  \mathbb{I}_q\left(\bm{v}^T\bm{T}\right)\Big|_{q = 1} - \mathbb{I}_q\left(\bm{v}^T\bm{T}\right)\Big|_{q = 0} + \sum_{\xi \in \{q,r,s\}} \left( \l\bm{Q},  \bm{B}_{\xi} (\grad_{D_{+}} \bm{Q} ) \r_{H}- \l \left(I_9\otimes\bm{D}_{+\xi}\right) \bm{Q}, \bm{F}_{\xi} (\bm{Q}) \r_{H}\right) ,
    \end{align*}
    which from Lemma \ref{lem:anti_sym_prop_disc} gives,
     {\small
    \begin{align}
        \frac{d}{dt} \|\mathbf{Q}\left(\cdot, \cdot, \cdot, t\right)\|_{HP}^2 = \mathbb{I}_q\left(\bm{v}^T\bm{T}\right)\Big|_{q = 1} - \mathbb{I}_q\left(\bm{v}^T\bm{T}\right)\Big|_{q = 0}.
    \end{align}
    }
    Collecting contribution from both sides of the interface and considering boundary terms only at the interface gives the result.
    \end{proof}
\end{theorem}
\subsection{Numerical interface treatment}
In this section, we will discuss the numerical implementation of the nonlinear interface conditions \eqref{eq:transformed_interface_condition} using the  SAT penalty method. The procedure is similar to the method derived in \cite{DuruandDunham2016} using traditional SBP operators. Here we extend the procedure to the DP SBP framework \cite{DURU2022110966,williams2021provably}.  First, we will  construct interface data, $\widehat{v}_\eta, \widehat{T}_\eta$, as discussed in section \ref{sec:hat_variable} and derive the  SAT penalty terms by penalizing data, that is $\widehat{v}_\eta, \widehat{T}_\eta$, against the in-going waves only. The SAT terms are appended to the semi-discrete problem \eqref{eq:gen_hyp_transformed_discrete} with appropriate penalty parameters  such that a discrete energy estimate, similar to Theorem \ref{theo:boundedness_hat}, is derived.

Consider $\mathbf{Q} \in \{\mathbf{Q}^{-}, \mathbf{Q}^{+}\}$, the semi-discrete approximation of the elastic wave equation \eqref{eq:transformedEQ} with the weak enforcement of the interface conditions is
\begin{align}\label{eq:gen_hyp_transformed_discrete_SAT}
\widetilde{\mathbf{P}}^{-1} \frac{d }{d t} {\mathbf{Q}} = \grad_{D_{-}}\bullet {\mathbf{F} \left({\mathbf{Q}} \right)} + \sum_{\xi \in \{ q, r, s \} }\mathbf{B}_\xi\left(\grad_{D_{+}}{\mathbf{Q}}\right) + \mathbf{SAT}_{q}\left(\mathbf{v}-\widehat{\mathbf{v}}, \mathbf{T}-\widehat{\mathbf{T}}\right).
\end{align}
Here $\mathbf{SAT}_{q}$ is the SAT penalty term enforcing the fault interface conditions. 
Below, we will derive the explicit form of the SAT terms, $\mathbf{SAT}_{q}$. 
  
 Introduce the penalisation terms
    {
    \begin{align}\label{eq:penalty_terms}
  &{G}_\eta^+ = \frac{1}{2} {Z}_\eta^+ \left({v}_\eta ^+- \widehat{{v}}_\eta^+ \right)- \frac{1}{2}\left({T}_\eta^+  - \widehat{{T}}_\eta^+ \right),   \quad \widetilde{G}_\eta^+ \coloneqq \frac{1}{{Z}_\eta}{G}_\eta^+  ,
  \\
  \nonumber
  &{G}_\eta^-  = \frac{1}{2} {Z}_\eta^-\left({v}_\eta^-  - \widehat{{v}}_\eta^- \right)+ \frac{1}{2}\left({T}_\eta^-  - \widehat{{T}}_\eta^- \right), \quad \widetilde{G}_\eta^- \coloneqq \frac{1}{{Z}_\eta }{G}_\eta^- .
  \end{align}
  }
  %
  The penalty terms are computed in the transformed coordinates $l,m,n$, and will be rotated to the physical coordinates $x,y,z$. We  have
  \begin{align}\label{eq:rotate_back_forth}
  {\mathbf{G}} := \begin{pmatrix}
{G}_{x} \\
{G}_{y} \\
{G}_{z}
\end{pmatrix}
 = \mathbf{R}^T\begin{pmatrix}
{G}_{n} \\
{G}_{m} \\
{G}_{l}
\end{pmatrix},
\quad 
\widetilde{\mathbf{G}}:= 
  \begin{pmatrix}
\widetilde{G}_{x} \\
\widetilde{G}_{y} \\
\widetilde{G}_{z}
\end{pmatrix} = \mathbf{R}^T\begin{pmatrix}
\widetilde{G}_{n} \\
\widetilde{G}_{m} \\
\widetilde{G}_{l}
\end{pmatrix},
\end{align}
where $\mathbf{G} \in \{\mathbf{G}^-, \mathbf{G}^+\}$.
  Note that
\begin{equation}\label{eq:identity_pen}
\begin{split}
 \left(\left(\mathbf{v}^+\right)^T \mathbf{G}^+ - \left(\mathbf{T}^+\right)^T \widetilde{\mathbf{G}}^+ + \left(\mathbf{v}^+\right)^T\left(\mathbf{T}^+\right)\right)= 
 \sum_{\eta \in \{ l,m,n \} } \left(\frac{1}{Z_\eta^+  }|G_\eta^+ |^2  + \widehat{T}_\eta^+ \widehat{v}_\eta^+ \right), \\
 \left(\left(\mathbf{v}^-\right)^T \mathbf{G}^- + \left(\mathbf{T}^-\right)^T \widetilde{\mathbf{G}}^- - \left(\mathbf{v}^-\right)^T\left(\mathbf{T}^-\right)\right) = 
\sum_{\eta \in \{ l,m,n \} } \left(\frac{1}{Z_\eta^-  }|G_\eta^- |^2  - \widehat{T}_\eta^{-} \widehat{v}_\eta^{-} \right).
\end{split}
\end{equation}
%
We introduce the SAT vectors that match the eigen--structure of the elastic wave equation 
 {
 \begin{align}
\mathbf{SAT}^+ = 
         \begin{pmatrix}
            {G}_x^+\\
            {G}_y^+\\
            {G}_z^+\\
            -{n_x}\widetilde{{G}}_x^+, \\
            -{n_y}\widetilde{{G}}_y^+ \\
            -{n_z}\widetilde{{G}}_z^+, \\
            -\left({n_y}\widetilde{{G}}_x^+ + {n_x}\widetilde{{G}}_y^+\right)\\
            -\left({n_z}\widetilde{{G}}_x^+ + {n_x}\widetilde{{G}}_z^+\right)\\ -\left({n_z}\widetilde{{G}}_y^+ + {n_y}\widetilde{{G}}_z^+\right)\\
         \end{pmatrix},
         \quad
\mathbf{SAT}^{-} = 
         \begin{pmatrix}
            {G}_x^{-}\\
            {G}_y^{-}\\
            {G}_z^{-}\\
            {n_x}\widetilde{{G}}_x^{-}, \\
            {n_y}\widetilde{{G}}_y^{-} \\
            {n_z}\widetilde{{G}}_z^{-}, \\
            \left({n_y}\widetilde{{G}}_x^{-} + {n_x}\widetilde{{G}}_y^{-}\right)\\
            \left({n_z}\widetilde{{G}}_x^{-} + {n_x}\widetilde{{G}}_z^{-}\right)\\ \left({n_z}\widetilde{{G}}_y^{-} + {n_y}\widetilde{{G}}_z^{-}\right)\\
         \end{pmatrix}.
\end{align}
      }
  Here, $\mathbf{n} = (n_x, n_y, n_z)^T$ is the positively pointing  unit normal vector at the fault interface, defined in \eqref{eq:normal_vector}.
  Note also that
  \begin{equation}\label{eq:scalar_product_flux}
  \left(\mathbf{Q}^+\right)^T\mathbf{SAT}^{+} = \left(\mathbf{v}^+\right)^T \mathbf{G}^{+} - \left(\mathbf{T}^+\right)^T \widetilde{\mathbf{G}}^{+} , \quad \left(\mathbf{Q}^-\right)^T\mathbf{SAT}^{-} = \left(\mathbf{v}^-\right)^T \mathbf{G}^{-} + \left(\mathbf{T}^-\right)^T \widetilde{\mathbf{G}}^{-}.
  \end{equation}
Finally, the SAT terms are defined as follows
\begin{align}\label{eq:SAT_General}
\mathbf{SAT}_{q}^{-} = -\bm{H}_{\xi}^{-1}\mathbf{e}_{q,n_q}\bm{J}\sqrt{\bm{\xi}_x^2 + \bm{\xi}_y^2 + \bm{\xi}_z^2}\mathbf{SAT}^{-}, 
\quad
\mathbf{SAT}_{q}^{+} = -\bm{H}_{\xi}^{-1}\mathbf{e}_{q,0}\bm{J}\sqrt{\bm{\xi}_x^2 + \bm{\xi}_y^2 + \bm{\xi}_z^2}\mathbf{SAT}^{+}.
\end{align}
Introduce the fluctuation term
\begin{align}\label{eq:fluctuation_term}
   {F}_{luc}\left({\bm{G}},\mathbf{Z}\right) \coloneqq  - \mathbb{I}_{q}\left(\sum_{\eta = l,m,n} \frac{1}{Z_\eta  }|G_\eta |^2\right) \le 0, 
\end{align}
and discrete surface interface term
    \begin{align}\label{eq:disc_boundary_term_hat}
   \mathbb{I}\mathbb{T}_s=\mathbb{I}_q\left((\widehat{\mathbf{v}}^-)^T\widehat{\mathbf{T}}^--(\widehat{\mathbf{v}}^+)^T\widehat{\mathbf{T}}^+\right)=  -\mathbb{I}_{q}\left(\sum_{j \in \{l, m\} }\frac{\widehat{\alpha}}{ \left(\eta_j +\widehat{\alpha}\right)^2}\Phi_j^2\right) 
   =
   -\mathbb{I}_{q}\left( \widehat{\alpha}|\widehat{V}|^2 \right) \le 0,
\end{align}
where the surface cubature $\mathbb{I}_{q}$ is defined in \eqref{eq:surface_cubature_1D}. 

We can prove the result
    \begin{theorem}\label{theo:discrete_boundedness}
    Consider the semi-discrete approximation \eqref{eq:gen_hyp_transformed_discrete_SAT} of the elastic wave equation with the SAT-terms  $\mathbf{SAT}_{q}$ defined in \eqref{eq:SAT_General}.
    We have
    \begin{align*}
        \frac{d}{dt} \left(\|\mathbf{Q}^-\left(\cdot, \cdot, \cdot, t\right)\|_{HP}^2 + \|\mathbf{Q}^+\left(\cdot, \cdot, \cdot, t\right)\|_{HP}^2\right) = {F}_{luc}\left({\bm{G}}^-,\mathbf{Z}^-\right) + {F}_{luc}\left({\bm{G}}^+,\mathbf{Z}^+\right) + \mathbb{I}\mathbb{T}_s  \le 0.
    \end{align*}
    \end{theorem}
    \begin{proof}
   For $\mathbf{Q} \in \{\mathbf{Q}^{-}, \mathbf{Q}^{+}\}$ consider
    {\small
    \begin{align*}
        \frac{d}{dt} \|\mathbf{Q}\left(\cdot, \cdot, \cdot, t\right)\|_{HP}^2 = \l \bm{Q} , P^{-1} \pf{}{t}  \bm{Q} \r_{H}  &= \l \bm{Q}, \grad_{D_{-}}\bullet {\mathbf{F} \left({\mathbf{Q}} \right)} + \sum_{\xi= q, r, s}\mathbf{B}_\xi\left(\grad_{D_{+}}{\mathbf{Q}}\right) \r_{H} \\
        &+  \l \bm{Q},\mathbf{SAT}_{q}\left({\mathbf{Q}}\right)\r_{H}.
    \end{align*}
    }
    By Theorem \ref{theo_sbp_sem_discrete} we have
    \begin{align}
        \frac{d}{dt} \|\mathbf{Q}\left(\cdot, \cdot, \cdot, t\right)\|_{HP}^2  = \mathbb{I}_q \left({\mathbf{v}}^T {\mathbf{T}}\right)\Big|_{q = 1} - \mathbb{I}_q \left({\mathbf{v}}^T {\mathbf{T}}\right)\Big|_{q = 0} +   \l \bm{Q},\mathbf{SAT}_{q}\left({\mathbf{Q}}\right)\r_{H},
    \end{align}
    with
    $$
    \l \bm{Q}, \mathbf{SAT}_{q}^+ \r_{H} = -\mathbb{I}_{q}\left(\mathbf{v}^T \mathbf{G}^+ - \mathbf{T}^T \widetilde{\mathbf{G}}^+\right), \quad \l \bm{Q}, \mathbf{SAT}_{q}^- \r_{H} = -\mathbb{I}_{q}\left(\mathbf{v}^T \mathbf{G}^- + \mathbf{T}^T \widetilde{\mathbf{G}}^-\right).
    $$
    By collecting contributions from both sides of the interface and using the identities \eqref{eq:identity_pen} and \eqref{eq:disc_boundary_term_hat} we have
    \begin{equation}
    \begin{split}
        &\frac{d}{dt} \left(\|\mathbf{Q}^-\left(\cdot, \cdot, \cdot, t\right)\|_{HP}^2 + \|\mathbf{Q}^+\left(\cdot, \cdot, \cdot, t\right)\|_{HP}^2\right) = \\
        &-\mathbb{I}_{q}\left(\left(\mathbf{v}^{+}\right)^T \mathbf{G}^+ - \left(\mathbf{T}^{+}\right)^T \widetilde{\mathbf{G}}^+ + \left(\mathbf{v}^{+}\right)^T\mathbf{T}^{+}\right) 
        -
        \mathbb{I}_{q}\left(\left(\mathbf{v}^{-}\right)^T \mathbf{G}^- + \left(\mathbf{T}^{-}\right)^T \widetilde{\mathbf{G}}^- - \left(\mathbf{v}^{-}\right)^T\mathbf{T}^-\right)\\
        &= -\mathbb{I}_{q}\left(\sum_{\eta \in \{ l,m,n \} } \frac{1}{Z_\eta^+  }|G_\eta^+ |^2 + \left(\widehat{\mathbf{v}}^+\right)^T\widehat{\mathbf{T}}^+ \right) 
        -
        \mathbb{I}_{q}\left(\sum_{\eta \in \{ l,m,n \} } \frac{1}{Z_\eta^-  }|G_\eta^- |^2-\left(\widehat{\mathbf{v}}^-\right)^T\widehat{\mathbf{T}}^- \right)\\
        & =  {F}_{luc}\left({\bm{G}}^-,\mathbf{Z}^-\right) + {F}_{luc}\left({\bm{G}}^+,\mathbf{Z}^+\right) +  \mathbb{I}\mathbb{T}_s   \le 0.
        \end{split}
    \end{equation}
    The proof is complete. 
    \end{proof}
The fluctuation term ${F}_{luc}\left({\bm{G}},\mathbf{Z}\right) \le 0 $ adds a little numerical dissipation on the interface. However, in the limit of mesh refinement the fluctuation term vanishes, that is ${F}_{luc}\left({\bm{G}},\mathbf{Z}\right) \to 0^+$ as $h \to 0^+$, and we have  $\mathbb{I}\mathbb{T}_s \to \mathrm{IT}_s$. 
The discrete result, Theorem \ref{theo:discrete_boundedness}, is completely analogous to the continuous counterparts, Theorems \ref{theo:boundedness} and \ref{theo:boundedness_hat}.

\section{Analysis of numerical errors}
Here, we will discuss numerical errors for the semi-discrete approximation \eqref{eq:gen_hyp_transformed_discrete_SAT}. Numerical error analysis for the discrete approximation of the 3D elastic wave equation subject to the nonlinear friction laws  \eqref{eq:Force_balance}--\eqref{eq:Friction_law} at the fault/interface, is a nontrivial task. %
This is mostly due to the fact that the well-posedness of the continuous IBVP  problem is not completely understood in general settings. To succeed we will make some simplifying assumptions. However our discussions are plausible and will help explain some of the behavior of the error seen in the numerical simulations performed later in this study.  
We will derive an error estimate for a linearised friction law. Next we will discuss the generation and propagation of  errors due to  DP SBP numerical differentiation of discontinuous and nearly singular exact solutions.
\subsection{Error estimate}
Let $\boldsymbol{\mathcal{Q}} \in \{\boldsymbol{\mathcal{Q}}^{-}, \boldsymbol{\mathcal{Q}}^{+}\}$ denote the exact solutions of the IBVP, and $\boldsymbol{\mathcal{Q}}(q_i, r_j, s_k, t)$ denote the restriction of the solution on the grid $(q_i, r_j, s_k)$. We introduce the error on the grid
$$
\boldsymbol{\mathcal{E}}_{ijk}(t) = \mathbf{Q}_{ijk}(t)-\boldsymbol{\mathcal{Q}}(q_i, r_j, s_k, t).
$$
For the velocity and traction $\left({\mathbf{v}}, \mathbf{T}\right)$ and their hat-variable $\left(\widehat{\mathbf{v}}, \widehat{\mathbf{T}}\right)$  at the fault interface, we denote the errors by
$
\left({\mathbf{e}}_v, {\mathbf{e}}_T\right), 
$
$
\left(\widehat{\mathbf{e}}_v, \widehat{\mathbf{e}}_T\right).
$
The error $\boldsymbol{\mathcal{E}}$ satisfies the error equation
\begin{align}\label{eq:gen_hyp_transformed_discrete_SAT_error}
\widetilde{\mathbf{P}}^{-1} \frac{d }{d t} \boldsymbol{\mathcal{E}} = \grad_{D_{-}}\bullet {\mathbf{F} \left(\boldsymbol{\mathcal{E}} \right)} + \sum_{\xi \in \{ q, r, s \} }\mathbf{B}_\xi\left(\grad_{D_{+}}\boldsymbol{\mathcal{E}}\right) + \mathbf{SAT}_{q}\left({\mathbf{e}}_v-\widehat{\mathbf{e}}_v, {\mathbf{e}}_T-\widehat{\mathbf{e}}_T\right) + \mathbb{T},
\end{align}
where $\mathbb{T}$ is the truncation error of the SBP FD operator. The truncation error $\mathbb{T}$ is a 3D vector field, however its structure  is similar in each spatial direction. For sufficiently smooth exact solutions $\mathcal{Q}$ sampled on the grid points $\xi_j = j h_{\xi}$ in the spatial direction $\xi \in \{q, r, s\}$ the truncation error is of the form
\begin{equation}\label{eq:truncation_error}
\begin{split}
& \mathbb{T}_{\xi,j}  = \left \{
\begin{array}{rl}
 h_{\xi}^{\gamma} \beta_j\frac{\partial^{\gamma+1} \mathcal{Q}}{\partial \xi^{\gamma +1}}\Big|_{\xi_j},  & \text{if boundary}  ,\\
 h_{\xi}^{\nu} \beta_j\frac{\partial^{\nu+1} \mathcal{Q}}{\partial \xi^{\nu+1}}\Big|_{\xi_j},
 & \text{if interior}.
\end{array} \right\} ,
\end{split}
\end{equation}
where $\mathbb{T}_{ijk} = \mathbb{T}_{q,i} + \mathbb{T}_{r,j} + \mathbb{T}_{s,k}$.
Here, $\beta_j$, are mesh independent constants,  $h_{\xi}>0$ is the grid spacing,  $\gamma \in \{1, 2, \cdots \}$ is the order of accuracy of the FD stencils close to the boundary and $\nu \in \{1, 2, \cdots \}$ is the order of accuracy of the SBP FD stencils in the interior, away from the boundaries.
For traditional SBP operators based on central difference stencils the interior accuracy is always even, and  we have $(\gamma, \nu) = (p, 2p)$, for $p = 1, 2, 3, \cdots$. DP SBP operators based on skewed stencils, $D_{+}$ and  $D_{-}$, the interior order of accuracy can be odd or even. As discussed in  \cite{DURU2022110966,Mattsson2017,williams2021provably}, DP SBP operators with even order $\left(2p\right)$-th accuracy in the interior are closed with $p$-th order accurate stencils close to boundaries, and we also have $(\gamma, \nu) = (p, 2p)$. Dual-paring SBP FD operators with odd order $\left(2p+1\right)$-th accuracy in the interior are closed with $p$-th order accurate stencils close to boundaries, giving $(\gamma, \nu) = (p, 2p+1)$. 
The traditional SBP operators and DP SBP operators can yield $\left(p+1\right)$-th global order of accuracy, for smooth solutions.

Note that if the hat-variables $\left(\widehat{\mathbf{e}}_v, \widehat{\mathbf{e}}_T\right)$ for the error   satisfy the algebraic identities \eqref{eq:identity_1}--\eqref{eq:identity_4}, then using Theorem \ref{theo:discrete_boundedness} we can derive an estimate for the error.
For the general  friction law \eqref{eq:Friction_law_0} with the nonlinear frictional strength $\alpha(\sigma_n, V, \psi)\ge 0$, it is difficult to show that the hat-variable $\left(\widehat{\mathbf{e}}_v, \widehat{\mathbf{e}}_T\right)$ satisfy the necessary algebraic identities \eqref{eq:identity_1}--\eqref{eq:identity_4} required for stability. 

The nonlinear frictional strength parameter is never negative, that is 
$$
\alpha(\sigma_n, V, \psi)\ge 0,\quad \text{for  all} \quad \sigma_n\ge0, V\ge0, \quad \psi\in \mathbb{R}.
$$ 
To succeed, we linearise the friction law and consider the frozen coefficient case 
$$\alpha(\sigma_n, V, \psi) \to \alpha = const. \ge 0, \quad \text{with}
\quad
0\le \alpha \le \infty.$$
With this linearisation we minimise some of the nonlinear effects. However, the solutions can be discontinuous and nearly singular at the interface, and   the hat-variables for the error $\left(\widehat{\mathbf{e}}_v, \widehat{\mathbf{e}}_T\right)$  satisfy the identities \eqref{eq:identity_1}--\eqref{eq:identity_4}.

We can prove the result for the linearised friction law.
    \begin{theorem}\label{theo:error_estimate}
    Consider the semi-discrete error equation \eqref{eq:gen_hyp_transformed_discrete_SAT_error}, with the numerical error $\boldsymbol{\mathcal{E}} = \left(\boldsymbol{\mathcal{E}}^-, \boldsymbol{\mathcal{E}}^+\right)$  and the truncation error $\boldsymbol{\mathbb{T}} = \left(\boldsymbol{\mathbb{T}}^-, \boldsymbol{\mathbb{T}}^+\right)$ of the elastic wave equation with the SAT-terms  $\mathbf{SAT}_{q}$ defined in \eqref{eq:SAT_General}.
   Let $\alpha(\sigma_n, V, \psi) \to \alpha\ge 0$, where $\alpha\ge0$ is a real constant, we have
    \begin{align*}
        \frac{d}{dt} \|\boldsymbol{\mathcal{E}}\left(t\right)\|_{HP}^2 + \mathbb{B}\mathbb{T}_s
        \le
        \|\boldsymbol{\mathcal{E}}\left(t\right)\|_{HP}\|\boldsymbol{\mathbb{T}}\left(t\right)\|_{HP^{-1}},
    \end{align*}
    where
    $\mathbb{B}\mathbb{T}_s
        =
        -\left({F}_{luc}\left({\bm{G}}^-,\mathbf{Z}^-\right) + {F}_{luc}\left({\bm{G}}^+,\mathbf{Z}^+\right) + \mathbb{I}\mathbb{T}_s\right) \ge 0,$ 
$
\mathbb{I}\mathbb{T}_s=-\mathbb{I}_{q}\left(\widehat{\alpha}|\widehat{V}_e|^2 \right) \le 0,
$ 
and

$
|\widehat{V}_e|^2 = \lJump{\widehat{e}_{vl} \rJump}^2 + \lJump{\widehat{e}_{vm} \rJump}^2.
$
\end{theorem}
By the frozen coefficient analysis above,  Theorem \ref{theo:error_estimate} indicates that the numerical error $\boldsymbol{\mathcal{E}}$ is bounded by the truncation error $\boldsymbol{\mathbb{T}}$, and will converge to zero if $\boldsymbol{\mathbb{T}}$ is square integrable. Obviously, the frozen coefficient analysis above hold for both traditional and DP SBP operators. Note however that the truncation errors  in \eqref{eq:truncation_error} include higher derivatives of the exact solution $\boldsymbol{\mathcal{Q}}$.  If $\boldsymbol{\mathcal{Q}}$ is sufficiently smooth such that the highest derivatives in \eqref{eq:truncation_error} are continuous then the numerical error will converge to zero optimally, $\boldsymbol{\mathcal{E}} = O(h^{\gamma+1})$.
For traditional SBP operators, the convergence rate $O(h^{\gamma+1})$ has been  verified numerically in \cite{DuruandDunham2016} using the method of manufacture solution, that is by forcing the IBVP with smooth exact solutions.  We believe that the numerical result for smooth solutions  hold also for DP SBP operators for friction laws.

For spontaneously self-evolving shear ruptures, the exact solutions are discontinuous across the fault for the particle velocity, and  for the stress fields continuous across the fault but nearly singular. The solutions $\boldsymbol{\mathcal{Q}}$ are rough, and the highest derivatives in \eqref{eq:truncation_error} may not exist. As we will discuss below the  smoothness of the solution and the parity of the order of accuracy  ($\nu$) of the interior stencil of the SBP FD operators will have a significant impact in the generation and propagation of numerical errors.

\subsection{Parity}\label{sec:parity}
Numerical errors for wave dominated problems  are often most prominent at high frequencies. For well-posed IBVPs with smooth solutions high frequency wave modes can be resolved by increasing mesh resolution. However, for discontinuous and/or (nearly) singular exact solutions numerical errors arising from high order accurate methods may not diminish with increasing mesh resolution, because of "Gibbs/Runge's phenomena". For these situations carefully designed numerical methods are necessary in order to compute accurate numerical solutions. Consequently, detailed theoretical analysis is needed in order to  understand the behaviours of numerical errors and control them.

To simplify the discussion here we will consider only the interior stencil, with the order of accuracy $\nu$.  Numerical errors introduced by numerical differentiation of a wave-like solution can be split into the {\it amplitude error} and the {\it phase error}. If the {\it amplitude error} is much larger than the {\it phase error} we say that the numerical error is {\it amplitude-dominated}. Similarly, if the {\it phase error} is much larger than the {\it amplitude error} we say that the numerical error is {\it phase-dominated}. 

We will show that the parity of the order of accuracy $\nu$ will determine whether the numerical error is {\it amplitude-dominated} or {\it phase-dominated}. And if the numerical error is {\it amplitude-dominated}, then the numerical method has the potential to amplify high frequency numerical artefacts for \emph{non-smooth solutions}.

Let $\widetilde{\mathbf{Q}}(t) = [\boldsymbol{\mathcal{Q}}(q_i, r_j, s_k, t)]$ denote the grid function restricting the exact solution on the grid. We have
$$
\bm{D}_{\eta \xi}\widetilde{\mathbf{Q}}\Big|_{j} =  \frac{\partial \mathcal{Q}}{\partial \xi}\Big|_{\xi_j} + h_{\xi}^{\nu} \beta_j\frac{\partial^{\nu+1} \mathcal{Q}}{\partial \xi^{\nu+1}}\Big|_{\xi_j},
$$
where $\bm{D}_{\eta \xi}$ with $\eta \in \{-, +\}$ are the DP SBP operators defined above with the interior order of accuracy $\nu \in \{2p, 2p+1\}$, $p = 1, 2, \cdots $.
The numerical approximation replaces the partial derivative ${\partial}/{\partial \xi}$ by the modified partial differential operator 
$$
\frac{\partial }{\partial \xi} \to \frac{\partial }{\partial \xi} + h_{\xi}^{\nu} \beta_j\frac{\partial^{\nu+1} }{\partial \xi^{\nu+1}}.
$$
To determine whether the numerical error is {\it amplitude-dominated} or {\it phase-dominated} we  consider the 1D monochromatic plane wave solution with unit amplitude $e^{ik_\xi \xi}$, $k_\xi \in \mathbb{R}$, $i = \sqrt{-1}$. The solution will be modified by the numerical approximation as follows
$$
e^{ik_\xi \xi} \to  e^{h_{\xi}^{\nu} \beta_j\left(ik_\xi\right)^{\nu+1} \xi} \cdot e^{ik_\xi \xi}=\left(1 + O(h_{\xi}^{\nu})\right)e^{ik_\xi \xi}.
$$ 
For any finite $k_\xi < \infty$ the modified solution will converge to the monochromatic plane wave solution as $h_\xi \to 0$. Obviously for any $h_\xi > 0$ there are numerical errors. For even order of accuracy $\nu = 2p$ we have
$$
e^{ik_\xi \xi} \to e^{ik_\xi\left(1 + (-1)^p  \beta_j\left(h_{\xi}k_\xi\right)^{\nu} \right)\xi},
$$
which is a phase error.
For odd order of accuracy $\nu = 2p+1$ we have
$$
e^{ik_\xi \xi} \to e^{\left( (-1)^{p+1}  k_\xi \beta_j\left(h_{\xi}k_\xi\right)^{\nu}\right) \xi} \cdot  e^{ik_\xi \xi},
$$
which gives an amplitude error.

Therefore, even order ($\nu = 2p$) accurate numerical operators generate {\it phase-dominated} errors while numerical errors generated by odd order ($\nu = 2p+1$) accurate  operators are {\it amplitude-dominated}. As we will see in the numerical experiments below, upwind SBP operators with odd order accurate interior stencils can amplify high frequency numerical artefacts for spontaneously propagating shear ruptures in 3D elastic solids.

\subsection{$\alpha$-DRP SBP operators}
Using the DP framework \cite{DURU2022110966,williams2021provably,CWilliams2021}, it is possible to design optimised finite difference stencils such that numerical dispersion errors are significantly diminished for all wave numbers. This was the approach taken by the authors in \cite{williams2021provably} where $\alpha$-DRP SBP operators are derived. These DP SBP operators are designed to guarantee maximum relative error $\le \alpha$ for all wave numbers. In particular  the $\alpha$-DRP SBP operators can resolve the highest frequency ($\pi$-mode) present on any equidistant grid at a tolerance of $\alpha = 5\%$ maximum error. For explicit schemes, these operators may provide a better resolution than any volume discretisation available today, including spectral methods, and significantly improves on the current standard for traditional operators that have a tolerance of $100 \%$ maximum error.  

In Figure \ref{fig:drpl2}, we display the numerical dispersion relation for the 1D wave equation,
\begin{align}\label{eq:1D_elastic_wave_equation}
 \frac{\partial v}{\partial t} = \frac{\partial \sigma}{\partial x}, 
    \qquad
\frac{\partial \sigma}{\partial t} = \frac{\partial v}{\partial x},
\end{align}
which has the linear analytical dispersion relation $\omega = k$. Here, $\omega$ is the angular frequency and $k\in \mathbb{R}$ is the spatial wave number. 
\begin{table}[h!]
    \centering
    \caption{$L^2$ relative dispersion errors for various SBP FD operators.}
    \begin{tabular}{c||c|c|c|c}
     & order 4 &  order 5 &  order 6 & order 7 \\
     \hline
     SBP  &  $64.35 \%$ & -   & $58.9 \%$ & - \\
     DP   &  $7.69 \%$ & $43.86 \%$ & $17.54 \%$ & $44.16 \%$ \\
     $\alpha$-DRP  &  $1.91 \%$ & $1.72 \%$ &  $1.36 \%$ & $1.28 \%$ \\
    \end{tabular}
    \label{tab:drpl2}
\end{table}
\begin{figure}[h!]
\centering
  \noindent
  \makebox[\textwidth]{\includegraphics[width =0.75\textwidth]{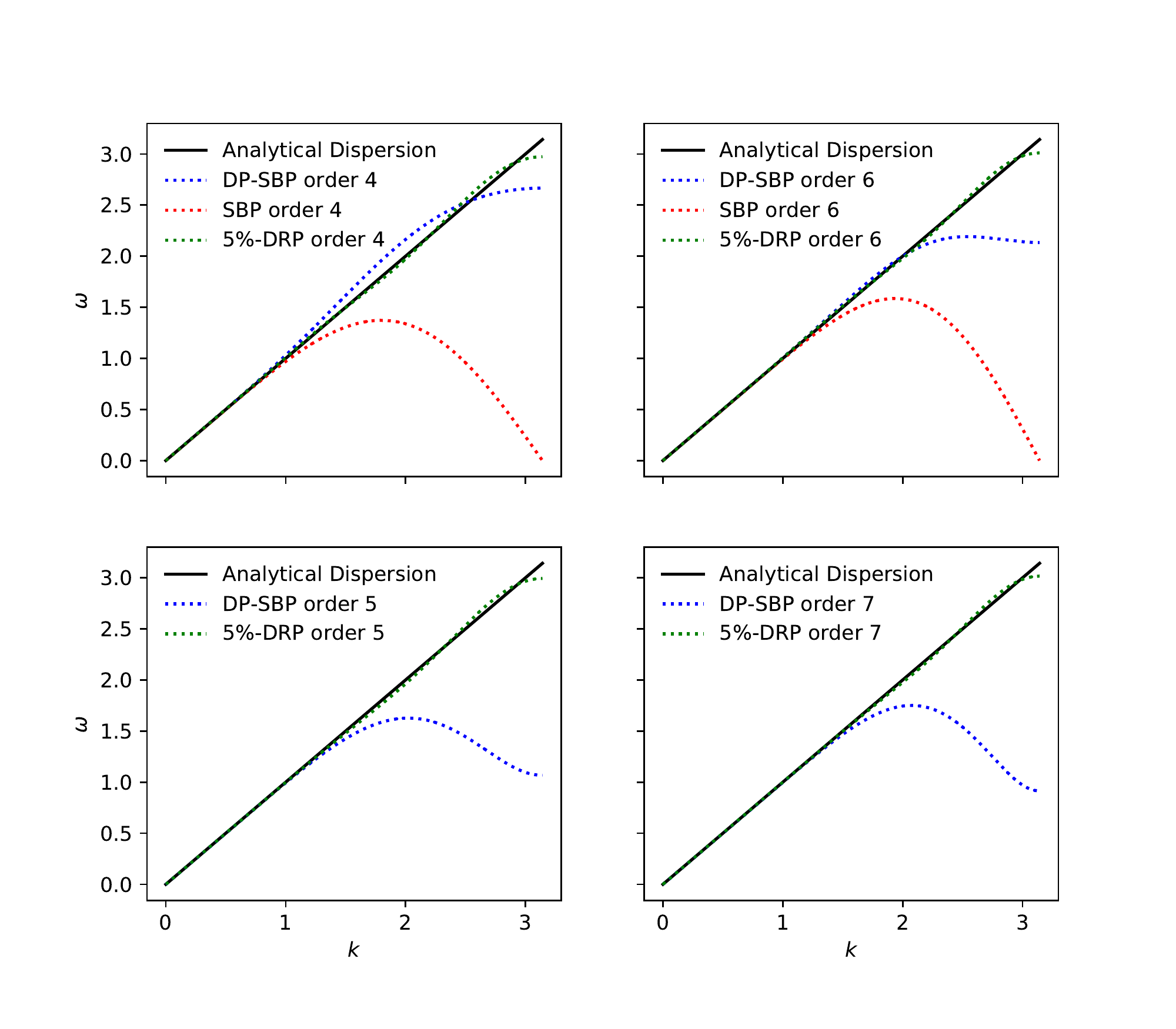}}%
  \caption{Dispersion relations for the 1D wave equation. Here we compare the traditional operators (SBP), the dual-paring operators (DP-SBP) and the $\alpha$-DRP DP operators for 4th, 5th, 6th and 7th order of accuracy.}
  \label{fig:drpl2}
\end{figure}
In Figure \ref{fig:drpl2}, we compare the traditional SBP operators, the DP SBP operators and the $\alpha$-DRP SBP operators. Note in particular, for the $\alpha$-DRP SBP operators the numerical dispersion relations fit  the continuous dispersion relation for the 1D wave equation excellently well for all wave numbers. In Table \ref{tab:drpl2} we tabulate the relative $L^2$ dispersion errors.
For more details we refer the reader to \cite{williams2021provably} where the $\alpha$-DRP SBP operators are derived and analysed. 
As will be shown in the next section, using the $\alpha$-DRP SBP operators we can eliminate the fatal errors arising from the upwind DP SBP operators \cite{Mattsson2017} for dynamic rupture simulations with non-smooth solutions.

    \section{Numerical simulations of dynamic earthquake ruptures in 3D elastic solids}
    In this section, we will perform numerical simulations to verify the analysis performed in the previous sections.
    To highlight the numerical properties of the traditional, upwind DP and DRP SBP operators, we will select specifically the dipping fault benchmark problem, TPV10, from the series of dynamic rupture benchmark problems proposed by The SCEC/USGS Spontaneous Rupture Code Verification Project \cite{SCEC2022}. We note that other benchmark problems with different levels of difficulties have also been considered, although these results are not reported here,  they can be found in the thesis of one of the co-authors  \cite{CWilliams2021}.  
    
    The numerical methods are efficiently implemented in WaveQLab3D \cite{DuruandDunham2016,DURU2022110966}, a petascale finite difference solver.
The solutions are integrated in time using an explicit $5$-stage $4$-th order accurate low-storage Runge-Kutta time stepping scheme \cite{CarpenterKennedy1994}. 
The explicit time-step $dt$ is determined by
\begin{align}
    dt = \mathrm{ CFL} \min_{\xi \in \{q, r, s\} }\sqrt{  \frac{h_{\xi}^2}{\left(c_p^2+c_s^2\right)\left(\xi_x^2+\xi_y^2+\xi_z^2 \right)}}, \quad CFL = 0.5. 
\end{align}
Here $c_p, c_s$ are the P-wave and S-wave speeds,  $\xi_x, \xi_y, \xi_z$ are the partial derivatives of $\xi \in \{ q, r, s \}$ with respect to $x,y,z$ respectively, $h_{\xi} = 1/n_\xi$ is the uniform spatial step in the transformed coordinates. 

We will compare and contrast the numerical resolution power of the traditional SBP FD operator,  upwind DP SBP operator and the  recently derived DRP SBP operator for nonlinear dynamic earthquake ruptures composed of discontinuous and nearly singular exact solutions. We note that our numerical solutions have been benchmarked against finite element and spectral element methods downloaded from \cite{SCEC2022}. Please see \cite{CWilliams2021} for details.
    These problems are computationally expensive, we will also demonstrate parallel efficiency and perfect scaling of our parallel implementation and numerical simulation on Gadi
    \footnote{Gadi is the current supercomputer hosted at the National Computational Infrastructure, Australia.  It contains total 3,074 nodes each containing two 24-core Intel Xeon Scalable 'Cascade Lake' processors and 192 Gigabytes of memory. The internal connection is through Mellanox Technologies' InfiniBand technology. The code is compiled on Gadi by Intel Compiler/2019.5.281 and the parallelisation is realised by Intel MPI/2019.5.281.}.
     For the production runs we will use two levels of mesh sizes, $h = 100, 50$ m, with 480 and 2304 MPI-ranks respectively on Gadi. We will briefly describe the TPV10 benchmark problem below, more elaborate descriptions can be found in \cite{SCEC2022}.

    \subsection{TPV10}
    The TPV10 benchmark is a \textbf{60-degree dipping} normal  fault embedded in a homogeneous half-space of isotropic elastic solid, see Figure \eqref{fig:TPV10}. Therefore the fault region is the half-plane cutting through $x=\widetilde{x}(y):= x_0 + y/\tan(60)$, that is ${\Gamma}_F = \{\widetilde{x}(y)\} \times (0, \infty) \times (-\infty, \infty)$.
    We introduce the down-dip direction  $\widetilde{y} = y/\sin(60)$, thus the fault plane is ${\Gamma}_F = \{\left(\widetilde{y}(y), z\right) | (y, z) \in (0, \infty) \times (-\infty, \infty)$.
    The rupturing part of the fault is a $15$ km-by-$30$ km rectangular surface,  ${\Gamma}_r = [0,15] \times [-15,15]$, on the fault plane. Outside ${\Gamma}_r$, the fault is locked and it is not permitted to slip. The $3$ km-by-$3$ km nucleation patch is centered at $12$ km down-dip that is $\Gamma_{\text{nuc}} = [10.5, 13.5] \times [-1.5,1.5]$. 
    
    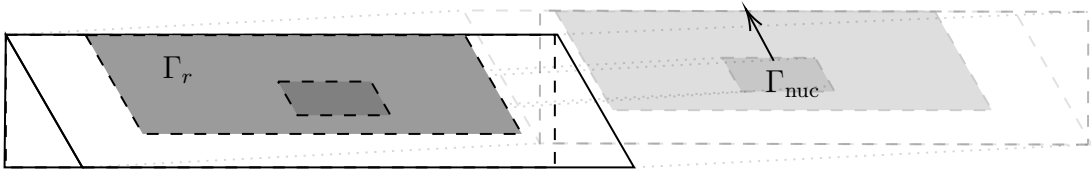
\begin{figure}[h!]
        \centering
        \tikzset{every picture/.style={line width=0.75pt}} 
        
        \begin{tikzpicture}[x=0.75pt,y=0.75pt,yscale=-1,xscale=1]
        
        \draw  [color={rgb, 255:red, 155; green, 155; blue, 155 }  ,draw opacity=0.3 ][dash pattern={on 4.5pt off 4.5pt}] (595,104.5) -- (316.77,103.87) -- (278.24,36.78) -- (556.47,37.41) -- cycle ;
        \draw  [color={rgb, 255:red, 0; green, 0; blue, 0 }  ,draw opacity=0.2 ][fill={rgb, 255:red, 155; green, 155; blue, 155 }  ,fill opacity=0.33 ][dash pattern={on 4.5pt off 4.5pt}] (324.73,37.26) -- (515.97,37.26) -- (544.75,87.11) -- (353.5,87.11) -- cycle ;
        \draw  [color={rgb, 255:red, 0; green, 0; blue, 0 }  ,draw opacity=0.2 ][fill={rgb, 255:red, 155; green, 155; blue, 155 }  ,fill opacity=0.33 ][dash pattern={on 4.5pt off 4.5pt}] (408.72,60.68) -- (455.65,60.68) -- (465.47,77.68) -- (418.53,77.68) -- cycle ;
        \draw  [color={rgb, 255:red, 155; green, 155; blue, 155 }  ,draw opacity=0.73 ][dash pattern={on 4.5pt off 4.5pt}] (593.24,104.49) -- (316.77,103.87) -- (316.92,36.87) -- (593.39,37.49) -- cycle ;
        \draw [color={rgb, 255:red, 0; green, 0; blue, 0 }  ,draw opacity=0.16 ] [dash pattern={on 0.84pt off 2.51pt}]  (194.53,89.68) -- (418.53,77.68) ;
        \draw [color={rgb, 255:red, 0; green, 0; blue, 0 }  ,draw opacity=0.16 ] [dash pattern={on 0.84pt off 2.51pt}]  (231.65,72.68) -- (455.65,60.68) ;
        \draw [color={rgb, 255:red, 0; green, 0; blue, 0 }  ,draw opacity=0.16 ] [dash pattern={on 0.84pt off 2.51pt}]  (241.47,89.68) -- (461,75.13) ;
        \draw [color={rgb, 255:red, 0; green, 0; blue, 0 }  ,draw opacity=0.16 ] [dash pattern={on 0.84pt off 2.51pt}]  (324.36,49) -- (593.39,37.49) ;
        \draw [color={rgb, 255:red, 0; green, 0; blue, 0 }  ,draw opacity=0.16 ] [dash pattern={on 0.84pt off 2.51pt}]  (184.72,72.68) -- (408.72,60.68) ;
        \draw [color={rgb, 255:red, 0; green, 0; blue, 0 }  ,draw opacity=1 ]   (434.74,62.19) -- (421.59,37.77) ;
        \draw [shift={(420.65,36)}, rotate = 61.71] [color={rgb, 255:red, 0; green, 0; blue, 0 }  ,draw opacity=1 ][line width=0.75]    (10.93,-3.29) .. controls (6.95,-1.4) and (3.31,-0.3) .. (0,0) .. controls (3.31,0.3) and (6.95,1.4) .. (10.93,3.29)   ;
        \draw   (47.47,49.1) -- (86.09,116) -- (46.82,116) -- cycle ;
        \draw  [dash pattern={on 4.5pt off 4.5pt}] (47.47,49.1) -- (324.36,49.1) -- (324.36,116) -- (47.47,116) -- cycle ;
        \draw   (47.41,49) -- (325.64,49) -- (364.32,116) -- (86.09,116) -- cycle ;
        \draw  [fill={rgb, 255:red, 155; green, 155; blue, 155 }  ,fill opacity=1 ][dash pattern={on 4.5pt off 4.5pt}] (87.73,49.26) -- (278.97,49.26) -- (307.75,99.11) -- (116.5,99.11) -- cycle ;
        \draw  [fill={rgb, 255:red, 128; green, 128; blue, 128 }  ,fill opacity=1 ][dash pattern={on 4.5pt off 4.5pt}] (184.72,72.68) -- (231.65,72.68) -- (241.47,89.68) -- (194.53,89.68) -- cycle ;
        \draw [color={rgb, 255:red, 0; green, 0; blue, 0 }  ,draw opacity=0.16 ] [dash pattern={on 0.84pt off 2.51pt}]  (86.09,116) -- (316.77,103.87) ;
        \draw [color={rgb, 255:red, 0; green, 0; blue, 0 }  ,draw opacity=0.16 ] [dash pattern={on 0.84pt off 2.51pt}]  (364.32,116) -- (593.24,104.49) ;
        \draw [color={rgb, 255:red, 0; green, 0; blue, 0 }  ,draw opacity=0.16 ] [dash pattern={on 0.84pt off 2.51pt}]  (47.47,49.1) -- (278.24,36.78) ;
        
        \draw (125,58.53) node [anchor=north west][inner sep=0.75pt]    {$\Gamma _{r}$};
        \draw (429,65.53) node [anchor=north west][inner sep=0.75pt]    {$\Gamma _{\text{nuc}}$};
        
        \end{tikzpicture}
        \caption{A \textbf{60-degree dipping} fault plane for TPV10.}
        \label{fig:TPV10}
    \end{figure}
    
    The fault obeys the slip-weakening friction law, with the fault strength
    $\tau = C_0 + f(S) \sigma_n$ where $\sigma_n$ is the compressive-norm-stress,  $f(S)$ is the friction coefficient  given by \eqref{eq:slip-weakening} and $C_0$[MPa] is the cohesion. 
    The friction parameters are giving in Table   \ref{tab:friction_tpv10}.
    \begin{table}[h!]
\centering
\begin{tabular}{c| c |  c | c | c }
%
{} & $f_s$ & $f_d$&  $d_c$[m] & $C_0$[MPa]  \\
\hline
$\Gamma_{r}$ & 0.76& 0.448& 0.5 & 0.2 \\
\hline
$\Gamma_{F} \setminus \Gamma_{r}$ &10000& 0.448& 0.5 & 1000  \\
\hline
\end{tabular}
\caption{Friction parameters.}
\label{tab:friction_tpv10}
\end{table}
The initial stresses on the fault are given in Table \ref{tab:init_stress}.
 \begin{table}[h!]
\centering
\begin{tabular}{c|  c| c | c   }
%
{} &  $T_{0n}$[MPa] & $T_{0m}$[MPa] & $T_{0l}$[MPa] 
\\
\hline
$\Gamma_{\text{nuc}}$  & $-7.387\times\widetilde{y}$ & $(f_s+0.00057)\times\sigma_{0n} + C_0 $ & 0 \\
\hline
$\Gamma_{F}\setminus\Gamma_{\text{nuc}}$   & $-7.387\times\widetilde{y}$ & $0.55\times\sigma_{0n}$ & 0\\
\hline
\end{tabular}
\caption{Initial stress.}
\label{tab:init_stress}
\end{table}
Here $T_{0n}$ is the initial normal-stress, $T_{0m}$ is the initial  shear-stress along dip, $T_{0l}$ is the initial  shear-stress along strike and the initial compressive normal-stress is $\sigma_{0n} = -T_{0n}$. 
At the nucleation patch $\Gamma_{\text{nuc}}$ the initial shear-stress exceeds the fault's peak strength, $\tau_0 > \tau_{p}$. The initiation of rupture is instantaneous at $t>0$.
  Note that on the fault plane $\Gamma_F$ but outside the $15$ km -by- $30$ km faulting area, that is in $\Gamma_{F} \setminus \Gamma_{r}$, there is a strength barrier that makes it impossible for the fault to slip. The strength barrier is established by the very large values for the cohesion $C_0 = 1000$ MPa and the static coefficient of friction $f_s = 10000$, given in Table \ref{tab:friction_tpv10}.

\subsection{Numerical simulations}   
To perform  simulations we truncate the domain and consider the finite computational domain $\Omega$ with
    \begin{align}
       \Omega=\Omega^{-}\cup \Omega^{+}, && \Omega^{-} = [-20,\widetilde{x}] \times [0,20] \times [-20,20], && \Omega^{+} = [\widetilde{x}, 20] \times [0,20] \times [-20,20].
    \end{align}
The sub-blocks $\Omega^{-}$ and  $\Omega^{+}$ of isotropic elastic solid are glued together by friction.
The material properties are homogeneous throughout the domain  with the following properties,
    \begin{align}
        c_p =5.716 ~\ \text{km/s}, && c_s = 3.300 ~\ \text{km/s}, && \rho = 2.700 ~\ \text{g/cm$^3$} .
    \end{align}
At the boundary on the Earth's surface,  $y = 0$, we impose the free-surface boundary condition by setting the traction vector to zero.  The artificial boundaries at $x = \pm 20$, $y = 20$ and $z = \pm 20$ are surrounded by the PML \cite{DuruKozdonKreiss2016,DuruRannabauerGabrielKreissBader2019} of sufficiently small width of  $1.2$ km  to absorb outgoing waves. Effective numerical treatment of the PML in the presence of nonlinear frictional sources is a topic of future work.

 A boundary conforming curvilinear mesh, obeying the fault topography, is generated for the sub-blocks $\Omega^{-}$ and $\Omega^{+}$.   Note that because of the dipping fault plane, the computational mesh is highly skewed.
 Numerical approximations are perform using the 6th order accurate traditional SBP operator, and upwind DP and DRP  SBP operators with order of accuracy $4$, $5$, $6$ and $7$. 
We will probe the solutions at a receiver station placed at $(7.5,12)$ on the fault plane and which is $13.2$ km from the hypo-centre  $(12,0)$.
 
We run the simulations until the final time $T = 15$ s.  Snapshots of the slip-rate on the fault plane are displayed in Figure \ref{fig:snap_shots_sliprate}, showing the evolution of the rupture on the fault plane, $\Gamma_r$.
Comparisons of the time history of the slip-rate at the receiver station $(7.5, 12)$ are shown in  Figure \ref{fig:slip_rate_100m_50m}, at $h = 100$ m and $h = 50$ m spatial step sizes.  
In Figure \ref{fig:snap_shots_sliprate_comparison}, we compare the slip-rate on the entire fault plane at $t =5.98$ s.
Similarly, comparisons of the time history of the shear-stress at the receiver station $(7.5, 12)$ are shown in  Figure \ref{fig:shear_stress_100m_50m}, at $h = 100$ m and $h = 50$ m spatial step sizes.
\begin{figure}[h!]
    \centering
    \includegraphics[width=0.24\textwidth]{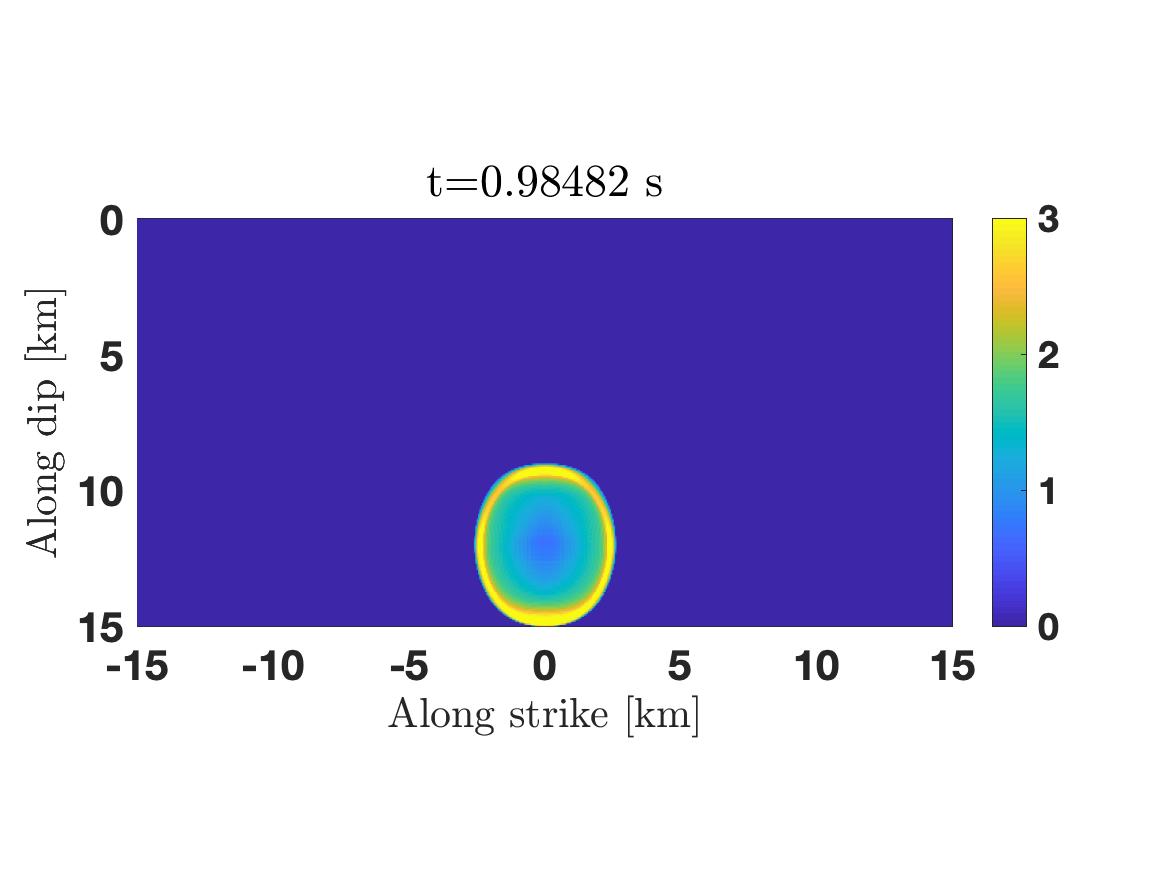}
    \includegraphics[width=0.24\textwidth]{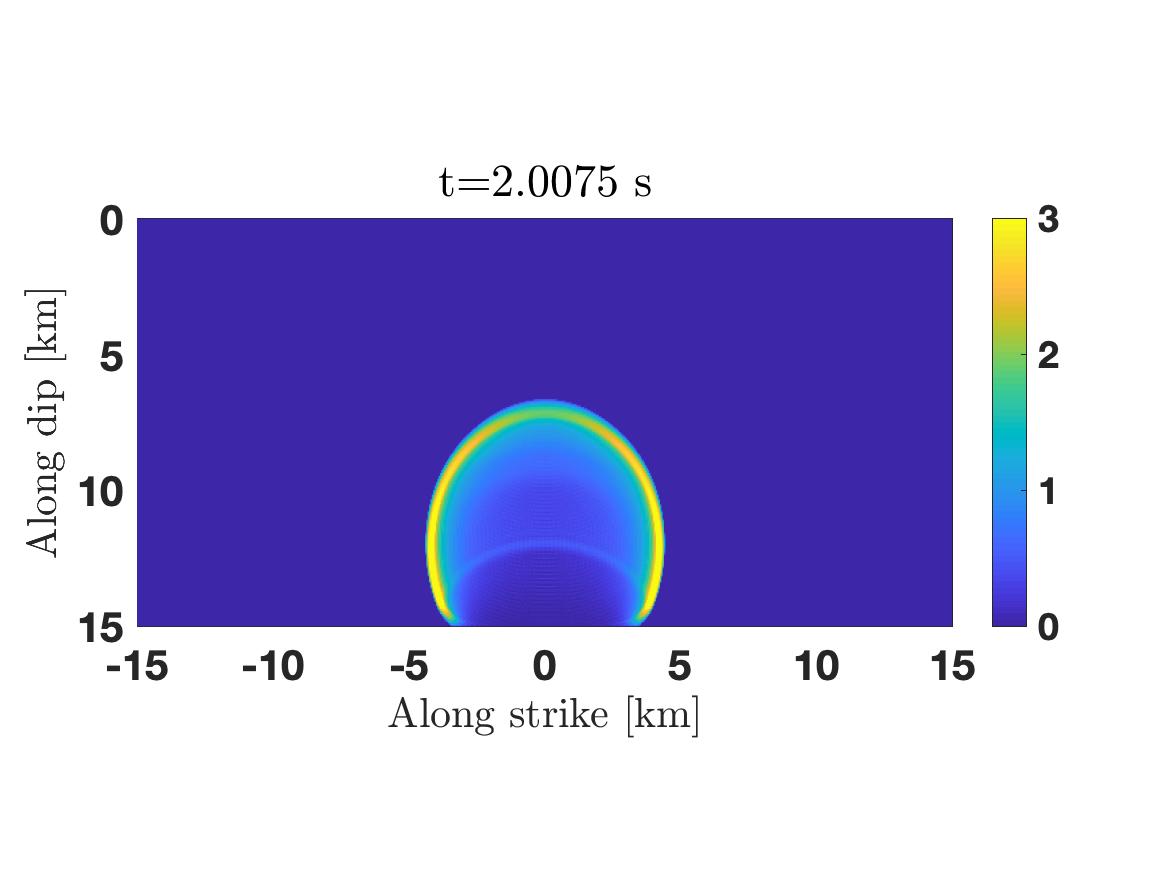}
    \includegraphics[width=0.24\textwidth]{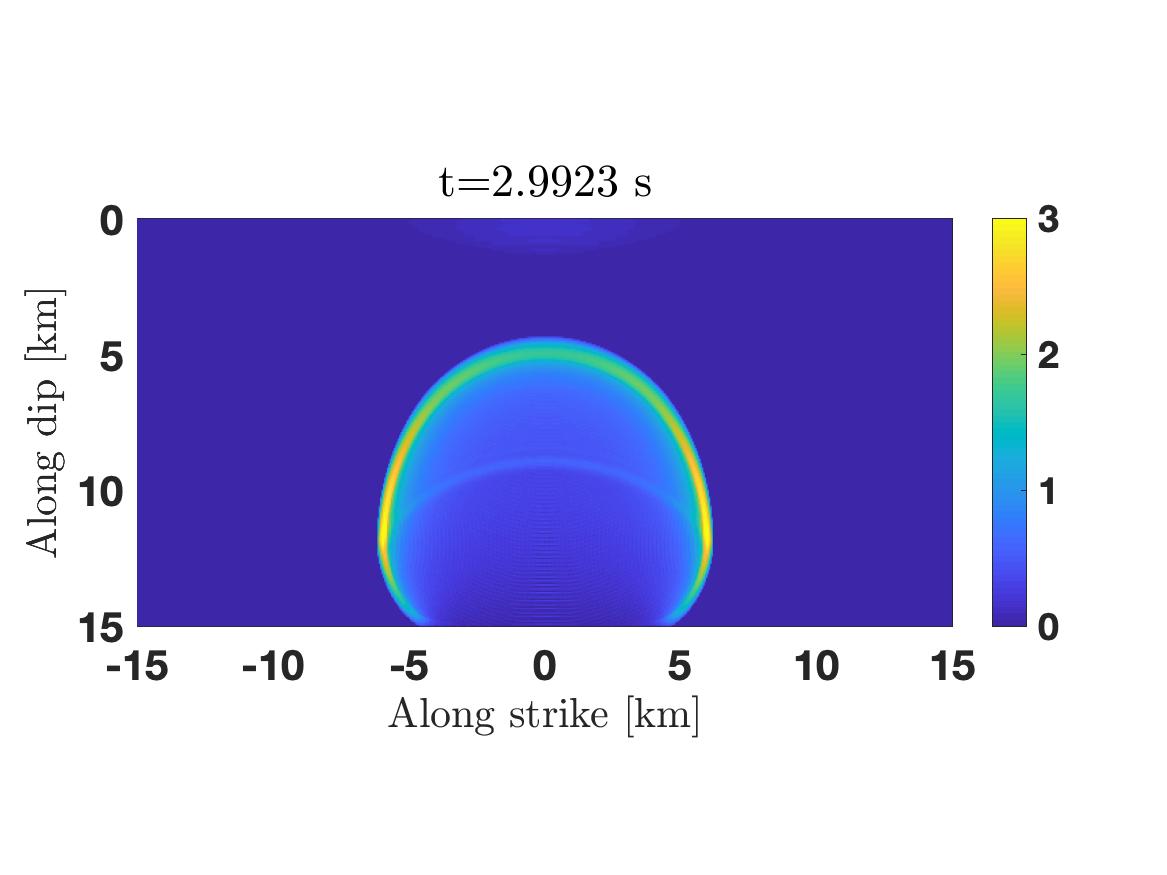}
    \includegraphics[width=0.24\textwidth]{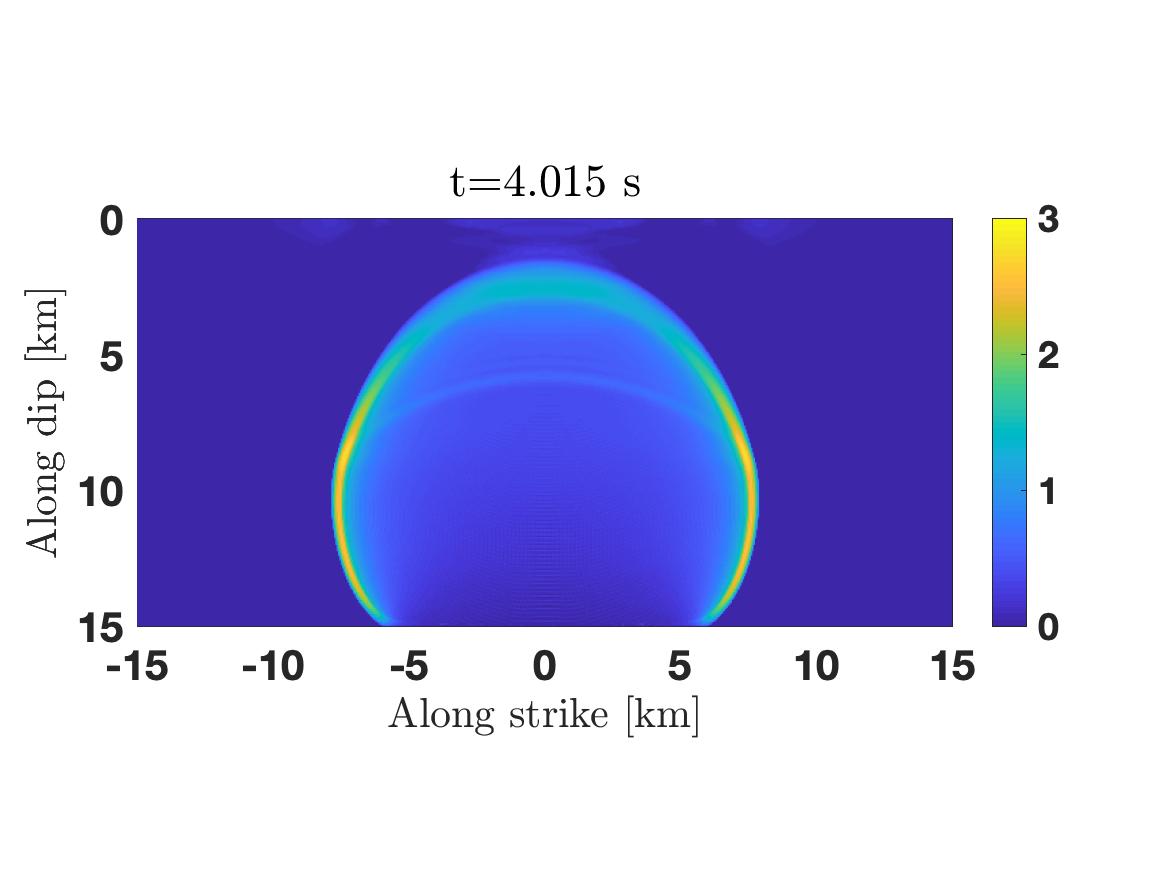}
    \includegraphics[width=0.24\textwidth]{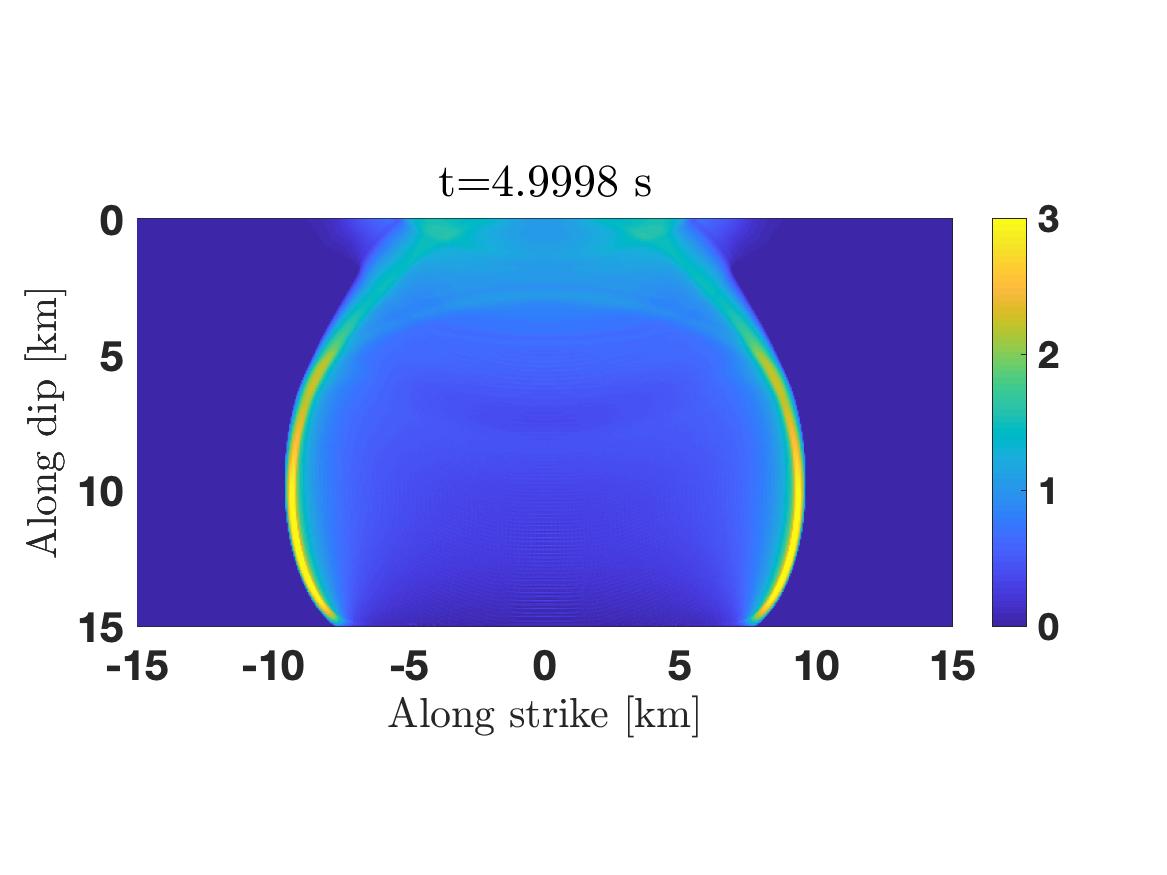}
    \includegraphics[width=0.24\textwidth]{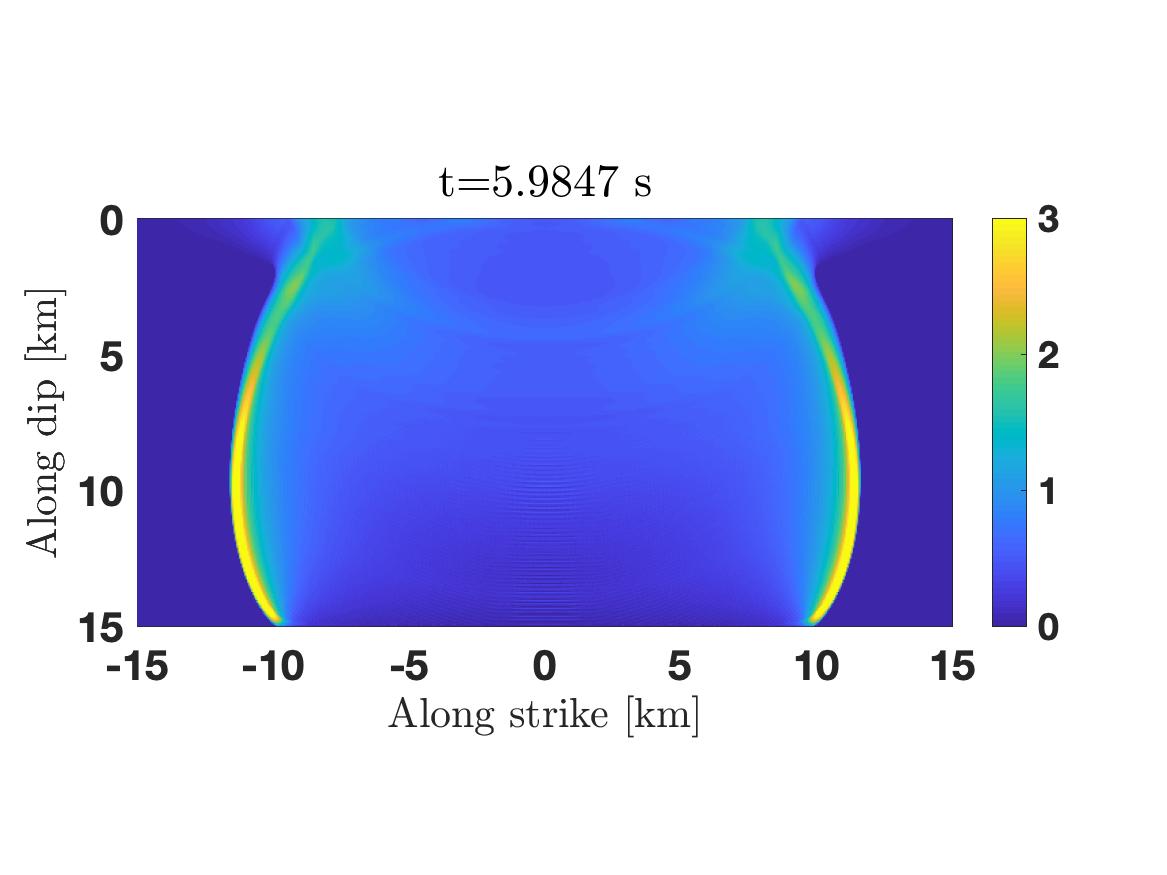}
    \includegraphics[width=0.24\textwidth]{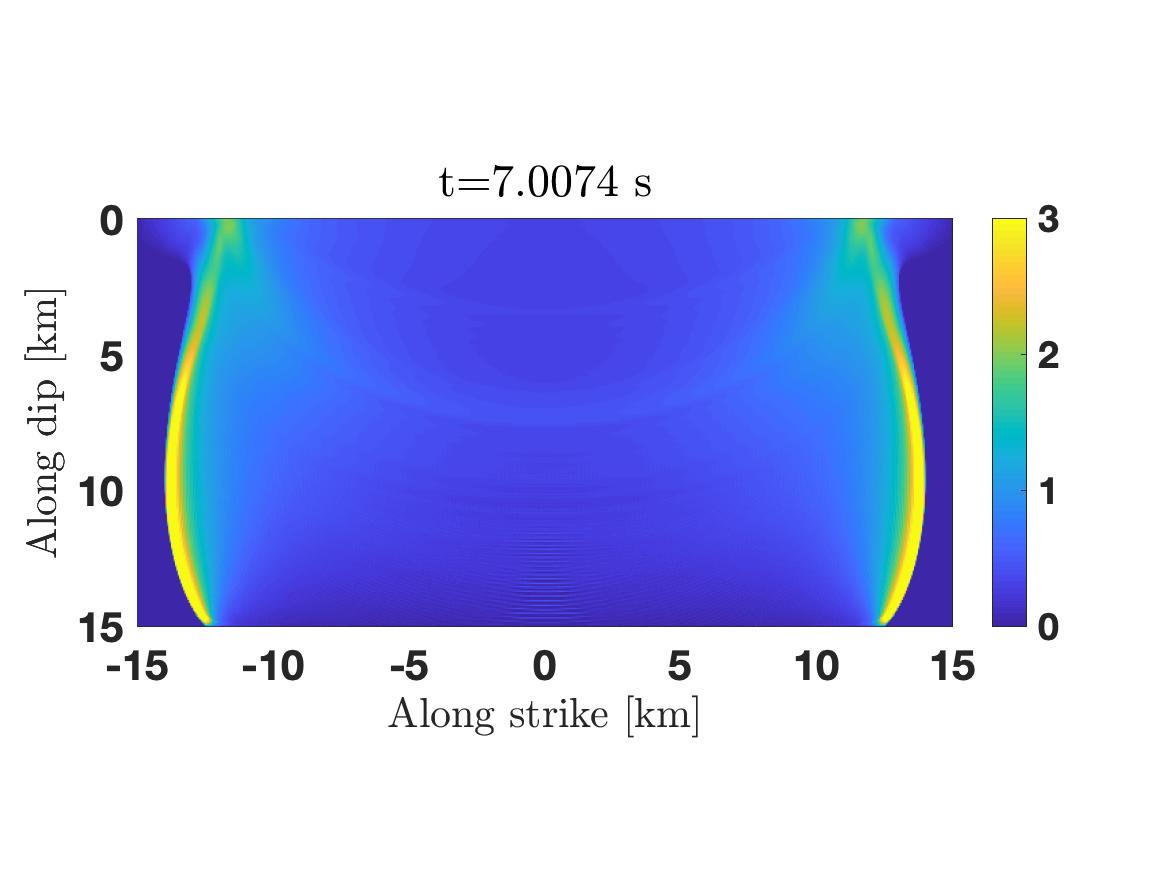}
    \includegraphics[width=0.24\textwidth]{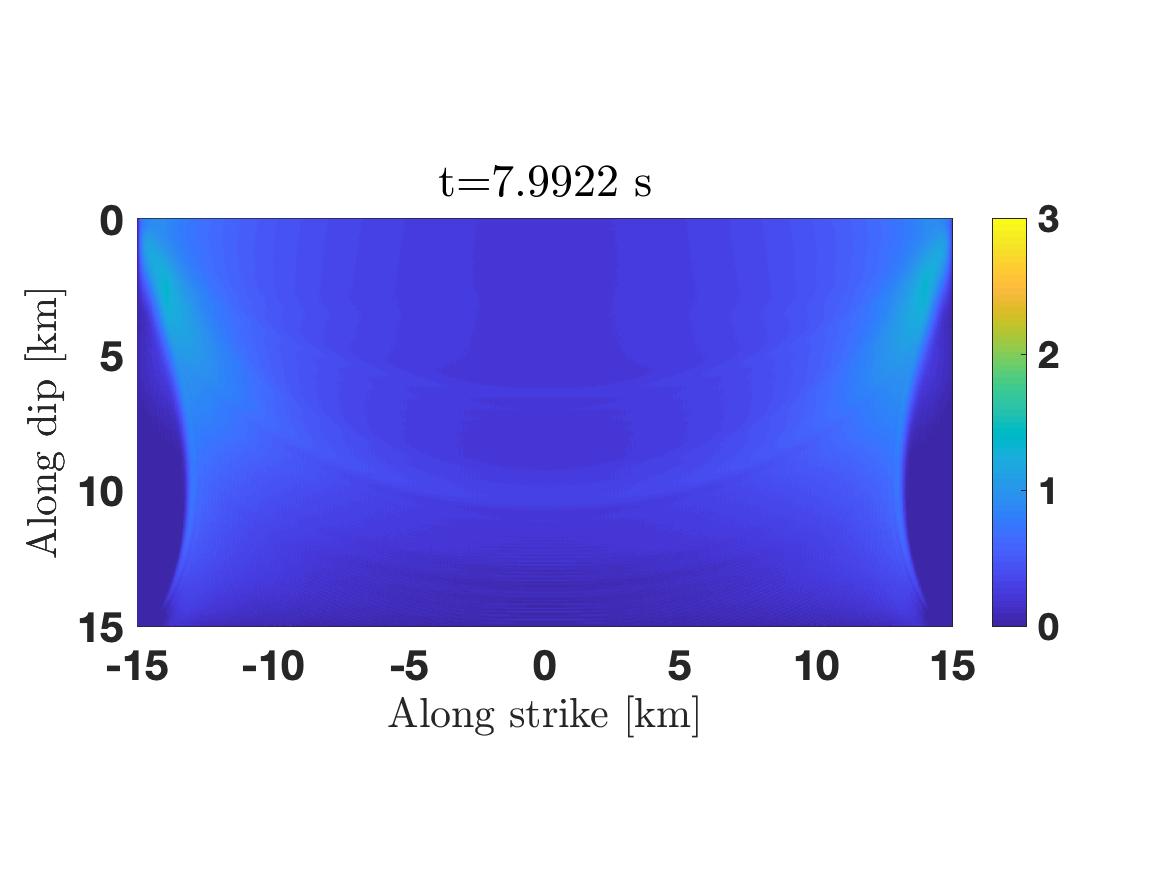}
    \caption{Snapshots of the slip-rate on  a dipping fault plane   at $t \in \{0.98, 2.0, 3.0, 4.0, 5.0, 6.0, 7.0, 7.99\}$ seconds. The simulation is computed with the DP SBP operator of order $4$ at h = 50 m grid spacing. } 
    \label{fig:snap_shots_sliprate}
\end{figure}

\begin{figure}[h!]
\centering
{
 \stackunder[5pt]{\includegraphics[width=8cm]{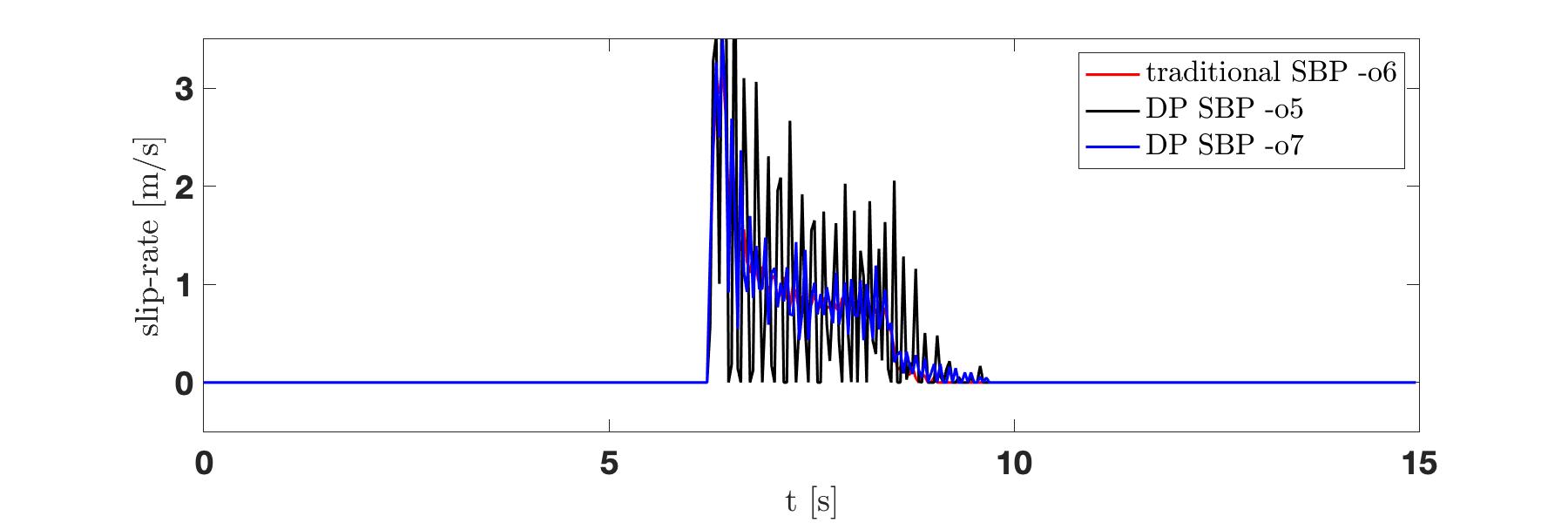}}{\text{DP odd order}}
\stackunder[5pt]{\includegraphics[width=8cm]{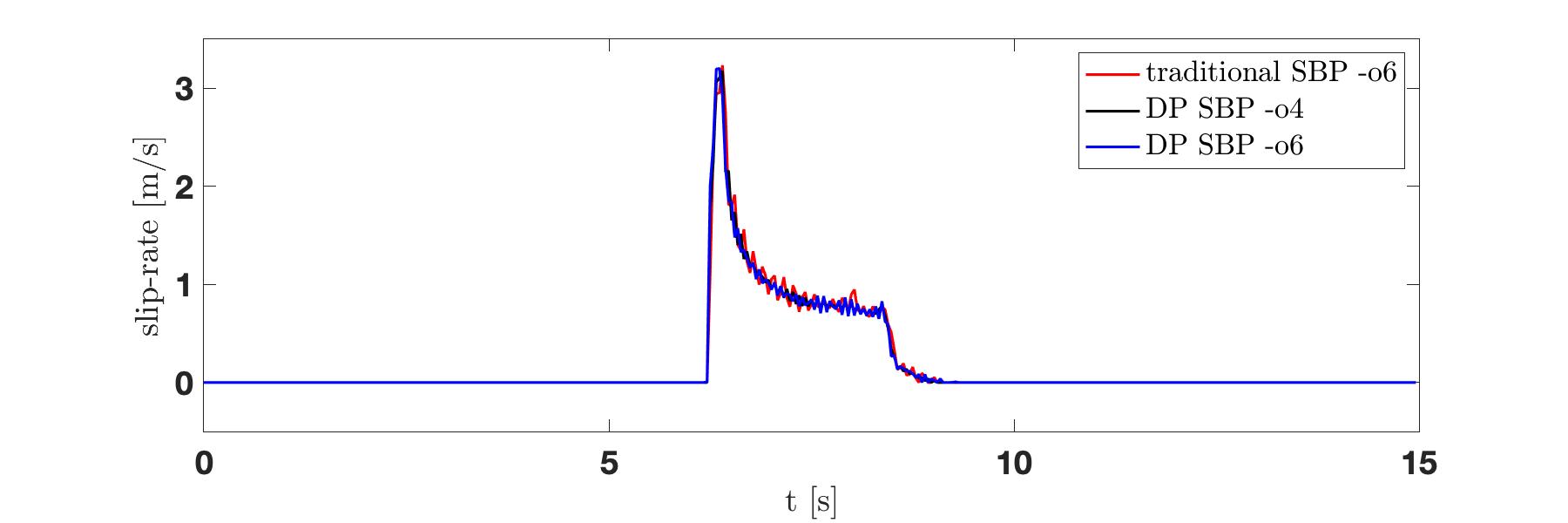}}{\text{DP even order}}
\stackunder[5pt]{\includegraphics[width=8cm]{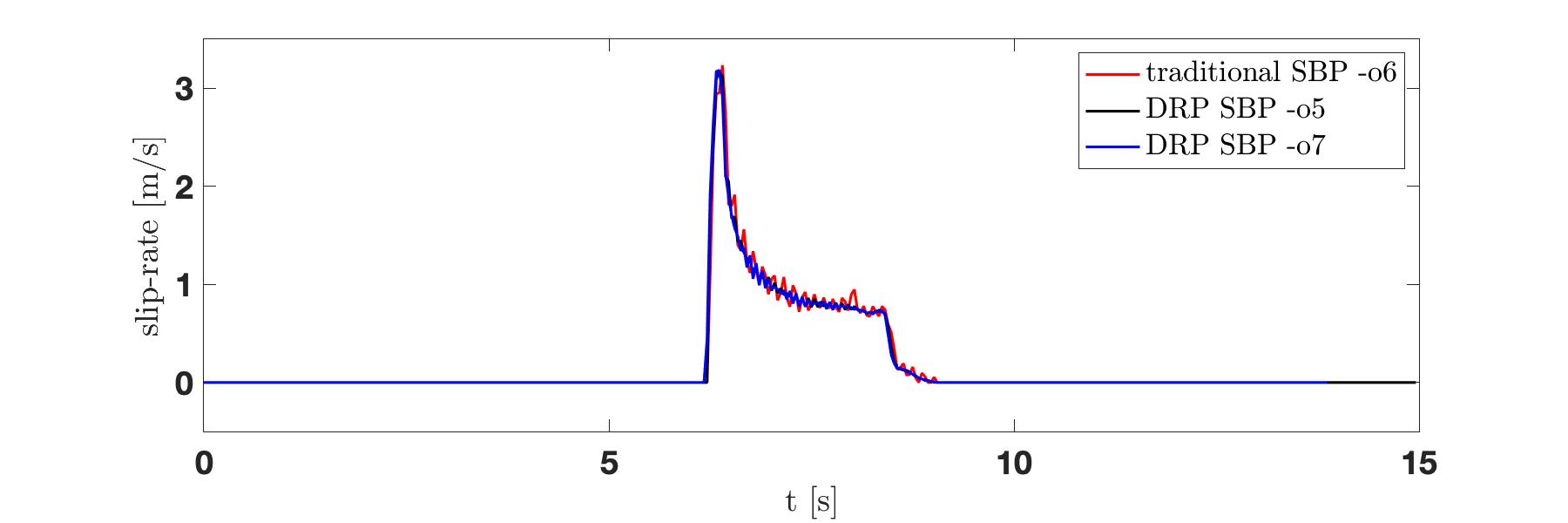}}{\text{DRP odd order}}
\stackunder[5pt]{\includegraphics[width=8cm]{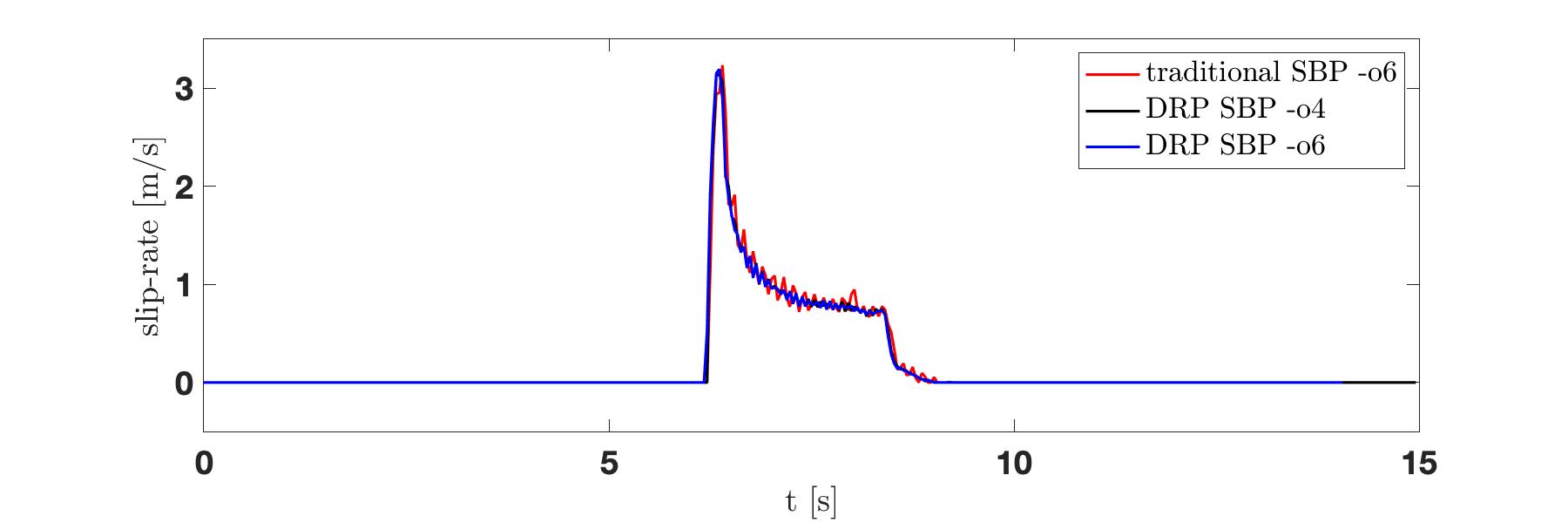}}{\text{DRP even order}}
\renewcommand{\thefigure}{6a}
\caption{Time-series of the vertical slip-rate at $h = 100$ m grid spacing for an on-fault station placed at $(7.5,12)$.}
 \stackunder[5pt]{\includegraphics[width=8cm]{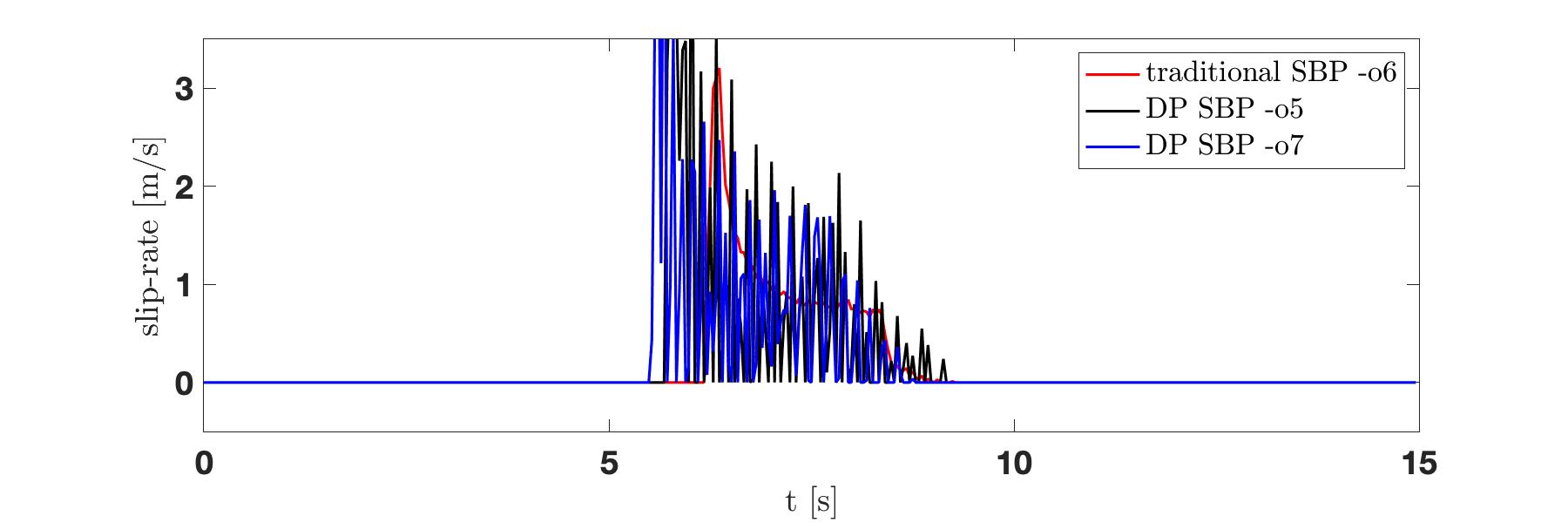}}{\text{DP odd order}}
\stackunder[5pt]{\includegraphics[width=8cm]{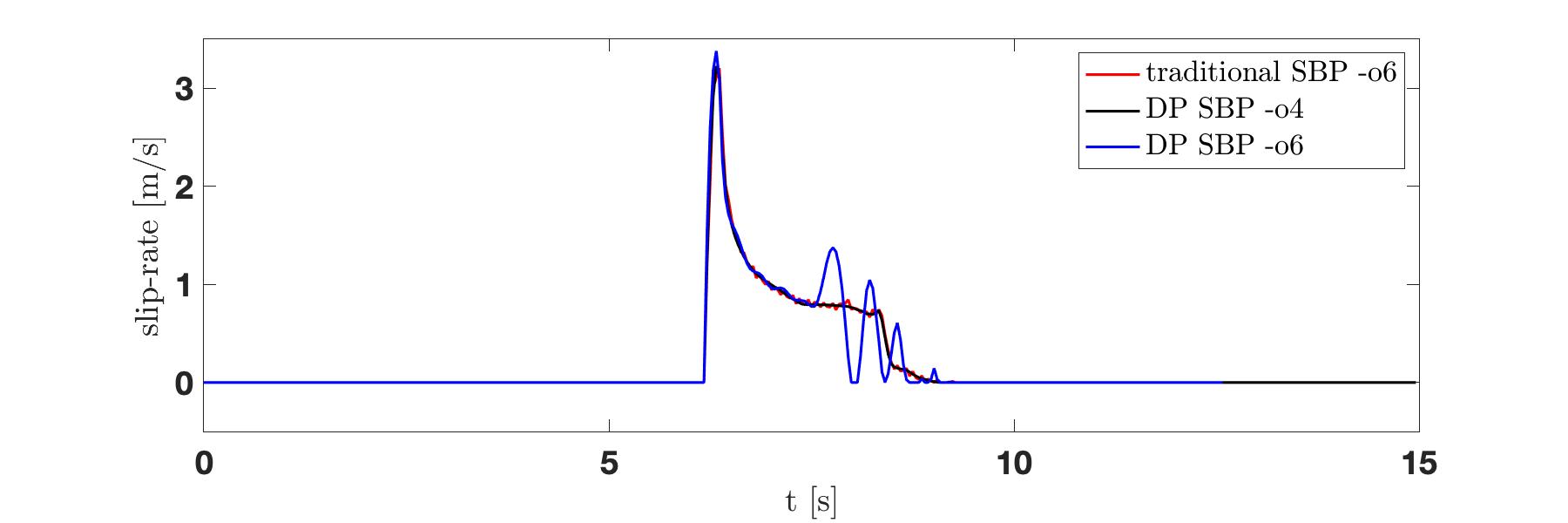}}{\text{DP even order}}
\stackunder[5pt]{\includegraphics[width=8cm]{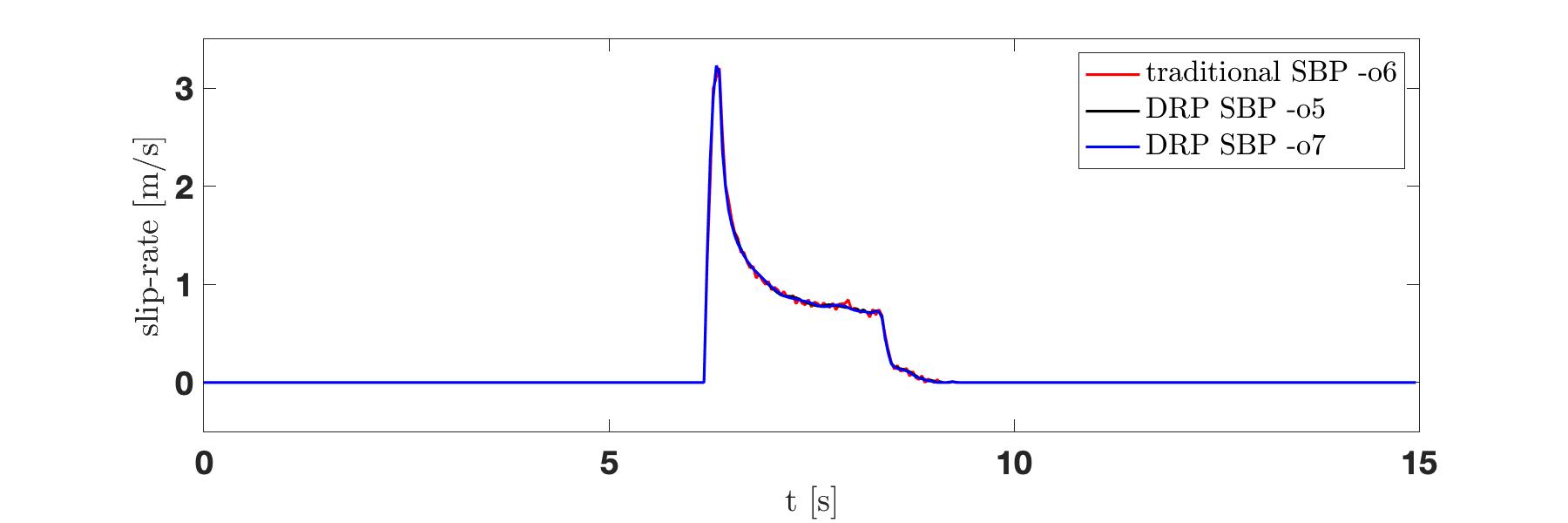}}{\text{DRP odd order}}
\stackunder[5pt]{\includegraphics[width=8cm]{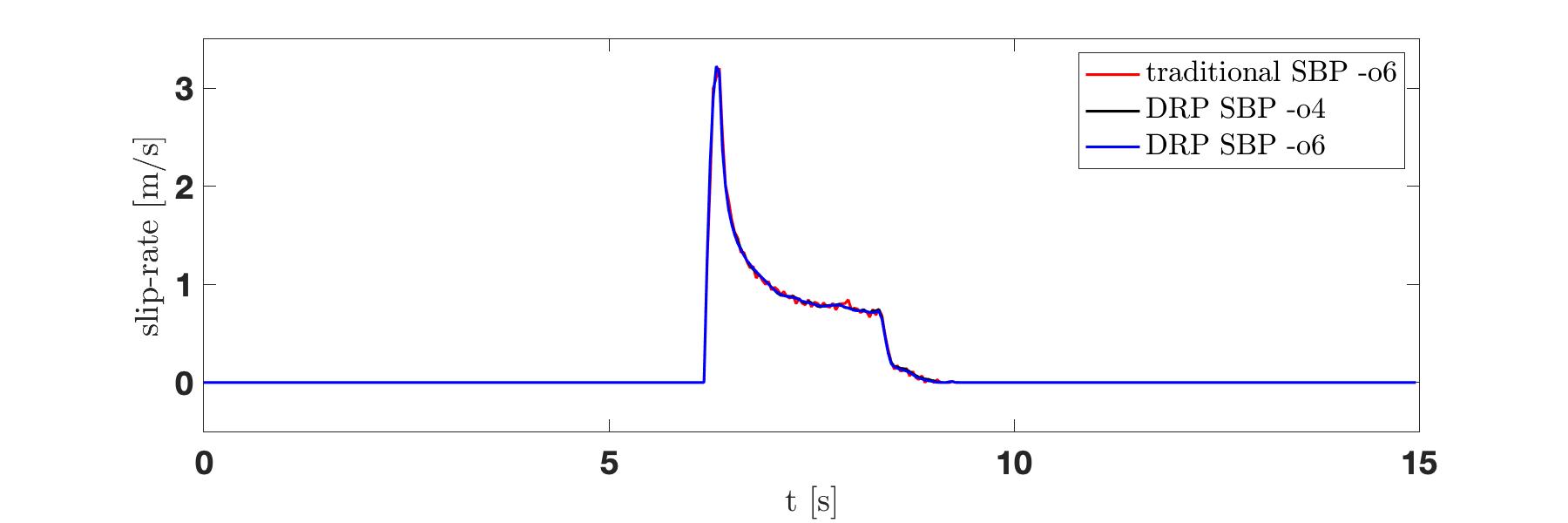}}{\text{DRP even order}}
\renewcommand{\thefigure}{6b}
\caption{Time-series of the vertical slip-rate at $h = 50$ m grid spacing for an on-fault station placed at $(7.5,12)$.}
}
\addtocounter{figure}{-2}
\caption{Comparisons of the slip-rate at at two levels of mesh refinements, $h = 100, 50$ m. The simulations are performed with the traditional SBP FD operator of order $6$,  the upwind DP and DRP SBP FD operators of order $4, 5, 6 7$. Note that at $h=100$ m mesh resolution, even (4 and 6) order DP SBP operators support minimal high frequency errors while the odd  (5 and 7) order DP SBP operators amplify high frequency numerical errors. The DRP SBP operators support much less  than the traditional SBP operator. With mesh refinement $h = 50$ m, the 4th order upwind DP SBP operator, the traditional SBP FD operator, and all DRP SBP operators compute convergent numerical solutions. For odd  (5 and 7) order DP SBP operators, mesh refinement   amplifies further high frequency numerical artefacts while the 6th order DP SBP operator generates a low frequency error.}
\label{fig:slip_rate_100m_50m}
\end{figure}

\begin{figure}[h!]
\centering
{
 \stackunder[5pt]{\includegraphics[width=8cm]{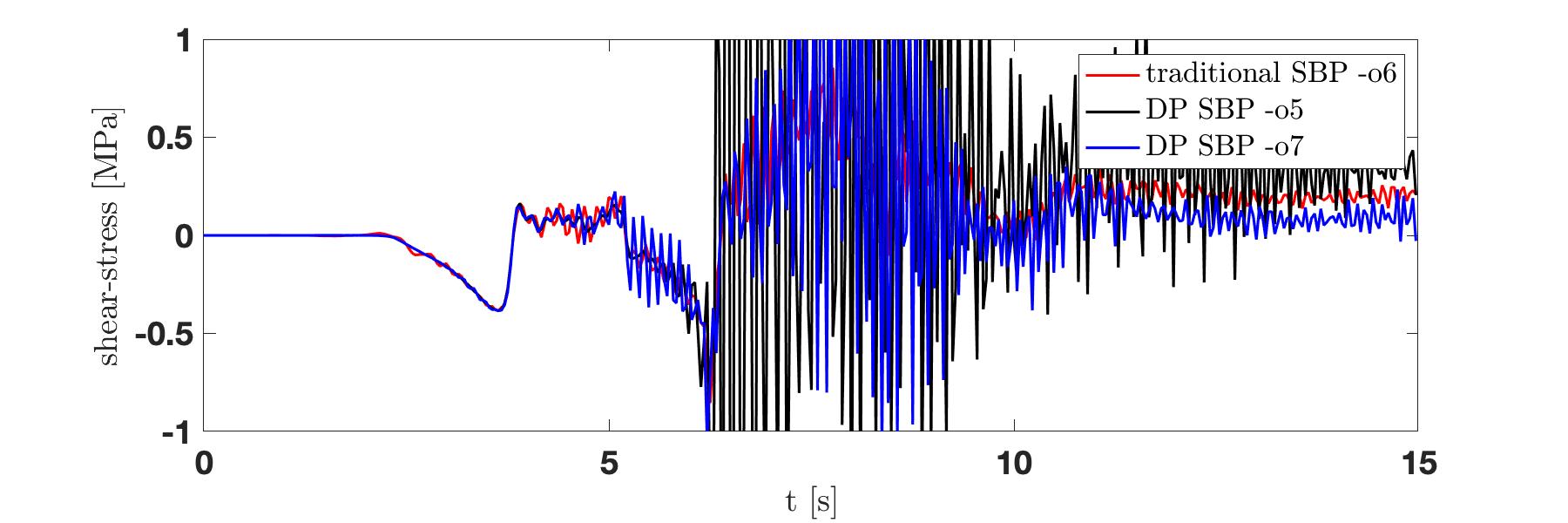}}{\text{DP odd order}}
\stackunder[5pt]{\includegraphics[width=8cm]{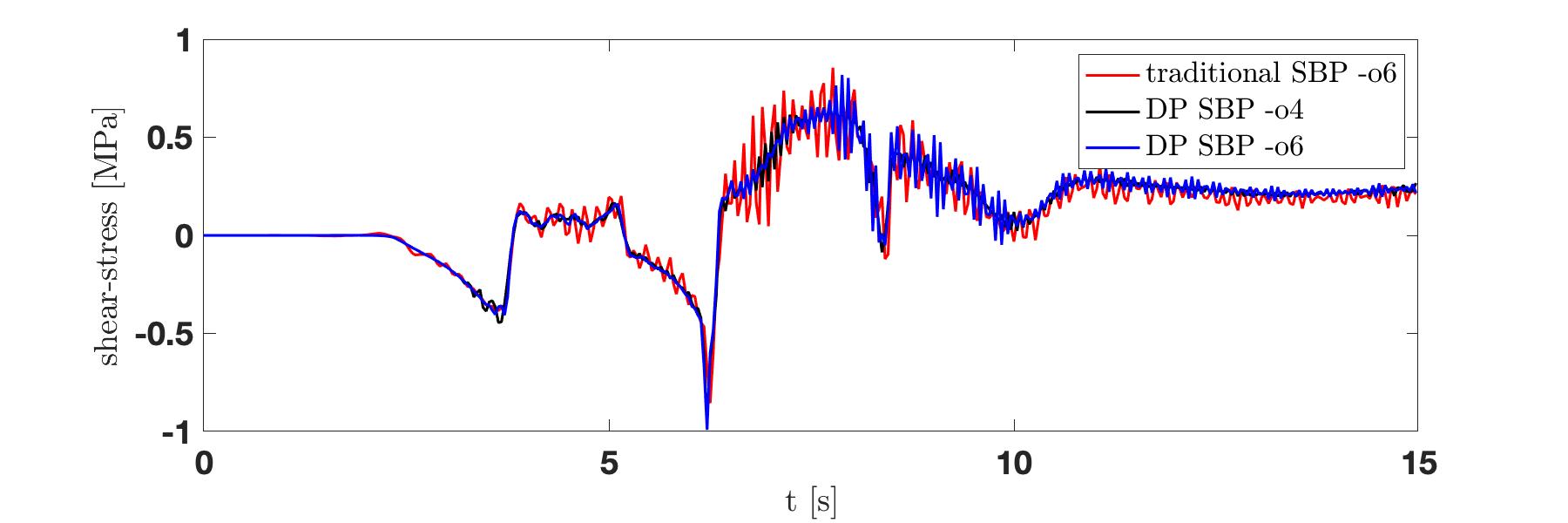}}{\text{DP even order}}
\stackunder[5pt]{\includegraphics[width=8cm]{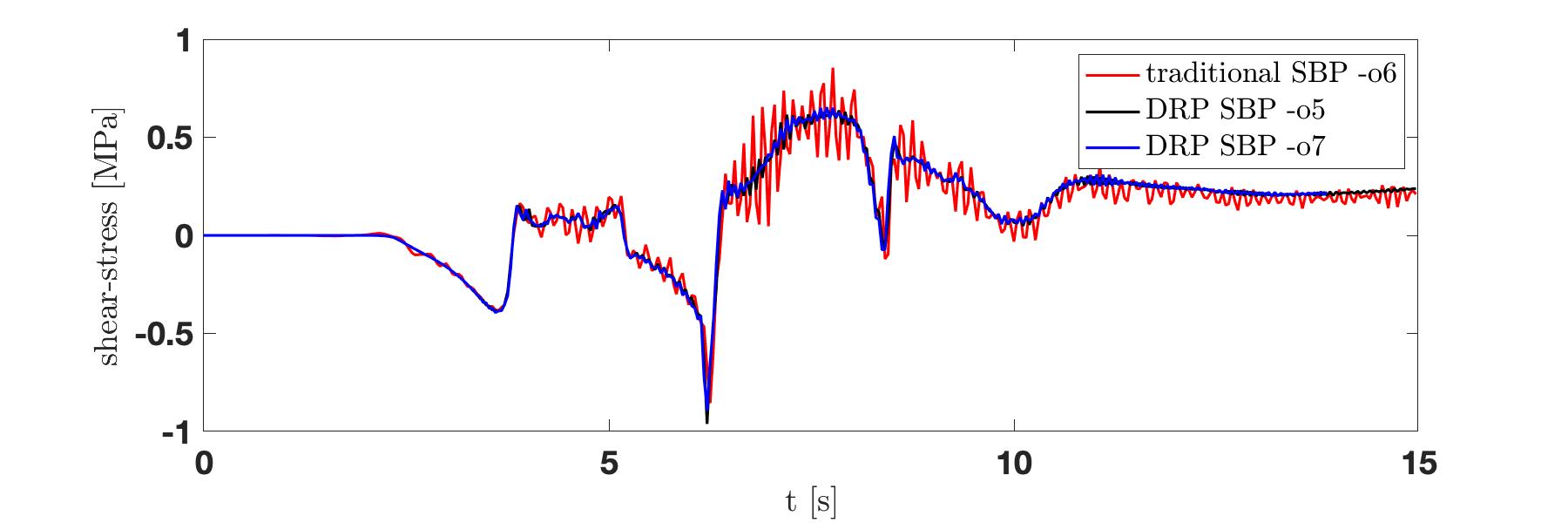}}{\text{DRP odd order}}
\stackunder[5pt]{\includegraphics[width=8cm]{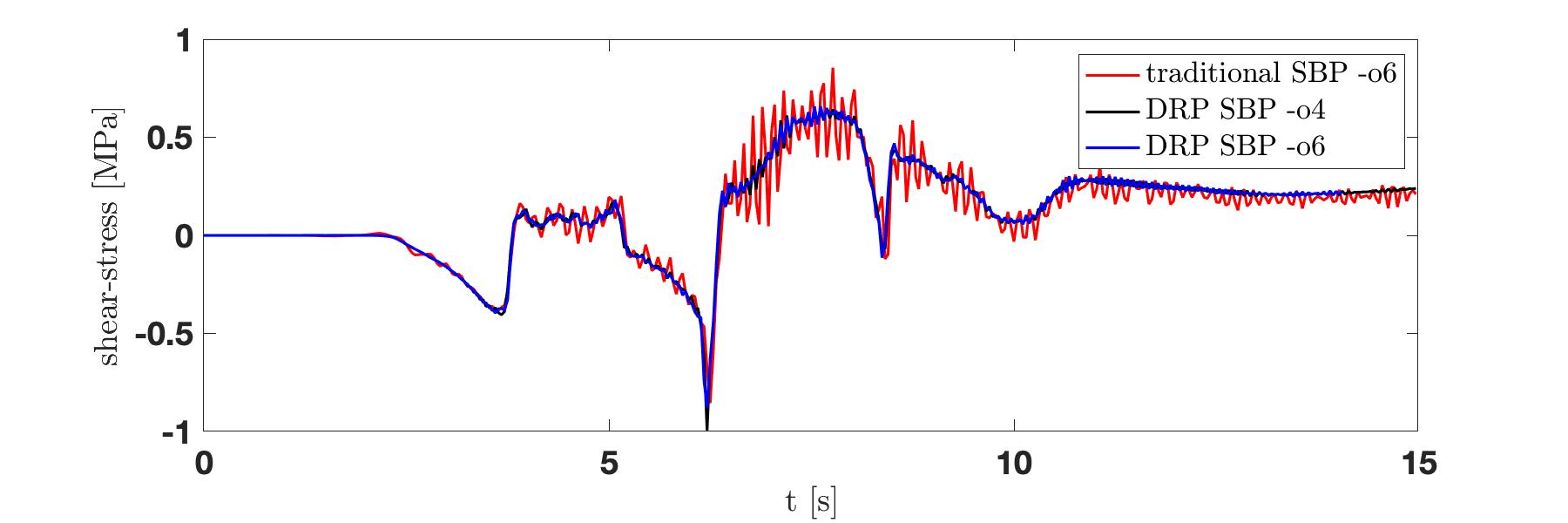}}{\text{DRP even order}}
\renewcommand{\thefigure}{7a}
\caption{Time-series of the horizontal shear-stress at $h = 100$ m grid spacing for an on-fault placed at $(7.5,12)$.}
 \stackunder[5pt]{\includegraphics[width=8cm]{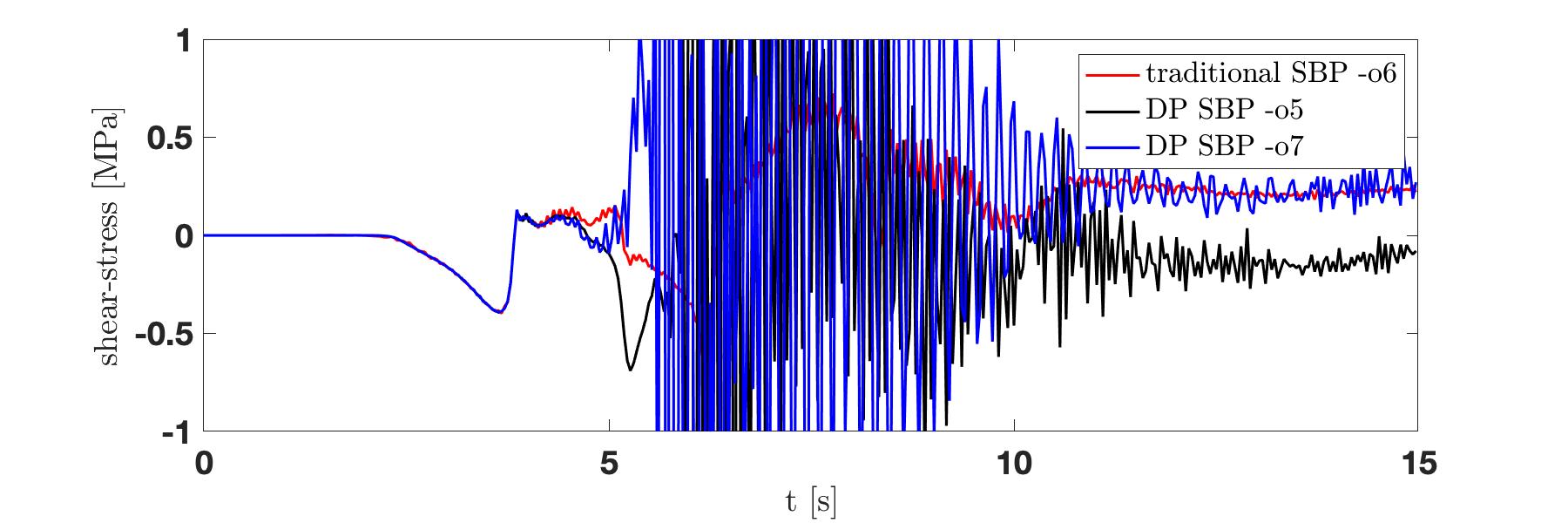}}{\text{DP odd order}}
\stackunder[5pt]{\includegraphics[width=8cm]{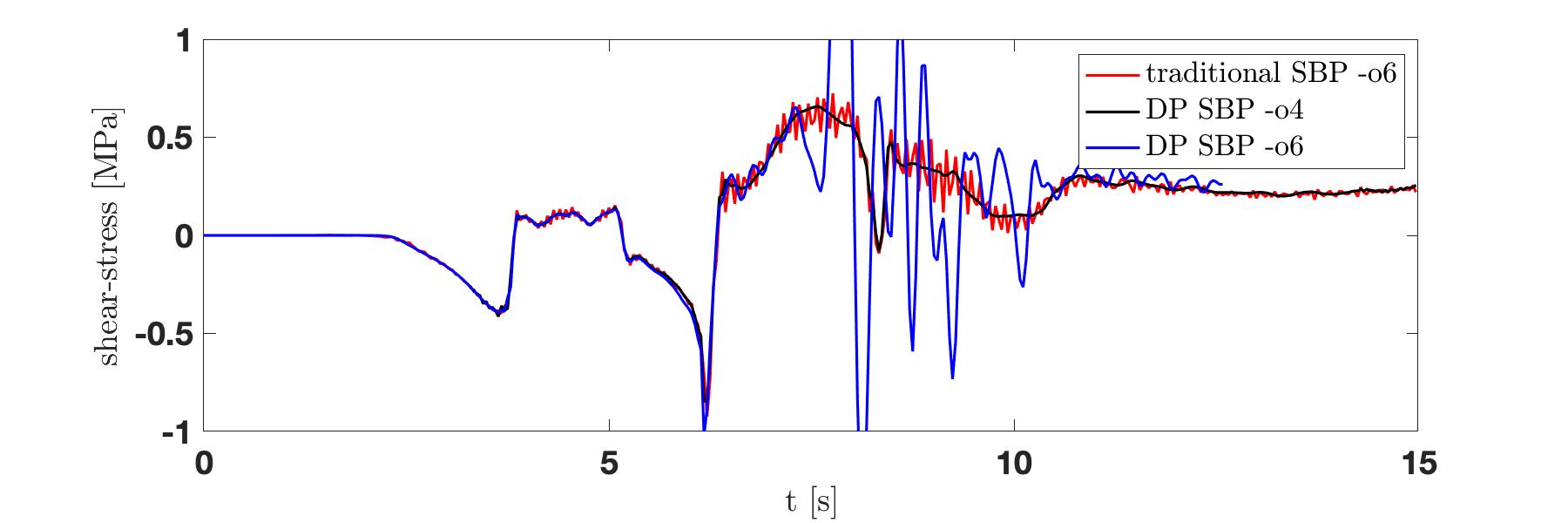}}{\text{DP even order}}
\stackunder[5pt]{\includegraphics[width=8cm]{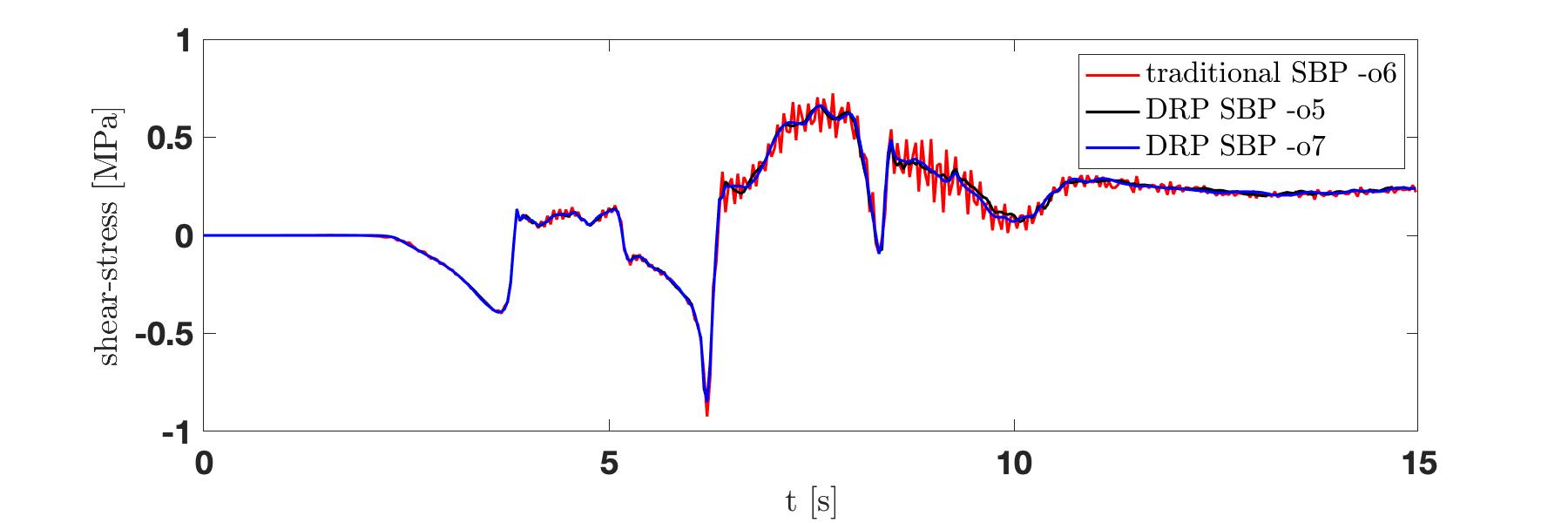}}{\text{DRP odd order}}
\stackunder[5pt]{\includegraphics[width=8cm]{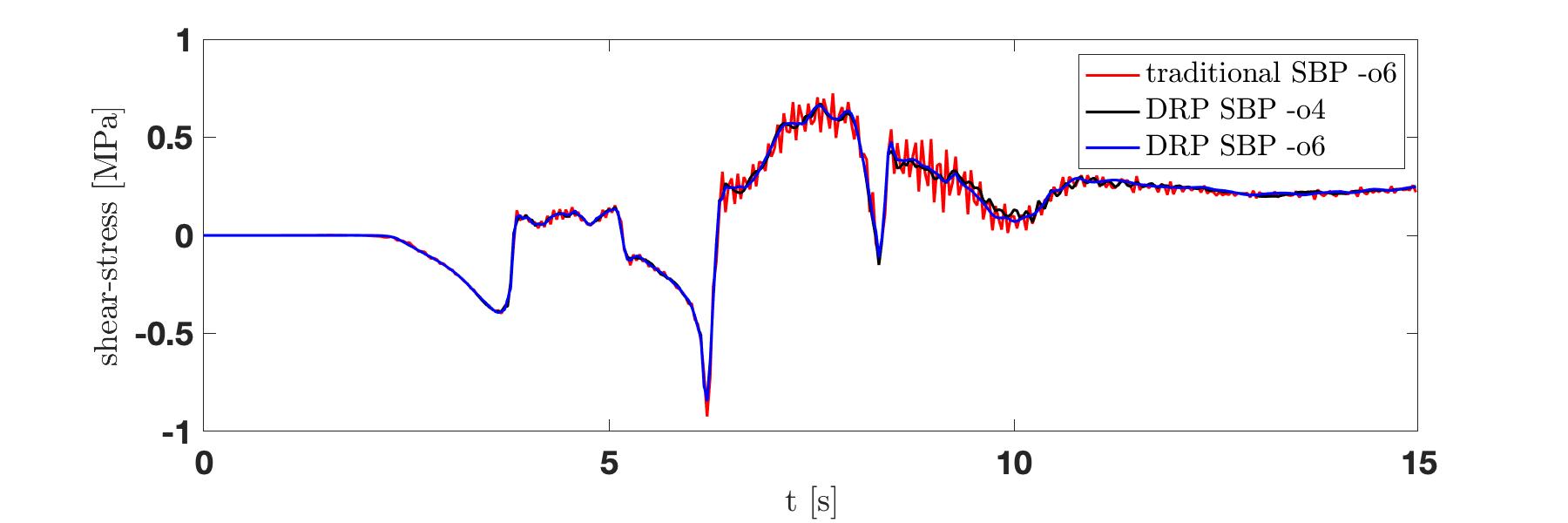}}{\text{DRP even order}}
\renewcommand{\thefigure}{7b}
\caption{Time-series of the horizontal shear-stress at $h = 50$ m grid spacing for an on-fault placed at $(7.5,12)$.}
}
\addtocounter{figure}{-2}
\caption{Comparisons of the shear-stress at at two levels of mesh refinements, $h = 100, 50$ m. The simulations are performed with the traditional SBP FD operator of order $6$,  the upwind DP and DRP SBP FD operators of order $4, 5, 6 7$. Note that at $h=100$ m mesh resolution, even (4 and 6) order DP SBP operators support minimal high frequency errors while the odd  (5 and 7) order DP SBP operators amplify high frequency numerical errors. The DRP SBP operators  have much less spurious oscillations than the traditional SBP operator. With mesh refinement $h = 50$ m, the 4th order upwind DP SBP operator, the traditional SBP FD operator, and all DRP SBP operators compute convergent numerical solutions. For odd  (5 and 7) order DP SBP operators, mesh refinement catastrophically amplifies  high frequency numerical artefacts while the 6th order DP SBP operator generates low frequency errors.}
\label{fig:shear_stress_100m_50m}
\end{figure}
\begin{figure}[h!]
\centering
 {
 \stackunder[5pt]{\includegraphics[width=0.275\textwidth]{tpv10_upwind4_sliprate5_9847s.jpg}}{$4$\text{th order}}%
\hspace{-0.725cm}%
\stackunder[5pt]{\includegraphics[width=0.275\textwidth]{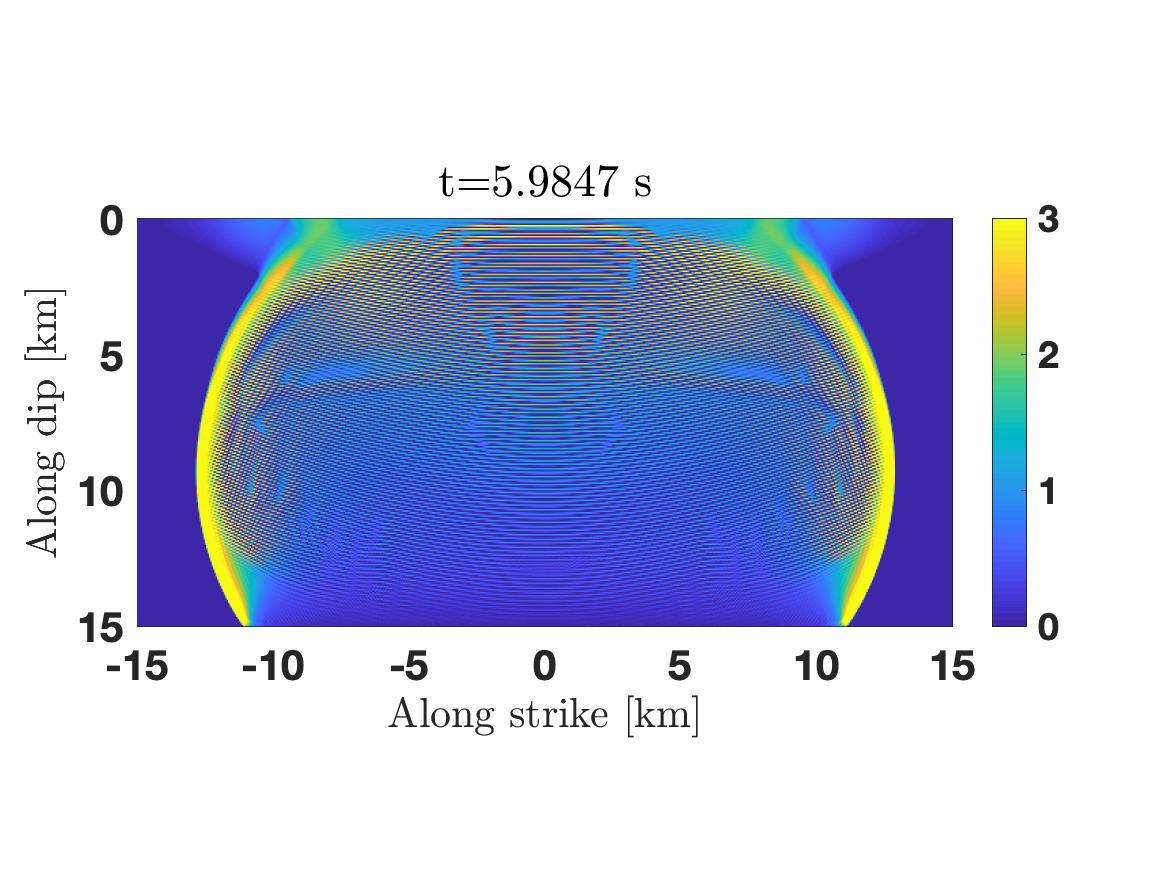}}{$5$\text{th order}}%
\hspace{-0.725cm}%
\stackunder[5pt]{\includegraphics[width=0.275\textwidth]{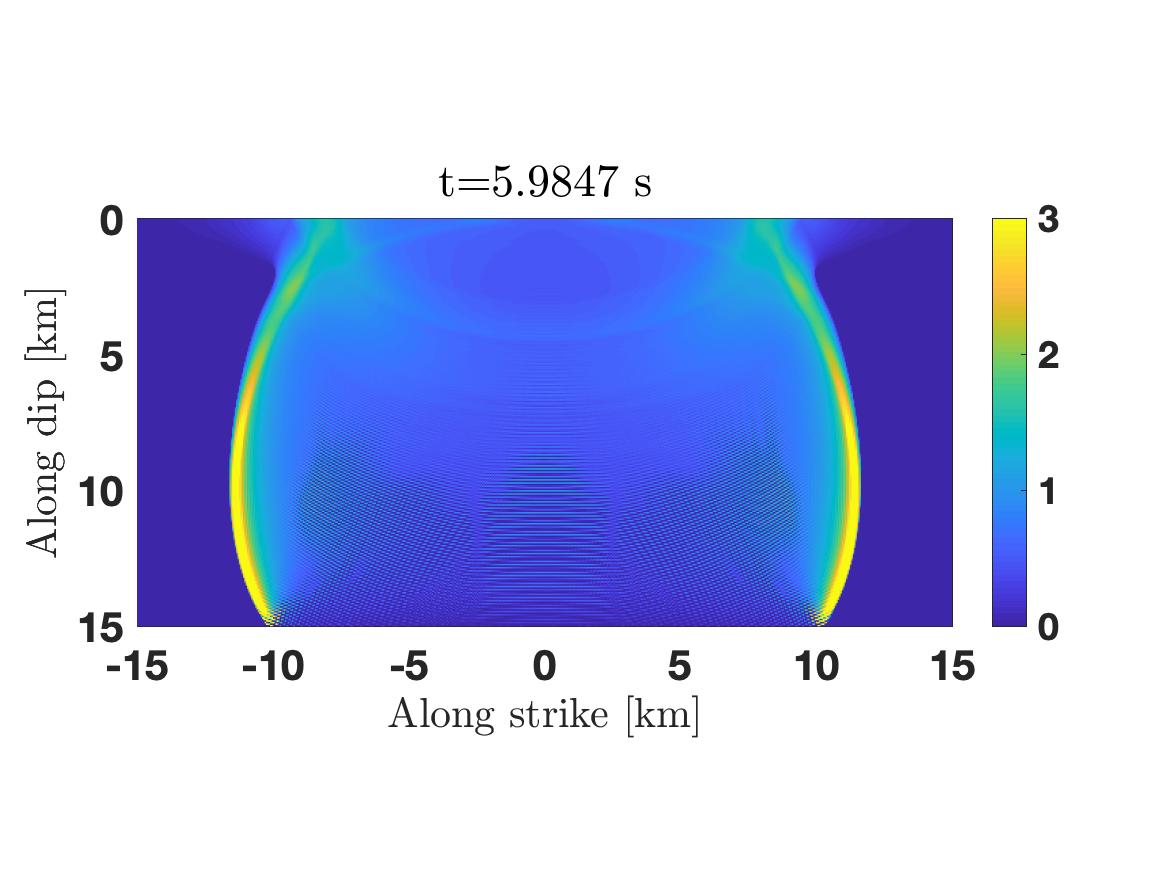}}{$6$\text{th order}}%
\hspace{-0.725cm}%
\stackunder[5pt]{\includegraphics[width=0.275\textwidth]{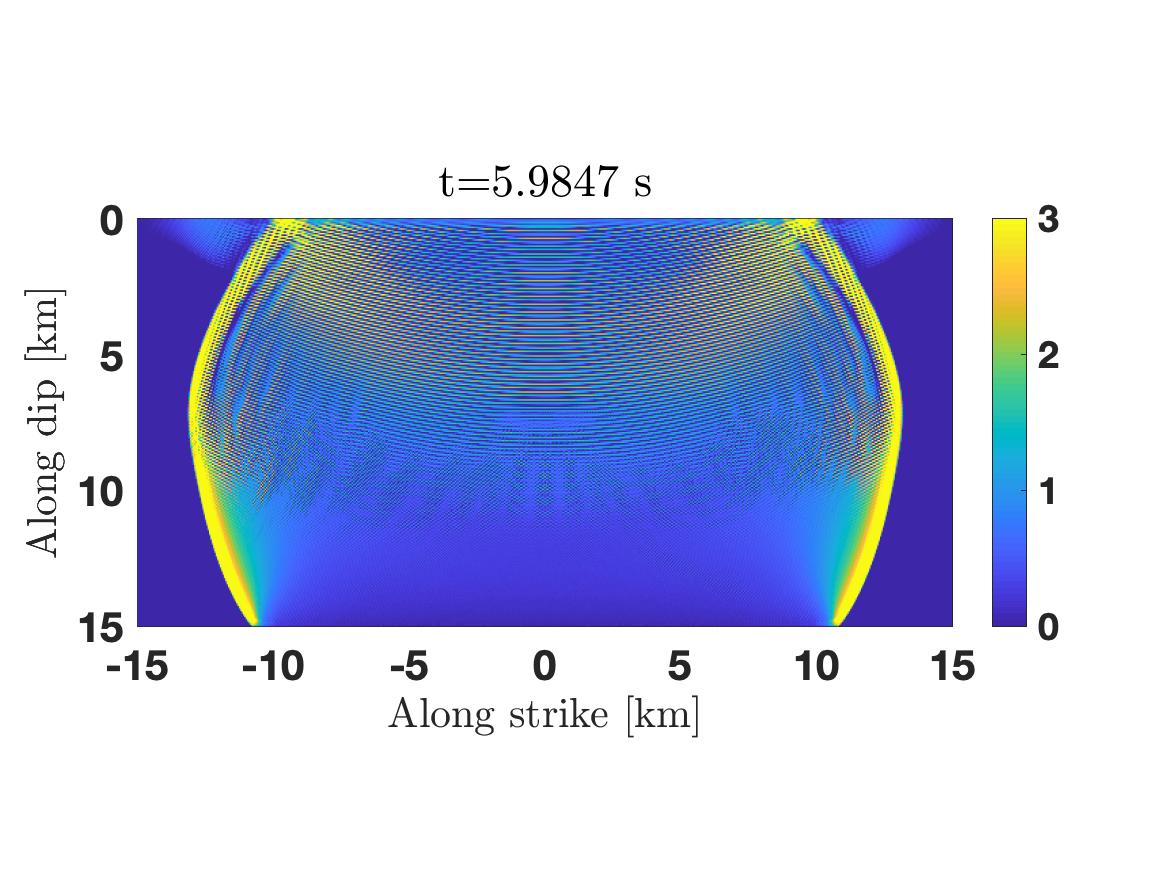}}{$7$\text{th order}}%
\renewcommand{\thefigure}{8a}
\caption{DP SBP operators.}
\stackunder[5pt]{\includegraphics[width=0.275\textwidth]{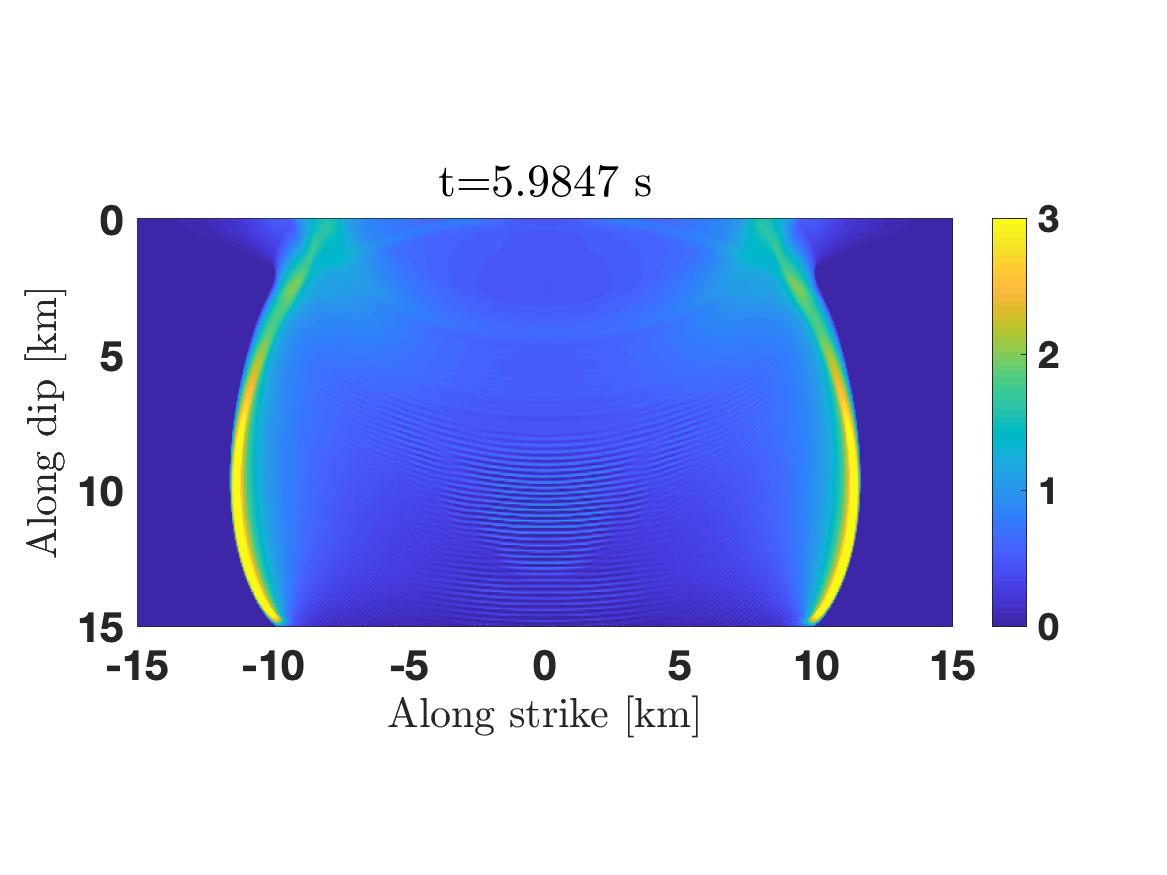}}{$4$\text{th order}}%
\hspace{-0.725cm}%
 \stackunder[5pt]{\includegraphics[width=0.275\textwidth]{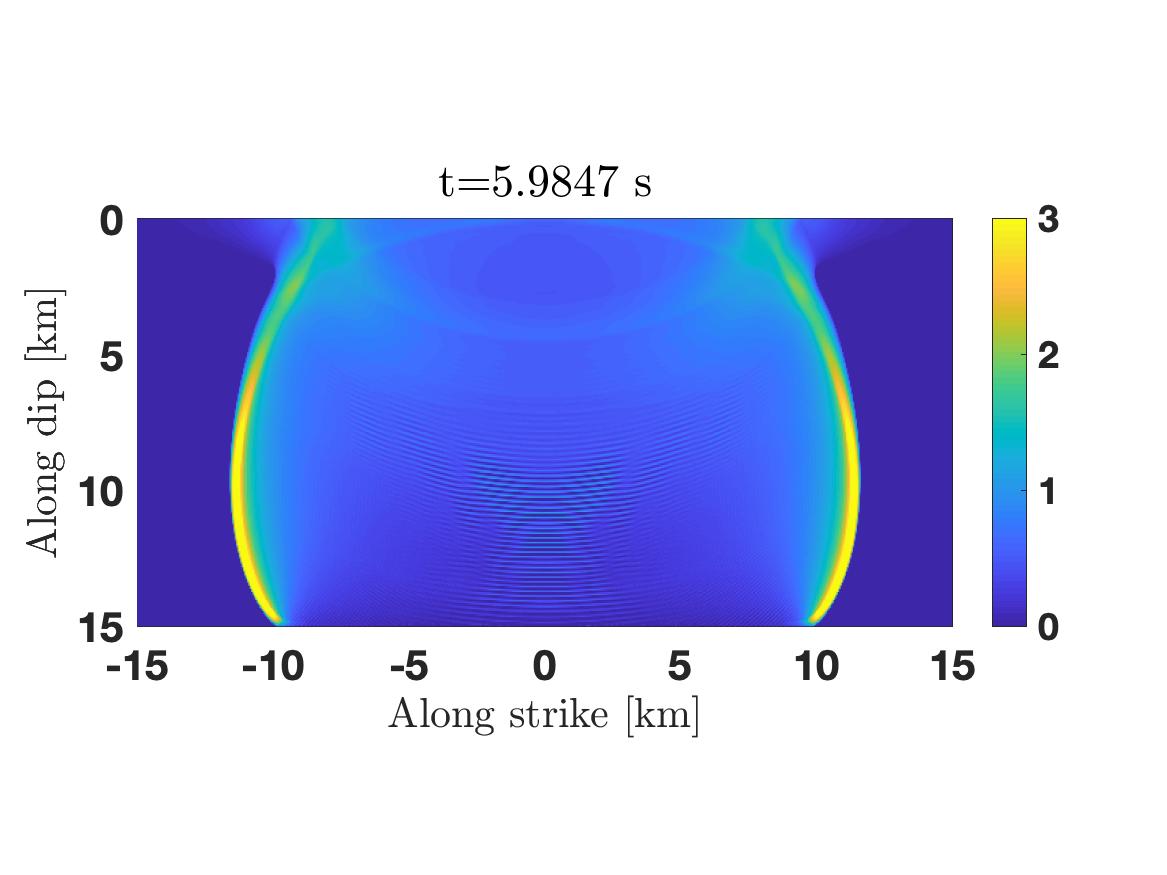}}{$5$\text{th order}}%
\hspace{-0.725cm}%
\stackunder[5pt]{\includegraphics[width=0.275\textwidth]{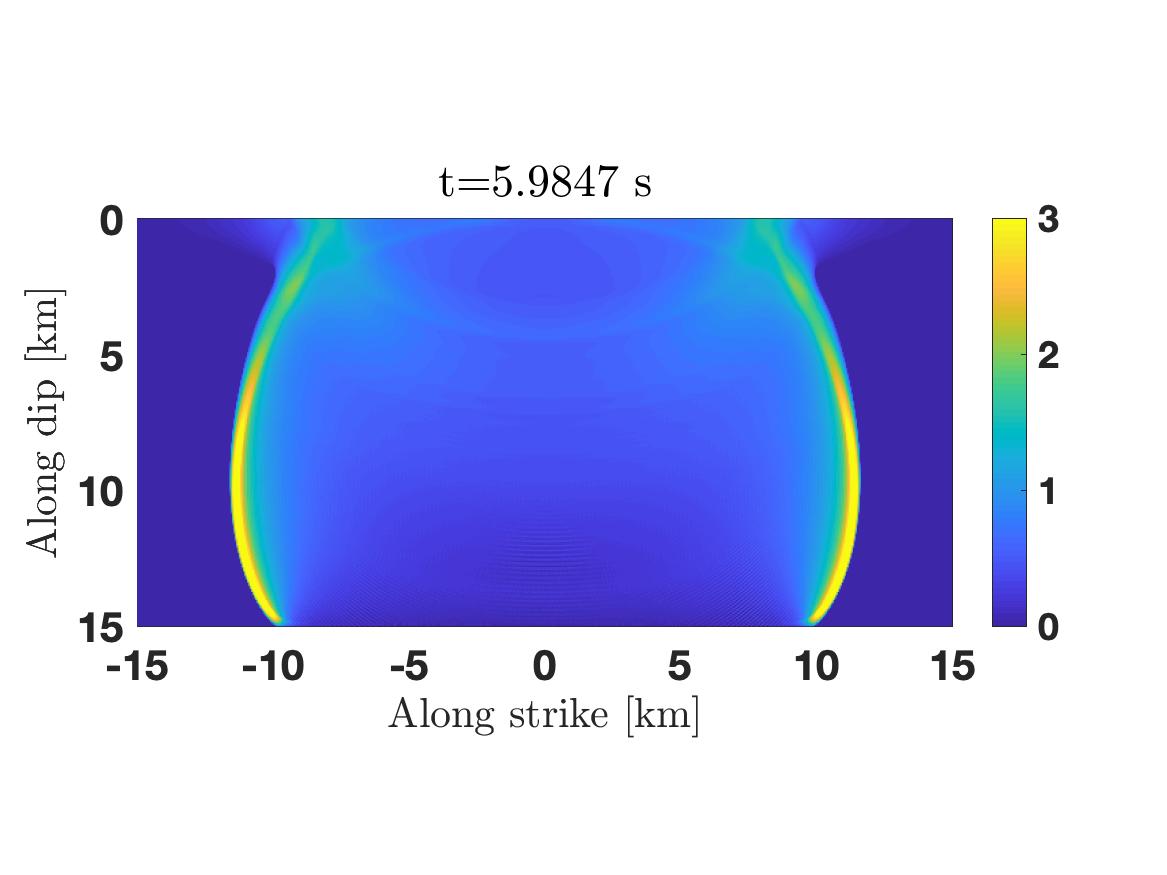}}{$6$\text{th order}}%
\hspace{-0.725cm}%
\stackunder[5pt]{\includegraphics[width=0.275\textwidth]{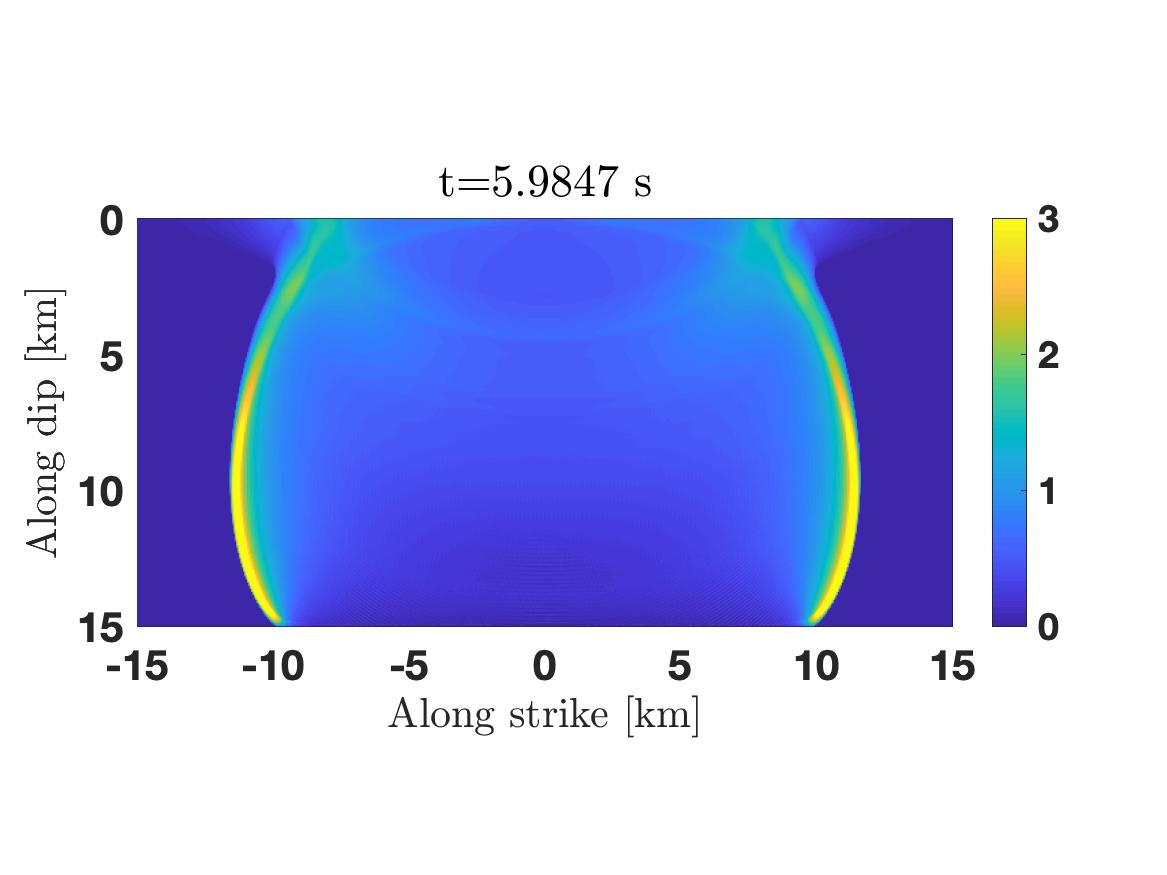}}{$7$\text{th order}}%
\renewcommand{\thefigure}{8b}
\caption{DRP SBP operators.}
}
\addtocounter{figure}{-2}
 \caption{Snapshots of the simulated slip-rate on  a dipping fault plane   at $t = 5.9847$ s using DP and DRP SBP operators of order $4, 5, 6, 7$ at h = 50 m grid spacing.  Note that odd order (5, 7) accurate DP SBP operators generate high frequency spurious oscillations which pollute the solutions everywhere.  Meanwhile all DRP operators do not support wild high frequency spurious oscillations.} 
    \label{fig:snap_shots_sliprate_comparison}   
\end{figure}

\subsection{Computational costs and scaling tests}
 Efficient HPC codes and parallel numerical implementations that scales perfectly with increasing supercomputing resources are necessary for effective numerical simulations. We have performed some scaling tests to demonstrate efficient parallel implementation of the DRP SBP operators for the numerical simulations of 3D nonlinear earthquake rupture dynamics, see Figure \ref{fig:scalingdrp}.
~\
Our new DRP operators are implemented in the large scale code with near perfect scaling.  The resources used is slightly larger than the $6$th order upwind DP and traditional  SBP schemes, but not by a large amount.  We consider the DRP operators to have a marginal higher cost than the traditional and DP SBP methods. 
    \begin{figure}[h!]
    \centering
    \makebox{\includegraphics[width = 0.95\textwidth]{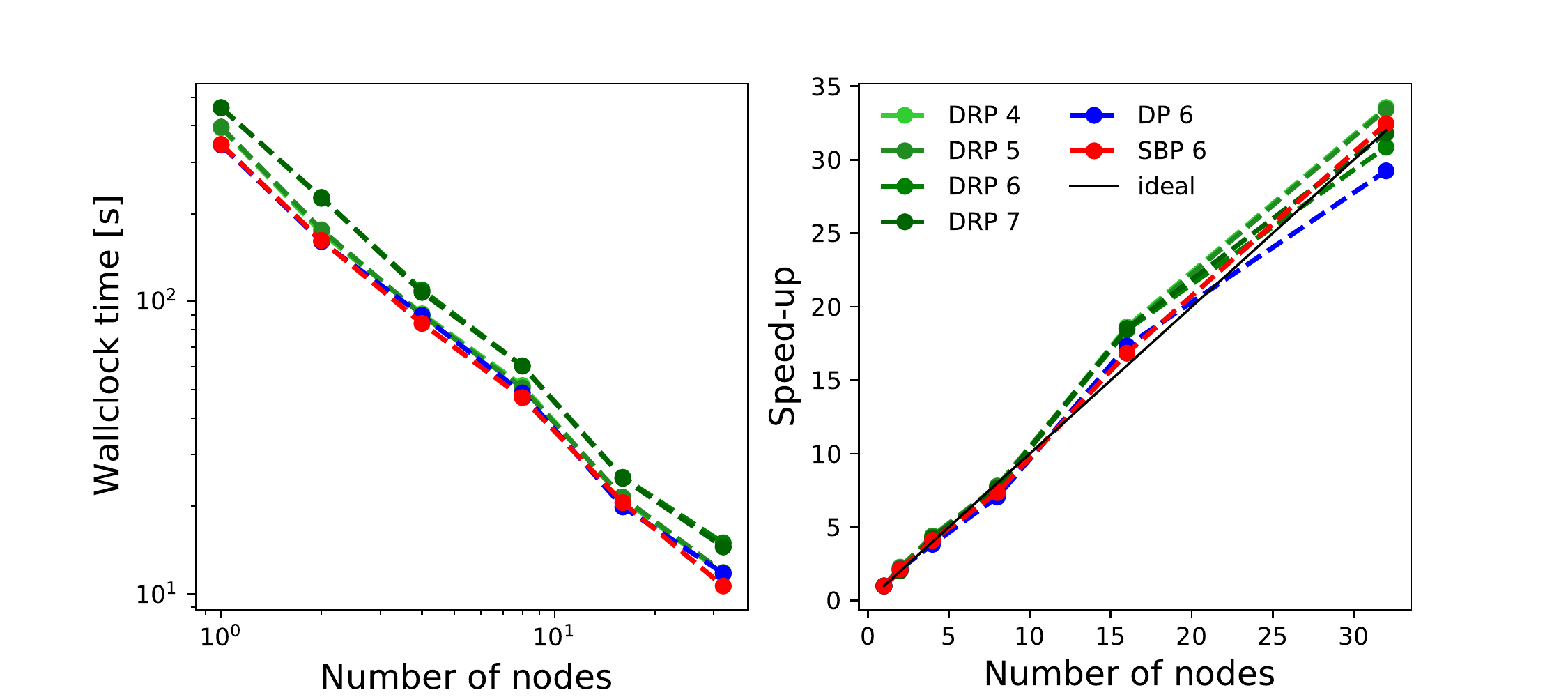}}
    \caption{Scaling plot for the dispersion relation preserving schemes in WaveQlab3D at $h = 100$ m grid spacing. The scaling test was performed for one second of the test problem TPV10.}
    \label{fig:scalingdrp}
\end{figure}

\subsection{Discussions}
Numerical simulations presented in Figure \ref{fig:slip_rate_100m_50m}--\ref{fig:shear_stress_100m_50m} compare our new DRP SBP operators with DP and the traditional SBP operator for the dynamic rupture simulation, TPV10 benchmark problem. The simulations are performed at two levels of mesh resolutions, $h = 100, 50$ m. As we can see, our DRP SBP operators have much fewer oscillations in their numerical solutions than their DP and traditional SBP counterparts. However,  the DRP operators seem to have a computational cost marginally higher than the computational costs of traditional and DP SBP methods, see Figure \ref{fig:scalingdrp}. 

Figure \ref{fig:slip_rate_100m_50m} and \ref{fig:shear_stress_100m_50m} show that, at $h = 100$ m grid spacing,  even (4 and 6) order DP SBP operators support minimal high frequency errors while the odd  (5 and 7) order DP SBP operators amplify high frequency numerical errors. These are in a good agreement with the numerical error analysis presented in Section \ref{sec:parity}. With mesh refinement $h = 50$ m, the 4th order upwind DP SBP operator, the traditional SBP FD operator, and all DRP SBP operators compute converging numerical solutions. For odd  (5 and 7) order DP SBP operators, mesh refinement   amplifies further high frequency numerical artefacts while the 6th order DP SBP operator generates a low frequency error. 

The accuracy of the numerical solutions of  DRP SBP operators at $h =100$ m resolution is comparable to the accuracy of the numerical solutions of  traditional SBP FD operator at $h =50$ m. Although mesh refinement can improve the accuracy of the traditional SBP FD operator, however, this  comes with a significant computational cost for the 3D problem.

Figure \ref{fig:snap_shots_sliprate_comparison} shows that the catastrophic high frequency numerical errors, for odd  (5 and 7) order DP SBP operators,  corrupts the solutions everywhere and significantly alters the nonlinear rupture process.

\section{Summary}
We derive and analyse high order finite difference numerical methods for the simulations of nonlinear frictional sliding along internal boundaries embedded in elastic solids. In particular,  we consider  earthquake source processes  modelled by spontaneously propagating shear ruptures along fault surfaces in 3D elastic solids.  This problem is both numerically and computationally challenging because of the unprecedented large gradients generated by nonlinear coupling of fields across the fault interface and the presence of discontinuous solutions on the fault. 

We discretise the 3D elastic wave equation with a pair of DP SBP operators, on boundary-conforming curvilinear meshes, and enforce the nonlinear interface conditions weakly using penalties.
 Using the energy method and the dual-pairing SBP framework \cite{Mattsson2017,DURU2022110966,CWilliams2021,williams2021provably}, we prove that the numerical method is energy-stable. We derive a priori error estimates and prove the convergence of the numerical error for a linearised friction law.
Further, we analyse the dependence of the parity (even/odd) of the operators on the numerical errors. For odd order accurate operators, our error analysis reveals that the DP SBP FD operators \cite{Mattsson2017} have the potential to generate spurious catastrophic high frequency errors that do diminish with mesh refinement. However, the $\alpha$-DRP SBP operators  \cite{williams2021provably,CWilliams2021} can resolve the highest frequency ($\pi$-mode) present on any equidistant grid at a tolerance of $\alpha = 5\%$ maximum error. The DRP  SBP operators \cite{williams2021provably,CWilliams2021}  potentially eliminates the catastrophic high frequency errors. 

Numerical simulations  performed in 3D corroborates  the analysis.
Our DRP  SBP operators reproduces dynamic rupture benchmark problems, proposed by Southern California Earthquake Center, with less computational effort than standard methods based on traditional SBP  operators.
    
\section*{Acknowledgments}
This research was undertaken with the assistance of resources and services from the National Computational Infrastructure (NCI), which is supported by the Australian Government's National Collaborative Research Infrastructure Strategy (NCRIS).  Frederick Fung and Christopher Williams acknowledge support from the Australian Government Research Training Program Scholarship.

\bibliographystyle{plain}
\bibliography{bibfile}{}  

\end{document}